\documentclass[11pt,letterpaper]{amsart}
\usepackage{amsmath}
\usepackage{yhmath}
\usepackage{ booktabs}
\usepackage{amsthm}
\usepackage{amssymb}
\usepackage{amscd}
\usepackage{amsfonts}
\usepackage{amsbsy}
\usepackage{graphicx,epstopdf,subfigure}
\usepackage[dvips]{psfrag}
\usepackage{url}
\usepackage{epsfig}
\usepackage{latexsym}
\usepackage[mathscr]{eucal}
\usepackage{graphics}
\usepackage{multirow}
\usepackage{color}
\usepackage{calc}
\usepackage{xcolor}

\allowdisplaybreaks \allowdisplaybreaks[4]

\newcommand{\R}{\ensuremath{\mathbb{R}}}

\newcommand{\al}{\alpha}
\newcommand{\be}{\beta}

\newtheorem {theorem} {Theorem} 
\newtheorem {proposition} [theorem] {Proposition}

\newtheorem {lemma} [theorem] {Lemma}

\newtheorem {remark} {Remark}
\newtheorem {example} {Example}

\begin{document}
	
	\title[Nonexistence of periodic orbits of  planar dynamical systems]
	{New criterions  on nonexistence of periodic orbits of  planar dynamical systems and their applications} 
	
	\author[H. Chen et al. ]
	{
 	Hebai Chen$^{1}$, Hao Yang$^{1}$,  Rui Zhang$^{1}$, Xiang Zhang$^{2}$
	 }
	
	\address{$^1$
		School of Mathematics and Statistics, HNP-LAMA, Central South University,
		Changsha, Hunan 410083, P. R. China
	}

  \address{$^2$ School of Mathematical Sciences,  MOE-LSC, Shanghai
 Jiao Tong University, Shanghai, 200240, P.R. China}
	\email{chen\_hebai@csu.edu.cn (H. Chen),
		y\_ang\_hao@163.com (H. Yang), zhang\_rui@csu.edu.cn (R. Zhang),
		 xzhang@sjtu.edu.cn (X. Zhang)
	}
	\subjclass[2010]{34C07, 34A05, 34A34, 34C14}
	
	\keywords{Limit cycle, non--existence, nilpotent equilibrium, Andreev topological classification,  global phase portrait.}

	\begin{abstract}
Characterizing existence or not of periodic orbit is a classical problem and it has both theoretical importance and many real applications. Here, several new criterions on nonexistence of periodic orbits of the planar dynamical system
	  $\dot x=y,~\dot y=-g(x)-f(x,y)y$ are obtained in this paper, and by examples showing that these criterions are applicable, but the known ones are invalid to them.
Based on these criterions, we further characterize the local topological structures of its equilibrium, which also show that one of the classical results by A.F. Andreev [Amer. Math. Soc. Transl. 8 (1958), 183--207] on local topological classification of the degenerate equilibrium is incomplete.
 Finally, as another application of these results, we classify the global phase portraits of a planar differential system, which comes from the third question in the list of the 33 questions posed by A. Gasull and {also from} a mechanical oscillator under suitable restriction to its parameters.
	\end{abstract}
	
	\maketitle

	\section{ Introduction }

As we all know,
 Hilbert's 23 problems
 were posed
  by the famous mathematian D. Hilbert at the International Congress of Mathematicians in 1900, see \cite{Hilbert},
  where the second
	half of Hilbert's 16th problem is to study the maximum number and their relative position of limit cycles of planar polynomial differential systems.
	Up to now, Hilbert's 16th problem is still unsolved.
	On the other hand,
	to study    dynamics of a planar dynamical system, it is usually very important to characterize existence of its limit cycles.
	
	For the aforementioned reasons, {the study of limit cycles of planar dynamical systems has been attracting many famous mathematicians working on it},
	see for instance Dulac \cite{Dulac},
	 Itenberg and Shustin \cite{IS},
	Lins et al \cite{LMP}, Roussaries \cite{Roussarie}, Smale \cite{Smale},  Ye \cite{Ye}and Zhang et al \cite{ZDHD}, and the references therein. Notice that
	most of the known results on nonlocal limit cycles are for existence and uniqueness.
In order to prove existence of limit cycles, one of the essential tools is the  Poincar\'e-Bendixson annulus  theorem.
  	On uniqueness of the limit cycle, most of the results were limited to  Li\'enard systems and generalized Li\'enard systems,  such as Levinson et al \cite{LS},  Liou and Cheng \cite{LC},
  	Wang and Kooij \cite{WK},
  	 Xiao and Zhang \cite{XZ},  Zeng \cite{Zeng}, Zeng et al \cite{ZZG}  and   Zhang et al \cite{ZDHD}, and so on.
	However, for  nonexistence of the limit cycle of planar differential equations, the theoretical results are far less than those for existence of limit cycles.
At present, Poincar\'e's method of tangential curves \cite[Theorem 1.6 of Chapter 4]{ZDHD} and the Dulac criterion \cite[Theorem 1.7 of Chapter 4]{ZDHD} are common tools to study nonexistence of limit cycles, but it is not an easy task to find the Poincar\'e function $F(x,y)$ or  Dulac function $B(x,y)$ in applications. Besides, most of the   results on nonexistence of limit cycles are focused on Li\'enard and generalized Li\'enard systems, which can be found in \cite{CC,CT,DR,Sugie} and their references.
	See Appendix, where we list some known results for comparing with ours and their applications.
Besides, there are also a few theoretical results on nonexistence of limit cycles for general planar differential systems.
	
The goal of this paper is to provide certain new criterions on nonexistence of limit cycles of the  planar differential system
	\begin{equation}
	\label{1}
	\left\{\begin{aligned}
	\dot{x}&=y,\\
	\dot{y}&=-g(x)-f(x,y)y,
	\end{aligned}
	\right.
	\end{equation}
	where $x\in(\alpha,\beta)$, $y\in\R$, $\alpha<0$, $\beta>0$.
	Notice that system \eqref{1} has been widely adopted to model real world problems in applied science and engineering,
	see \cite{Ame1982,BBN,NB} and the references therein.
	
	The organization  of this paper is as follows. In section \ref{mainresults}, we state our main results, which are new criterions on  nonexistence of periodic orbits of system \eqref{1}, and characterization on local topological structures of the related system at its equilibrium. Here we complete the local classification of a degenerate equilibrium {(nilpotent one)}, which is a {classical} result but incomplete as will be shown. It was initially proved by Andreev \cite{Andreev} in 1958, and then stated and proved in \cite{ZDHD} {and so on}. As we have seen, Andreev's result was also repeatedly stated in many monographs and papers for classifying topological structures of planar differential systems.
	 Section	\ref{pmr}  is the
	proofs of our main results. Section \ref{ap} is partly an application of our theoretical results, where we characterize all global topological phase portraits (Theorems \ref{theo} and \ref{thm8}) of a system under certain restriction of parameters, which comes from the first half of the third question in the list of the 33 questions posed by Gasull \cite{Gasull} and also from a mechanical oscillator.
	
	\section{Main results}
	\label{mainresults}
	In this section, we state our main results of this paper.
	The first {one provides} a criterion on nonexistence of periodic orbits.
	
	\begin{theorem}
		\label{thm1}
		Assume that $g(x)=-g(-x)$ for all $0\leqslant x<\min\{-\alpha,\beta\}$, and that the following conditions hold:
		\begin{itemize}
			\item[\bf (\romannumeral1)]  $xg(x)>0$ for all $(\alpha,0)\cup(0,\beta)$;
			\item[\bf (\romannumeral2)]{ $g(x)$ is Lipschitzian continuous for $x\in(\alpha,0)\cup(0,\beta)$, and $f(x,y)$ is Lipschitzian continuous for $(x,y)\in(\alpha,\beta)\times \R$;}
			\item[\bf (\romannumeral3)] either $f(x,y)\geqslant-f(-x,y)$ or   $f(x,y)\leqslant-f(-x,y)$  for all $0\leqslant x<\min\{-\alpha,\beta\}$ and $y\in\mathbb{R}$;
			\item[\bf (\romannumeral4)] $f(x,y)\not\equiv-f(-x,y)$ for $x\in(0,\zeta)$ and $y\in\mathbb R$, where $0<\zeta\ll1$.
		\end{itemize}	
		Then, system \eqref{1} has no closed orbits in the  strip $\alpha<x<\beta$.
	\end{theorem}

	{
	The second one {is} another criterion on nonexistence of periodic orbits.
	\begin{theorem}
	\label{thm1b}
	Assume that the conditions {\bf (i)},
	{\bf (ii)}, {\bf (iv)} and the following one hold:
	\begin{itemize}
		\item[\bf (\romannumeral3$'$)] either $f(x,y)\geqslant-f(x,-y)$ or   $f(x,y)\leqslant-f(x,-y)$  for all $(x,y)\in(\alpha,\beta)\times \mathbb R$.
	\end{itemize}	
	Then, system \eqref{1} has no closed orbits in the  strip $\alpha<x<\beta$.
\end{theorem}
}

	\begin{remark}
	{   When $g(x)=x$ and $f(x,y)=\hat f(y)$, after the change of variables $(x,y,t)\to(y,x,-t)$,
	system \eqref{1} is transformed to
			\begin{equation}
	\notag
		\left\{\begin{aligned}
		\dot{x}&=y+\hat f(x)x,\\
		\dot{y}&=-x,
		\end{aligned}
		\right.
		\end{equation}
which is in the Li\'enard form. By  {\rm Theorem \ref{thm1b}}
one can directly obtain the results in \cite[Proposition 1]{LMP} $($see {\rm Theorem \ref{LMP}}  in {\rm  Appendix A}$)$.
In this sense,}
{\rm Theorem \ref{thm1b}} is an extended version of \cite[Proposition 1]{LMP}.
\end{remark}

The next is the third criterion on nonexistence of periodic orbits.
	
	\begin{theorem}
		\label{thm2}
		Assume that $g(-x)\not\equiv-g(x)$ for $0\leqslant x<\min\{-\alpha,\beta\}$,
		 	 \begin{equation*}
			g(x)=
x^mh(x)	,	\ {\rm if} \ x\in(-\epsilon, \epsilon),
		\end{equation*}
	 	where  $h(0)>0$,
  $m=p/q\geqslant 1$, $p,q$ are odd and $\epsilon>0$ is small.
	 In addition to the conditions {\bf{(\romannumeral1)}} and  {\bf{(\romannumeral2)}} of {\rm Theorem \ref{thm1}}, suppose that for all $\hat{x}<0<x$ satisfying
	\begin{equation}
	\label{c}
		\int_{\text{0}}^{x}{g(s)}ds=\int_{\text{0}}^{\hat{x}}{g(s)}ds,
	\end{equation}
the following hold
		\begin{itemize}
			\item[\bf(\romannumeral5)]
		either	$f(x,y)/g(x)\geqslant f(\hat{x},y)/g(\hat{x})$ or $f(x,y)/g(x)\leqslant f(\hat{x},y)/g(\hat{x})$  for $0\leqslant x<\min\{-\alpha,\beta\}$ and $y\in\mathbb R$;
			\item[\bf(\romannumeral6)]$f(x,y)/g(x)\not\equiv f(\hat{x},y)/g(\hat{x})$ for $x\in (0,\zeta)$  and $y\in\mathbb R$, where $0<\zeta\ll1$.
		\end{itemize}
		Then, system \eqref{1} has no closed orbits in the  strip $\alpha<x<\beta$.
	\end{theorem}
	
\begin{remark}
	In {\rm Theorem \ref{thm2}},  if $g(x)\equiv-g(-x)$ then $
		\int_{\text{0}}^{x}{g(s)}ds=\int_{\text{0}}^{\hat{x}}{g(s)}ds
		$ holds with $\hat{x}=-x$. Consequently, $f(x,y)/g(x)\geqslant f(\hat{x},y)/g(\hat{x})$ in {\bf(\romannumeral5)}  implies $f(x,y)\geqslant-f(-x,y)$, which is {\bf(\romannumeral3)}. In this sense, {\rm Theorem \ref{thm1}} is a special case of   {\rm Theorem \ref{thm2}}.
Here, stating {\rm Theorem \ref{thm1}} separately has two reasons: one is for its easy application, and second is for distributing the technical parts of the proofs. As it is easy to see, system \eqref{1} is an extension of the  Li\'enard system
		$$
		\dot{x}=y,\,\,\dot{y}=-g(x)-f(x)y,
		$$
in which $g(x)$ is usually an odd function, like $g(x)=x$, or $g(x)=x+x^3$, etc, and so {\rm Theorem \ref{thm1}} can be conveniently applied to it.
      	\end{remark}
	
The fouth one is an extension of Theorem \ref{thm1}, which admits existence of other equilibria than the origin.

	\begin{theorem}\label{thm3}
		In case $g(x)=-g(-x)$ for all $0\leqslant x<\min\{-\alpha,\beta\}$, suppose
the conditions {\bf(\romannumeral2)},  {\bf(\romannumeral3)},  {\bf(\romannumeral4)} of {\rm Theorem \ref{thm1}}
hold.
Then,  system \eqref{1} has no closed orbits   surrounding the origin $O$ in the  strip $\alpha<x<\beta$.
	\end{theorem}

	In Theorems \ref{thm1}, {\ref{thm1b}} and \ref{thm2},
 system \eqref{1} has a unique equilibrium $O$  and has no closed orbits when the corresponding conditions hold.
 Naturally, we want to further characterize the qualitative properties of $O$ and the phase portrait of system \eqref{1}.
 To avoid much degeneracy, suppose that both $g(x)$ and $f(x,y)$ are  analytic functions  in a small neighborhood $S_\delta(O)$ of the origin $O$, except a few exception.

   If the condition {\bf{(\romannumeral1)}}  of Theorem \ref{thm1}  holds, then
$g'(0)\geqslant0$.  Otherwise,  if $g'(0)<0$, there exists an $x_*>0$ such that $g(x_*)<g(0)=0$,  which contradicts  $xg(x)>0$ for $x\neq0$.
Set $a:=g'(0)$, $b:=f(0,0)$. It is easy to check $b\geqslant0$ for $f(x,y)/g(x)\geqslant f(\hat{x},y)/g(\hat{x})$, $b\leqslant0$ for $f(x,y)/g(x)\leqslant f(\hat{x},y)/g(\hat{x})$, and $a\geqslant0$,
where $\hat x\leqslant0\leqslant x$ and $\int_{\text{0}}^{x}{g(s)}ds=\int_{\text{0}}^{\hat{x}}{g(s)}ds$.

In our next result, we characterize local phase portraits of system \eqref{1} with $f(x,y)/g(x)\geqslant f(\hat{x},y)/g(\hat{x})$ and
$\int_{\text{0}}^{x}{g(s)}ds=\int_{\text{0}}^{\hat{x}}{g(s)}ds$ for $\hat x\leqslant0\leqslant x$.
The case $f(x,y)/g(x)\leqslant f(\hat{x},y)/g(\hat{x})$ and
$\int_{\text{0}}^{x}{g(s)}ds=\int_{\text{0}}^{\hat{x}}{g(s)}ds$ for $\hat x\leqslant0\leqslant x$ can be treated via the transformation $(x, \hat x, y,t)\to(\hat x, x, y,-t)$.
So without loss of generality, we consider only the case $b\geqslant0$.

For our consideration,  the region
\[
\mathscr{G}:=\{(a,b)\in\mathbb{R}^2: a\geqslant 0, b\geqslant 0\}
\]
will be separated in the six subregions
\begin{equation}\notag	\begin{aligned}
\mathscr{G}_1&:=\{(a,b)\in\mathscr{G}: a>0, b>0, b^2-4a>0 \},\\
\mathscr{G}_2&:=\{(a,b)\in\mathscr{G}: a>0, b>0, b^2-4a<0 \},\\
\mathscr{G}_3&:=\{(a,b)\in\mathscr{G}: a>0, b>0, b^2-4a=0\},\\
\mathscr{G}_4&:=\{(a,b)\in\mathscr{G}: a>0, b=0\},\\
\mathscr{G}_5&:=\{(a,b)\in\mathscr{G}:  a=0, b>0\},\\
\mathscr{G}_6&:=\{(a,b)\in\mathscr{G}:  a=b=0\}.
\end{aligned}		
\end{equation}
When $a=b=0$, in   $S_\delta (O)$,  $g(x)$ and $f(x,y)$ can be written in the forms
\begin{equation}\label{eab}
	g(x)=a_kx^k+O(x^{k+1}),\quad f(x,y)=b_nx^n+O(x^{n+1})+yp(x,y),
\end{equation}	
	where $a_k\ne0$, $k>1$, $b_n\in \mathbb{R}$,  and $p(x,y)$ is an analytic function {in  $S_\delta (O)$}. By the condition {\bf{(\romannumeral1)}} of Theorem \ref{thm1}, we have
\begin{equation}\notag
\begin{aligned}
x(a_kx^k+O(x^{k+1}))>0\ \ \ \ \mbox{\rm i.e.} \ \ \ \
a_kx^{k+1}+O(x^{k+2})>0,\,\,\,\,\,x\in S_\delta (O) ,
\end{aligned}
\end{equation}
which implies that
\begin{equation}\label{eab1}
 a_k>0 \quad \mbox{\rm and } \quad k \ \ \mbox{\rm is odd}.
 \end{equation}
Without loss of generality, we only consider $b_n\geqslant0$. The case $b_n<0$ can be done via the transformation $(x,y,t, b_n)\to(x,-y,-t,-b_n)$. Next we further divide $\mathscr{G}_6$ into the three subregions
	\begin{equation}\notag	\begin{aligned}
\mathscr{G}_{61}:=&\{(a,b)\in\mathscr{G}_6: b_n=0\}\cup\{(a,b)\in\mathscr{G}_6: b_n>0, n>(k-1)/2\}\\
&\cup\{(a,b)\in\mathscr{G}_6: b_n>0, n=(k-1)/2, b_n^2-2(k+1)a_k<0\},\\
\mathscr{G}_{62}:=&\{(a,b)\in\mathscr{G}_6: b_n>0, n<(k-1)/2, n \text{ even}\}\\
&\cup\{(a,b)\in\mathscr{G}_6: b_n>0, n=(k-1)/2, b_n^2-2(k+1)a_k\geqslant0, n \text{ even} \},\\
\mathscr{G}_{63}:=&\{(a,b)\in\mathscr{G}_6: b_n>0, n<(k-1)/2, n \text{ odd}\}\\
&\cup\{(a,b)\in\mathscr{G}_6: b_n>0, n=(k-1)/2,  b_n^2-2(k+1)a_k\geqslant0, n~ \text{odd} \}.\\
\end{aligned}		
\end{equation}	

Having the above preparation, we can state our next {results}
on  {global} structure of system \eqref{1}.

	{
	\begin{theorem}
		\label{lem2}	For system \eqref{1},
		 suppose that
		 \begin{itemize}
		 		\item
	  $\alpha=-\infty$ and $\beta=+\infty$,   $g(x)$ is analytic in $|x|<\epsilon$ $(\epsilon>0$ is small $)$ and $f(x,y)$ is  analytic in $S_\delta(O)$, 		
			\item $xg(x)>0$ for all $x\neq0$, $g(x)$ and $f(x,y)$ are Lipschitzian continuous for $(x,y)\in\R^2$,
			\item either $f(x,y)\equiv-f(-x,y)$ for $g(x)=-g(-x)$ in $(x,y)\in(0,+\infty)\times \R$, or  $f(x,y)/g(x)\equiv f(\hat{x},y)/g(\hat{x})$ for  $g(x)\not\equiv-g(-x)$ in $(x,y)\in(0,+\infty)\times \R$, where $\hat{x}<0<x$ satisfying  equation \eqref{c}.
		\end{itemize}
{The following statements hold.
 \begin{itemize}
 \item The origin $O$ of   system \eqref{1} is a center, or $S_\delta(O)$  consists of
		one elliptic sector and one hyperbolic sector, or  $S_\delta(O)$ consists of
		one elliptic sector, one hyperbolic sector and two parabolic sectors.  {\rm Figure \ref{tu2}} illustrates these local structures at $O$.
		\begin{figure}[hpt]
			\centering
			\subfigure[{ $(a,b)\in\mathscr{G}_{63}$}]{
				\includegraphics[width=0.3\textwidth]{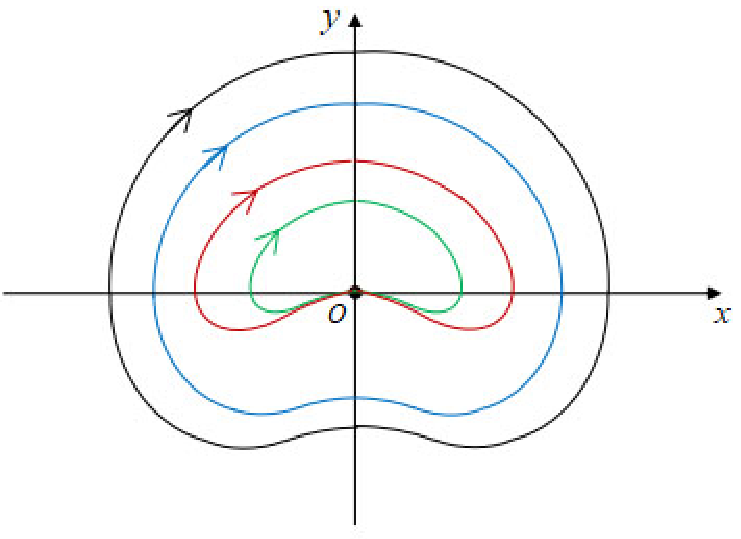}}
			\quad
			\subfigure[{ $(a,b)\in\mathscr{G}_{63}$}]{
				\includegraphics[width=0.3\textwidth]{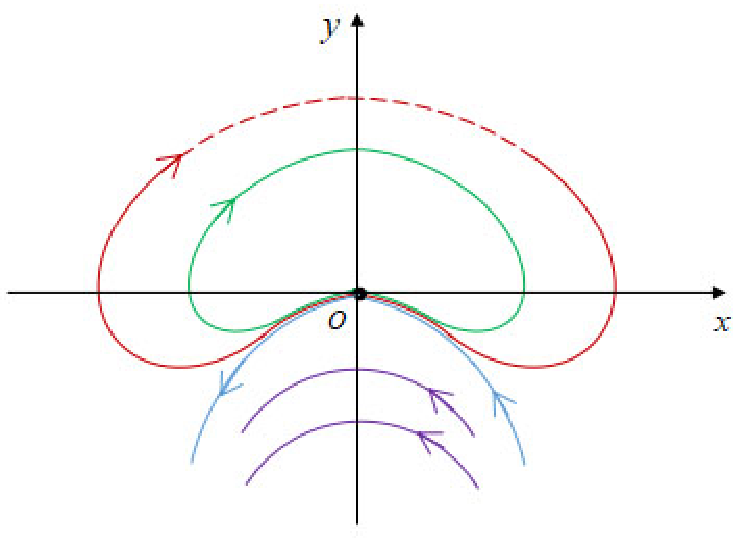}}
			\quad
			
			\subfigure[{ $(a,b)\in\mathscr{G}_{63}$}]{
				\includegraphics[width=0.3\textwidth]{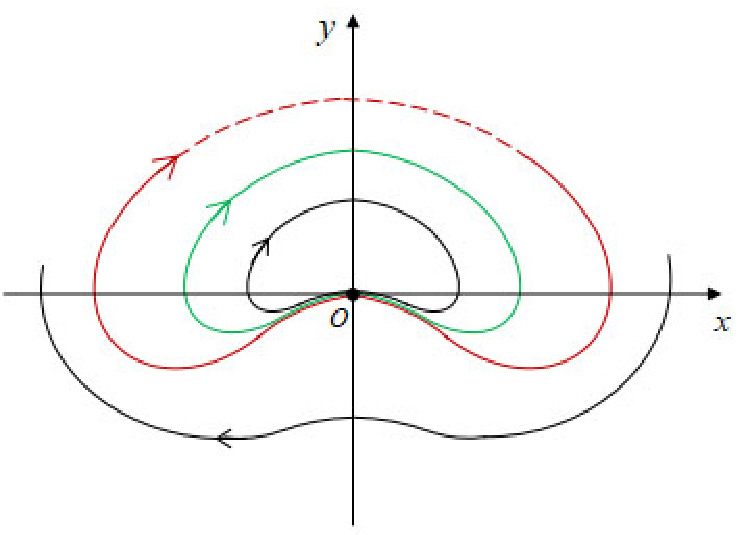}}
			\quad
			\subfigure[{$(a,b)\in \mathscr{G}_4\cup\mathscr{G}_{61}$ }]{
				\includegraphics[width=0.3\textwidth]{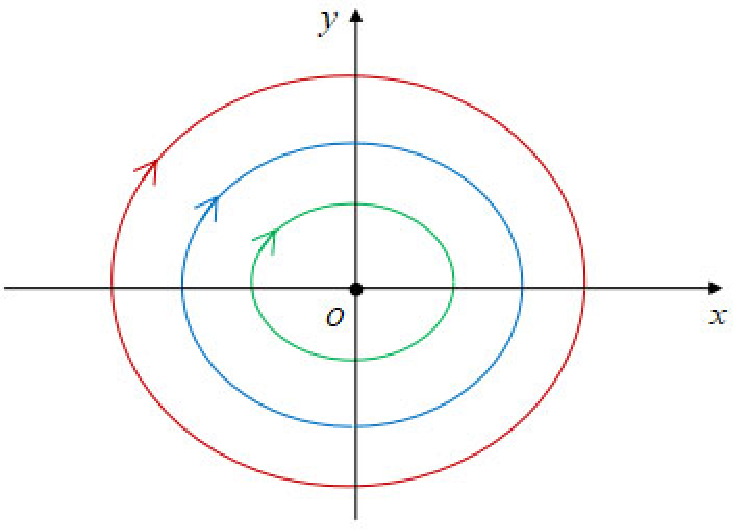}}
			\caption{Phase portraits of system \eqref{1} when $f(x,y)/g(x)\equiv f(\hat{x},y)/g(\hat{x})$. (The elliptic sector is bounded in (a), and unbounded in (b) and (c).) 
			} \label{tu2}		
		\end{figure}
\item When the elliptic sector is bounded, there has no parabolic sectors at $O$. In other words, two parabolic sectors appearing in a neighborhood of $O$ happens only in the unbounded case of the elliptic sector. These two situations can be realized by concrete examples.
\end{itemize}
}
\end{theorem}
}

\begin{theorem}
	\label{thm4}
	For system \eqref{1}, suppose that
	 \begin{itemize}
	 			 		\item
	 	$\alpha=-\infty$ and $\beta=+\infty$,   $g(x)$ is analytic in $|x|<\epsilon$ $(\epsilon>0$ is small $)$, and $f(x,y)$ is  analytic in $S_\delta(O)$,
		\item $xg(x)>0$ for all $x\neq0$, $g(x)$ and $f(x,y)$ are Lipschitzian continuous for $(x,y)\in\R^2$,
		\item either $f(x,y)\geqslant-f(-x,y)$ for $g(x)=-g(-x)$ in $(x,y)\in(0,+\infty)\times \R$, or  $f(x,y)/g(x)\geqslant f(\hat{x},y)/g(\hat{x})$ for  $g(x)\not\equiv-g(-x)$ in $(x,y)\in(0,+\infty)\times \R$, where $\hat{x}<0<x$ satisfying  equation \eqref{c},
		\item $f(x,y)\not\equiv-f(-x,y)$ as $g(x)=-g(-x)$ for $(x,y)\in(0,\zeta)\times\R$ and $f(x,y)/g(x)\not\equiv f(\hat{x},y)/g(\hat{x})$ as $g(x)\not\equiv-g(-x)$ for $(x,y)\in(0,\zeta)\times\R$,  where $0<\zeta\ll1$.
	\end{itemize}
	Then, the qualitative property of the unique equilibrium $O$ is as that shown in {\rm Table
		\ref{QPO}} and system \eqref{1} could have { only  possibly seven} local  phase portraits, as those shown in {\rm Figures \ref{tu2}(b) and (c)} and {\rm Figure \ref{1pps2}}.
		
	{
	\begin{table}[htp]
	\renewcommand\arraystretch{2.2}
	\setlength{\tabcolsep}{1mm}{
		\caption{	\label{QPO} The qualitative property of $O$ of system \eqref{1}.}
		
		\begin{tabular}{c|c|c|c}
			\hline
			$(a,b)$ & Type of $O$ & \multicolumn{2}{c}{Geometric configurations}
			\\
			
			\hline		
			$ \mathscr{G}_{1}\cup \mathscr{G}_{5}$ & stable node & \multicolumn{2}{c}
			{\begin{tabular}[c]{@{}c@{}}
					in $S_\delta(O)$ except one pair of orbits  approaching   \vspace{-0.5cm}
					\\
				$O$	in one of the directions $\theta_2$ and $\theta_4$, \vspace{-0.5cm}
					\\
					all other orbits approaching  $O$ in one of \vspace{-0.5cm}
					\\ the directions $\theta_1$ and $\theta_3$, see  {\rm Figure \ref{1pps2}(a)}.
			\end{tabular}}
			\\
			
			\hline		
			$\mathscr{G}_{2}\cup \mathscr{G}_{4}\cup \mathscr{G}_{61}$ & stable focus & \multicolumn{2}{c}
			{in $S_\delta(O)$  all  orbits rotate clockwise, see  {\rm Figure \ref{1pps2}(b)}.}
			\\
			
			\hline		
			$\mathscr{G}_{3}\cup \mathscr{G}_{62}$ &
			\begin{tabular}[c]{@{}c@{}}
				stable \vspace{-0.5cm}
				\\
				improper node
			\end{tabular}
			& \multicolumn{2}{c}
			{\begin{tabular}[c]{@{}c@{}}
					in $S_\delta(O)$  all orbits approaching  $O$ \vspace{-0.5cm}
					\\ 					
					along one of the directions $\theta_5$ and $\theta_6$, see  {\rm Figure \ref{1pps2}(c)}.
			\end{tabular}}
			\\
			
			\hline		
			\multirow{4}{*}
			{\begin{tabular}[c]{@{}c@{}}\\ \\ \\ $\mathscr{G}_{63}$\end{tabular}} &
			\multirow{4}{*}
			{\begin{tabular}[c]{@{}c@{}}\\ \\ \\ a degenerate \vspace{-0.5cm}
					\\
					one\end{tabular}} &
			\begin{tabular}[c]{@{}c@{}}
				elliptic sector \vspace{-0.5cm}
				\\
				is bounded \end{tabular} &
			\begin{tabular}[c]{@{}c@{}}
				$S_\delta(O)$ consist  of one elliptic sector, \vspace{-0.5cm}
				\\
				one hyperbolic sector and   \vspace{-0.5cm}
				\\
				one parabolic sector, see  {\rm Figure \ref{1pps2}(d)}.
			\end{tabular}
			\\
			
			\cline{3-4}
			&  &
			\multirow{3}{*}
			{\begin{tabular}[c]{@{}c@{}}\\ \\
					elliptic sector \vspace{-0.5cm}
					\\
					is unbounded
			\end{tabular}} &
			\begin{tabular}[c]{@{}c@{}}
				$S_\delta(O)$ consist  of one elliptic sector,  \vspace{-0.5cm}
				\\
				one hyperbolic sector and
				\vspace{-0.5cm}
				\\
				zero  parabolic sector, see  {\rm Figure \ref{tu2}(c)}.
			\end{tabular}
			\\
			
			\cline{4-4}
			&  &  &
			\begin{tabular}[c]{@{}c@{}}
				$S_\delta(O)$ consist  of one elliptic sector,  \vspace{-0.5cm}
				\\
				one hyperbolic sector and \vspace{-0.5cm}
				\\
				one  parabolic sector, see  {\rm Figure \ref{1pps2}(e)}.
			\end{tabular}
			\\
			
			\cline{4-4}
			&  &  &
			\begin{tabular}[c]{@{}c@{}}
				$S_\delta(O)$ consist  of one elliptic sector,  \vspace{-0.5cm}
				\\
				one hyperbolic sector and \vspace{-0.5cm}
				\\
				two  parabolic sectors, see  {\rm Figure \ref{tu2}(b)}.
				
			\end{tabular}
			\\
			\hline
	\end{tabular}}
	\begin{center}
		Remark : $\theta_1 =\pi -\arctan \left(b-\sqrt{b^2-4a})/2\right),
		\theta_2=\pi -\arctan \left(b+\sqrt{b^2-4a})/2\right) ,
		\theta_3 =2\pi -\arctan \left(b-\sqrt{b^2-4a})/2\right),
		\theta_4=2\pi -\arctan \left(b+\sqrt{b^2-4a})/2\right), \theta_5=\pi-\arctan \left(b/2\right), \theta_6=2\pi-\arctan \left(b/2\right)$.
		
	\end{center}
\end{table}
}
			\begin{figure}[hpt]
				\centering
				\subfigure[{ $(a,b)\in\mathscr{G}_1\cup\mathscr{G}_5$ }]{
					\includegraphics[width=0.3\textwidth]{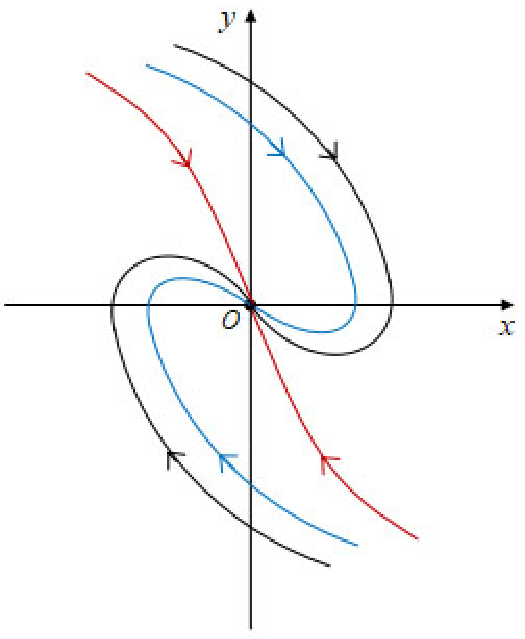}}
				\quad
				\subfigure[{ $(a,b)\in\mathscr{G}_2\cup\mathscr{G}_4\cup\mathscr{G}_{61}$ }]{
					\includegraphics[width=0.3\textwidth]{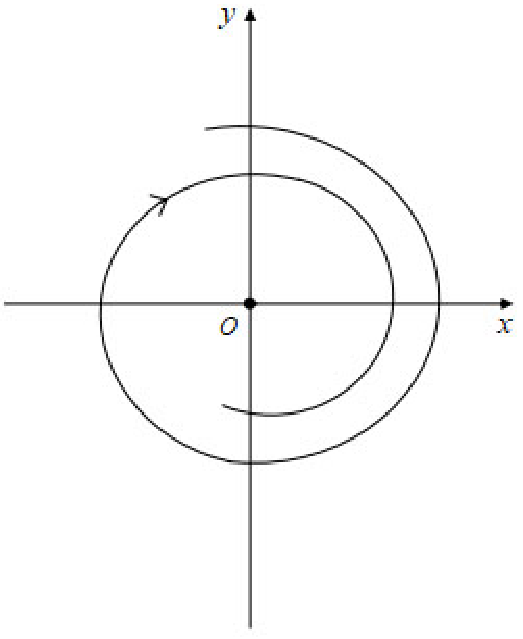}}
				\quad
				
				\subfigure[{$(a,b)\in\mathscr{G}_3\cup\mathscr{G}_{62}$ }]{
					\includegraphics[width=0.3\textwidth]{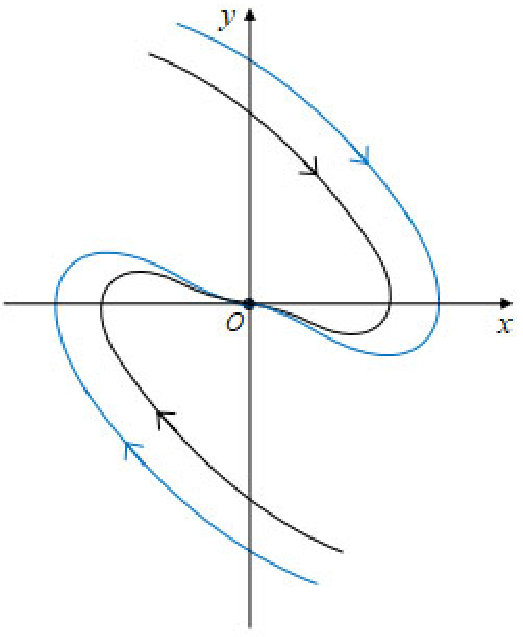}}
				\quad	
				\subfigure[{  $(a,b)\in\mathscr{G}_{63}$}]{
					\includegraphics[width=0.3\textwidth]{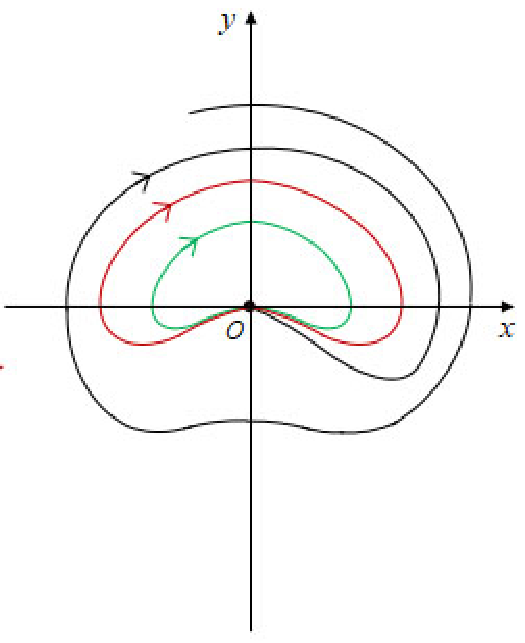}}
				\quad	
				\subfigure[{  $(a,b)\in\mathscr{G}_{63}$}]{
					\includegraphics[width=0.3\textwidth]{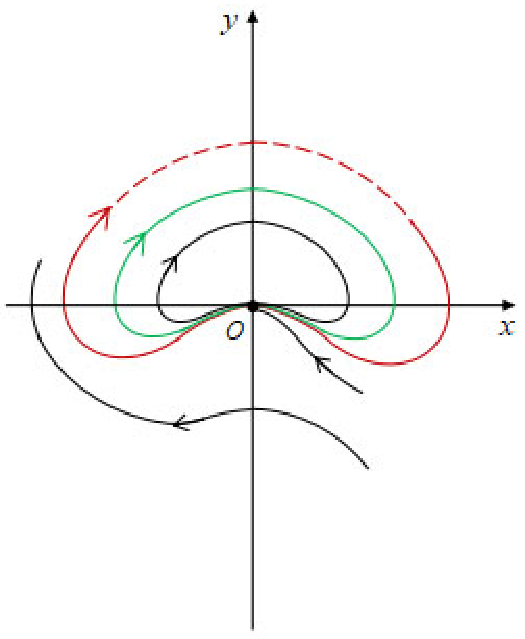}}
				\caption{Phase portraits of system \eqref{1} { when $f(x,y)/g(x)\geqslant f(\hat{x},y)/g(\hat{x})$ in $x>0$ and $y\in\R$. (The elliptic sector is bounded in (d), and unbounded in (e).)} }
				\label{1pps2}
			\end{figure}
	\end{theorem}

	\begin{remark}\label{rem1}
	For system \eqref{AA} in {\rm Appendix B}, Andreev \cite{Andreev} proved that when  $b_n\ne0$ and $n<(k-1)/2$, or $b_n\ne0$, $n=(k-1)/2$ and $b_n^2-2(k+1)a_k\geqslant0$ with $n$ and $k$  odd,
any small neighborhood  $S_\delta(O)$ of system \eqref{AA} at $O$ consists of
	only one elliptic sector and  one hyperbolic sector.
This result has also been collected in many Monographs and papers for their applications on local classification of equilibria, see for instance \cite[Theorem 7.2 of Chapter 2]{ZDHD}, 	\cite[Theorem 3.5 of Chapter 3]{DLA} and so on.
	However, as we show, this classification is incomplete. In fact, beside	the elliptic and hyperbolic sectors, there may include zero { $($see {\rm system \eqref{eg3}} with {\rm Figure~\ref{fEHsymmetry}}$)$, or one $($see {\rm Example \ref{eg2}} with {\rm Figure~\ref{fEHasymOneparabolic}}$)$, or two parabolic sectors $($see {\rm Example \ref{eg1}} with {\rm Figure~\ref{gpps34}}$)$} in $S_\delta(O)$.
	{		Andreev \cite{Andreev} proved the existence of the elliptic sector without discussing its boundedness. In fact, our results indicate that even in the bounded case, the existence of a parabolic sector strongly depends on symmetry of the vector field with respect to some line passing the equilibrium $($no parabolic sector in symmetric case, and a unique one in nonsymmetric case$)$.
}
\end{remark}

	\begin{remark}\label{rem3}
{ If the elliptic sector is bounded and the vector field is not symmetry, $S_\delta(O)$ includes exactly one parabolic sector.
	However, when the elliptic sector is unbounded, we have only proved that $S_\delta(O)$ contains at most two parabolic sectors, and provide examples which do have two parabolic sectors. But we cannot find examples with zero or one parabolic sector in $S_\delta(O)$, and also cannot prove that $S_\delta(O)$ must have exactly two parabolic sectors. This will remain for further study.}
\end{remark}

	When $(a,b)\in\mathscr{G}_{63}$,
the next two examples illustrate  existence of two parabolic sectors together with one {  unbounded} elliptic and one hyperbolic sectors, and also of one parabolic sector together with one  { bounded} elliptic and one hyperbolic sectors in a neighborhood of a nilpotent equilibrium. {
}
{
\begin{example}\label{eg1}
The planar differential system 	
\begin{equation}\label{ub}	
	\begin{aligned}
	\dot{x}=y,\quad
	\dot{y}=-2x^3-4xy+\varepsilon\widetilde{g}(r)x^2y,
	\end{aligned}
\end{equation}
with $\varepsilon\geqslant0$ sufficiently small, $r^2=x^2+y^2$,  and 	
\[
\widetilde{g}(r)=\left\{ \begin{aligned}
& 1,~&&{\rm if}~0\leqslant r\leqslant1,
\\
& 0,~&&{\rm if}~r\geqslant2,
\end{aligned} \right.
\]
decreasing and smooth for $r\geqslant0$, has the global phase portrait as that shown in {\rm Figure \ref{gpps34}}.
	\begin{figure}[hpt]
 		\centering
{ 			\includegraphics[width=0.35\textwidth]{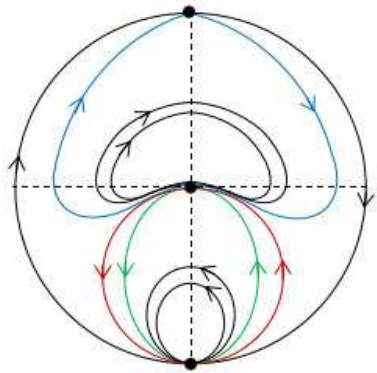}}
 		\caption{The global phase portraits for system \eqref{ub}  with $\varepsilon\geqslant 0$.  }
 		\label{gpps34}
 	\end{figure}
\end{example}
}

Notice that system \eqref{ub} satisfies all conditions of {\rm Theorem \ref{thm4}} and $(a,b) \in \mathscr{G}_{63}$.
{ The proof of the conclusion in this example will be included in the proof of Theorem \ref{thm4}.}

\begin{example}\label{eg2}
Beside the elliptic sector $($bounded by assumption$)$ and hyperbolic sector system \eqref{1} has only  one parabolic sector in $S_\delta(O)$. {\rm Figure \ref{fEHasymOneparabolic}}
illustrates this asymmetric case with one parabolic sector for system  $\dot x= y, \ \dot y=  -2 x^5+x^6 - x  y$ in a neighborhood of the origin by Matlab with the initial points $(0,0.31),\ (0,0.11),\ (0,-0.3)$.
\begin{figure}[hpt]
	\centering
	{			\includegraphics[width=0.35\textwidth]{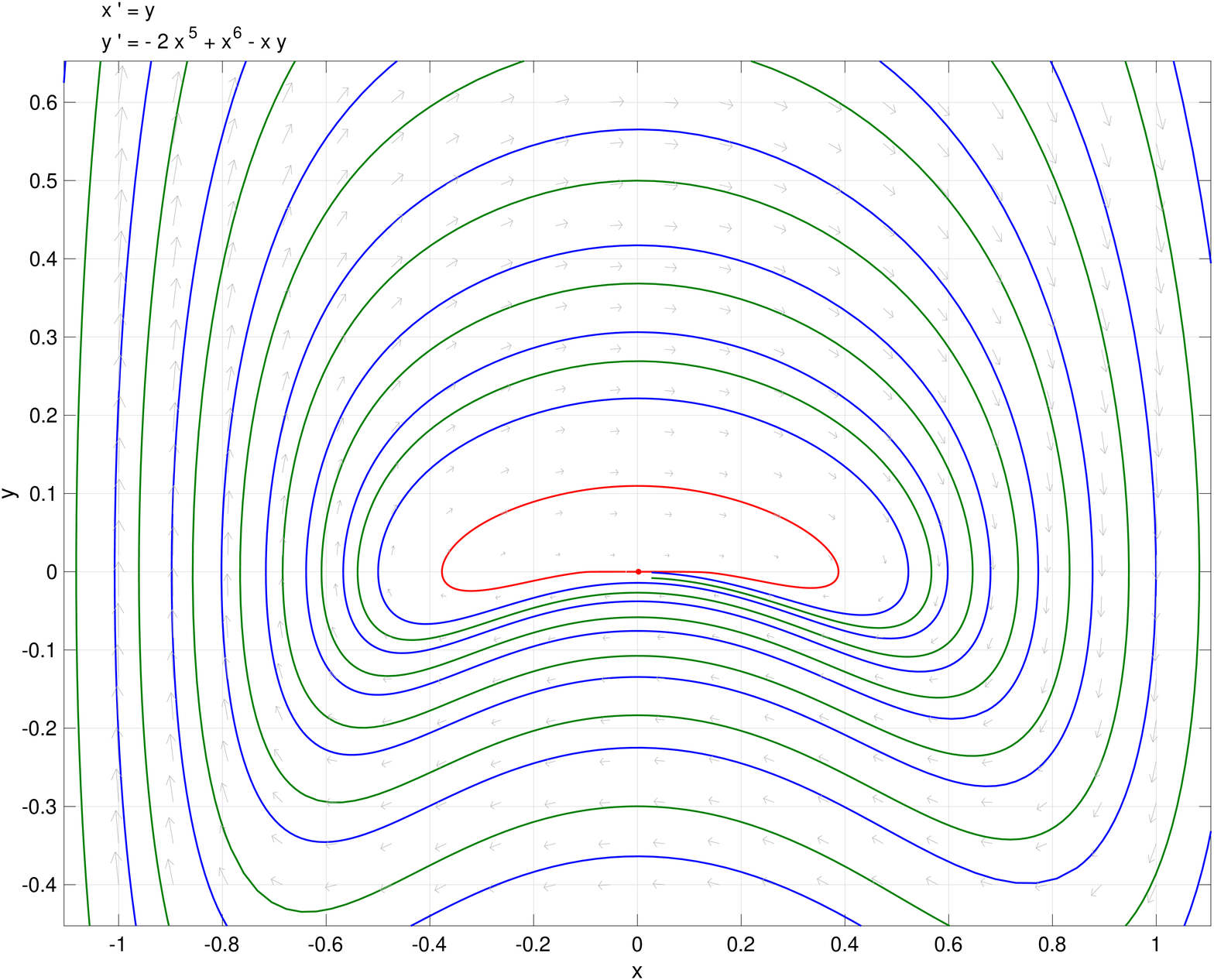}}
	\caption{Asymmetric case with only one elliptic, one hyperbolic sector and one parabolic sectors near $O$.} \label{fEHasymOneparabolic}
\end{figure}
\end{example}


As shown in Table
\ref{QPO} of Theorem \ref{thm4}, the qualitative property of $O$ is related to the boundedness of the elliptic sector when $(a,b)\in\mathscr{G}_{63}$.  The following proposition provides a criterion on existence of the bounded elliptic sector (if exists) for the Li\'enard differential system \eqref{1} with $f(x,y)$ in $x$ only. Taking $f(x,y)=\widetilde{f}(x)$, with the transformation $(x,y)\to(x,y-F(x))$, system \eqref{1} can be written in the classical Li\'enard system
\begin{equation}
	\label{lienard}
\dot{x}=y-F(x),\,\,\dot{y}=-g(x),
\end{equation}
 where $F'(x)=\widetilde{f}(x)$.

\begin{proposition}
	\label{P7}
For system \eqref{lienard}, suppose that
  \begin{itemize}
  \item $xg(x)>0$ for all $x\neq0$, and $\int_{0}^{\infty}g(x)dx=+\infty$,
   \item $g(x)$ and $F(x)$ are Lipschitzian continuous in $\mathbb{R}$, and are analytic in   $|x|<\epsilon$ for some small $\epsilon>0$ is small,
 \item system \eqref{lienard} has an elliptic sector at $O$, i.e. $(a,b)\in\mathscr{G}_{63}$.
 \end{itemize}
 If there exists an $x_0\ne 0$ such that $F(x_0)=0$, the  elliptic sector is bounded.
\end{proposition}

{Notice that Ding \cite{Ding}
	proved the boundedness of the maximal elliptic sector of the following  Li\'enard system
	\begin{eqnarray}
	\dot x=y+\frac{1}{2}x^2-\frac{1}{3}x^3, \  \dot y=-kx^3,
	\label{Lie}
	\end{eqnarray}
	for $k\in(0,1/8)$.}
Applying Proposition \ref{P7}, one can conclude that the elliptic sector of system \eqref{Lie} at the origin is bounded for all $k>0$. This generalizes the result of Ding \cite{Ding}.

	
	\section{Proofs of main results}	
	\label{pmr}

This section is the proofs of {our} main results.

	\subsection{Proof of Theorem \ref{thm1}}
	\label{sec1}
	
	We consider only the case $f(x,y)\geqslant-f(-x,y)$ for $x\geqslant0$ of the condition {\bf (\romannumeral3)},
	since the case    $f(x,y)\leqslant-f(-x,y)$ for $x\geqslant0$ can be transformed to the former by $(x,y,t)\to (-x,y,-t)$.     With the transformation $(x,y)\to(-x, y)$,  system \eqref{1} is changed
	into
	\begin{equation}
	\label{2}
	\left\{\begin{aligned}
	\dot{x}&=-y,
	\\
	\dot{y}&=g(x)-f(-x,y)y.
	\end{aligned}
	\right.
	\end{equation}
It is clear that systems \eqref{1} and \eqref{2} can be rewritten as
	\begin{equation}
	\label{3}
	\frac{dy}{dx}=-\frac{g(x)}{y}-f(x,y),
	\end{equation}
	and
	\begin{equation}
	\label{4}
	\frac{dy}{dx}=-\frac{g(x)}{y}+f(-x,y),
	\end{equation}
	respectively.
	
Assume by contrary that system \eqref{1} has a closed orbit $\Gamma$ in the strip $ \alpha<x<\beta$.  It is obvious that $\Gamma$
	must intersect the negative and positive $y$--axes, saying at $A$ and $B$, respectively, and the positive $x$--axis, saying at $C$. Observe that the orbit arc, denoted by $\widehat{AB}$, of $\Gamma$ in $x\geqslant0$ satisfies equation \eqref{3}. Because of $\dot{x}=y$ there, the orbit arc $\widehat{AC}$ can be represented as $y=y_1(x)$.
Consider the  orbit arc $\widehat{AD}$ of equation \eqref{4} starting from $A$ and ending at $D$ a point on the positive $x$--axis, and represent it by $y=y_2(x)$. See Figure \ref{neco}.
	\begin{figure*}[htp]
		\centering
		\includegraphics[ height=4.5cm,width=4.5cm]{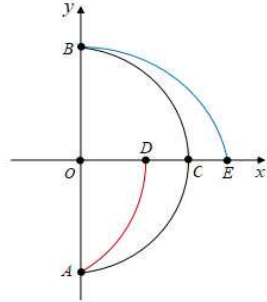}
		\caption{Location of the orbit arcs for showing nonexistence of a closed orbit.}
		\label{neco}
	\end{figure*}
	
	Let $\varphi(x)=y_2(x)-y_1(x)$. Direct calculations show that
	\begin{equation}
	\label{5}
	\begin{aligned}
	\varphi(x)&=y_2(x)-y_1(x)
	\\
	&=\left.(y_2(s)-y_1(s))\right|^x_0
	\\
	&=\int_0^x\left(-\frac{g(s)}{y_2(s)}+f(-s,y_2(s))\right)ds-\int_{0}^{x}\left(-\frac{g(s)}{y_1(s)}-f(s,y_1(s))\right)ds
	\\
	&=\int_{0}^{x}\left(g(s)\frac{y_2(s)-y_1(s)}{y_2(s)y_1(s)}\right)ds+\int_{0}^{x}\left(f(s,y_1(s))+f(-s,y_2(s))\right)ds
	\\
	&=\int_{0}^{x}\left(g(s)\frac{y_2(s)-y_1(s)}{y_2(s)y_1(s)}+f(s,y_1(s))-f(s,y_2(s))\right)ds
	\\
	&\ \ +\int_{0}^{x}\left(f(s,y_2(s))+f(-s,y_2(s))\right)ds
	\\
	&=\int_{0}^{x}M(s)\left(y_2(s)-y_1(s)\right)ds+N(x),
	\end{aligned}
	\end{equation}
	where 
\begin{equation*}
 M(x)=	\left\{\begin{aligned}
	&\frac{g(x)}{y_2(x)y_1(x)}-\frac{f(x,y_2(x))-f(x,y_1(x))}{y_2(x)-y_1(x)},\,\,\,&{\rm if }\,y_1(x)\ne y_2(x),\\
&	\frac{g(x)}{y_2(x)y_1(x)},\,\,\,&{\rm if }\,y_1(x)=y_2(x),
\end{aligned}
\right.
\end{equation*}
and $$ N(x)=\int_{0}^{x}f(s,y_2(s))+f(-s,y_2(s))ds.
	$$
By \eqref{5}, it follows that
	\begin{equation}
	\label{6}
	\begin{aligned}
	\varphi(x)=\int_{0}^{x}M(s)\varphi(s)ds+N(x).
	\end{aligned}
	\end{equation}
	Letting $H(x)=\int_0^xM(s)\varphi(s)ds$,  we obtain
	\begin{equation}
	\label{7}
	\begin{aligned}
	\frac{d H(x)}{dx}=M(x)\varphi(x)=M(x)H(x)+M(x)N(x).
	\end{aligned}
	\end{equation}
	Using  variation of constants formula to solve \eqref{7} yields
	\begin{equation}
	\label{8}
	\begin{aligned}
	H(x)&=\exp\left(\int_{0}^{x}M(s)ds\right)\cdot\int_{0}^{x}\left[M(s)N(s)\exp\left(-\int_{0}^{s}M(\xi)d\xi\right)\right]ds
	\\
	&=\int_{0}^{x}M(s)N(s)\exp\left(\int_{s}^{x}M(\xi)d\xi\right)ds.
	\end{aligned}
	\end{equation}
	Combining the formulas \eqref{6} through \eqref{8}, one has
	\begin{equation}
	\notag
	\begin{aligned}
	\varphi(x)&=H(x)+N(x)
	\\
	&=\int_{0}^{x}M(s)N(s)\exp\left(\int_{s}^{x}M(\xi)d\xi\right)ds+N(x)
	\\
	&=N(0)\exp\left(\int_{0}^{x}M(s)ds\right)+\int_{0}^{x}N'(s)\exp\left(\int_{s}^{x}M(\xi)d\xi\right)ds
	\\
	&=\int_{0}^{x}N'(s)\exp\left(\int_{s}^{x}M(\xi)d\xi\right)ds
	\\
	&=\int_{0}^{x}\left(f(s,y_2(s))+f(-s,y_2(s))\right)\exp\left(\int_{s}^{x}M(\xi)d\xi\right)ds.
	\end{aligned}
	\end{equation}
	By the conditions $f(x,y)\geqslant-f(-x,y)$ for ~$0<x<\min\{-\alpha,
	\beta\}$, $y\in\mathbb R$,
	and $f(x,y)\not\equiv -f(-x,y)$ for  $x\in(-\zeta,\zeta)$,  $y\in\mathbb R$,  $\forall~   \zeta>0$, it follows that
	\begin{equation}
	\notag
	\begin{aligned}
	\varphi(x)=\int_{0}^{x}\left(f(s,y_2(s))+f(-s,y_2(s))\right)\exp\left(\int_{s}^{x}M(\xi)d\xi\right)ds>0.
	\end{aligned}
	\end{equation}
Consequently, $y_2(x)>y_1(x)$. That is, the point $D$ must be located on the left hand--side of the point $C$. Similar arguments verify that the orbit arc $\widehat{BE}$ of equation \eqref{4} starting from $B$ and ending at $E$ on the positive $x$--axis will have $E$ being on the right hand-side of $C$. This implies that system \eqref{4} does not have an orbit arc  linking the points $A$ and $B$, and so has no a closed orbit in the strip $\alpha<x< \beta$.

It completes the proof of Theorem \ref{thm1}. $\hfill\square$

	\begin{remark}
		In the proof of {\rm Theorem \ref{thm1}}, we cannot apply directly the comparison theorem $($see  \cite[Theorem 6.1 of Chapter 1]{Hale}$)$ to obtain $\varphi(x)>0$,
	since we only obtain	$\varphi(x)\geqslant0$ by that theorem.
	That is the reason why we need a more accurate estimate.		
	\end{remark}

		\subsection{Proof of Theorem \ref{thm1b}}

 	\begin{figure*}[htp]
 	\centering
 	\includegraphics[ height=4cm,width=5cm]{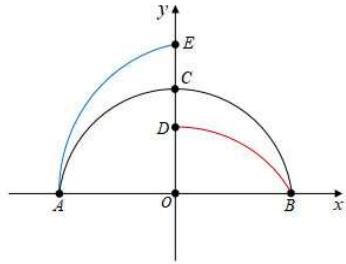}
 	\caption{Location of the orbit arcs for showing nonexistence of a closed orbit.}
 	\label{neco2}
 \end{figure*}
	
	We consider only the case $f(x,y)\geqslant-f(x,-y)$ for $y\geqslant0$ of the condition {\bf (\romannumeral3$'$)},
	since the case    $f(x,y)\leqslant-f(x,-y)$ for $x\geqslant0$ can be transformed to the former by $(x,y,t)\to (x,-y,-t)$.     With the transformation $(x,y)\to(x, -y)$,  system \eqref{1} is changed
	into
	\begin{equation}
	\label{2b}
	\left\{\begin{aligned}
	\dot{x}&=-y,
	\\
	\dot{y}&=g(x)-f(x,-y)y.
	\end{aligned}
	\right.
	\end{equation}
	It is clear that systems \eqref{1} and \eqref{2b} can be rewritten as
\eqref{3}
	and
	\begin{equation}
	\label{4b}
	\frac{dy}{dx}=-\frac{g(x)}{y}+f(x,-y),
	\end{equation}
	respectively.
	
	Assume by contrary that system \eqref{1} has a closed orbit $\Gamma$ in the strip $ \alpha<x<\beta$.  It is obvious that $\Gamma$
	must intersect the negative and positive $x$--axes, saying at $A$ and $B$, respectively, and the positive $y$--axis, saying at $C$. Observe that the orbit arc, denoted by $\widehat{AB}$, of $\Gamma$ in $y\geqslant0$ satisfies  equation \eqref{3}. Because of $\dot{x}=y$ there, the orbit arc $\widehat{AC}$ can be represented as $y=y_1(x)$.
	Consider the  orbit arc $\widehat{AE}$ of equation \eqref{4b} starting from $A$ and ending at $E$ a point on the positive $y$--axis, and represent it by $y=y_2(x)$. See Figure \ref{neco2}.
	
		Let $\varphi(x)=y_2(x)-y_1(x)$.  Then, we can  show $\varphi(0)>0$.
	The remainder of the proof of this result is quiet similar to that given earlier for Theorem \ref{thm1} and so is omitted.
	\subsection{ Proof of Theorem \ref{thm2}}
	
With the transformation 	
\begin{equation}\label{tr}
u=\left({(m+1)G(x)}\right)^{\frac{1}{m+1}}\text{sgn}(x), \quad d\tau=\frac{g(x)}{u^m}dt,	
\end{equation}
system \eqref{1} is changed to
\begin{equation}\label{bh}
\left\{\begin{aligned}
\frac{du}{d\tau }&=y,\\
\frac{dy}{d\tau }&=-u^m-\frac{f(x(u),y)}{g(x(u))}u^my=:-u^m-F(u,y)y,
\end{aligned}
\right.
\end{equation}
where $G(x)=\int_0^{x}g(s)ds$.
{	By \eqref{tr} and $xg(x)>0$ for $x\in(\alpha,0)\cup(0,\beta)$,
	it follows that
$u=\left({(m+1)G(x)}\right)^{1/(m+1)}\text{sgn}(x)$ has  the inverse function  $x=x(u)$ and $x(0)=0$.
Moreover, $x=x(u)$ is increasing.
}


 It is easy to obtain that
 \[
\lim_{u\to0} \frac{u^m}{g(x(u))}=\frac{1}{\sqrt[m+1]{h(0)}}.
 \]
 One can check that {
 $F(u,y)$ is Lipschitzian continuous} for $u\in \left(u(\al),u(\be)\right)$ and $y\in\mathbb R$, i.e., the condition {\bf{(\romannumeral2)}}  of   {\rm Theorem \ref{thm1}} holds. Furthermore,
$$
\frac{f(x,y)}{g(x)}\geqslant(\text{resp.}\leqslant)\frac{f(\hat{x},y)}{g(\hat{x})}
$$
implies
$$
F(u,y)\geqslant(\text{resp.}\leqslant)-F(-u,y),
$$
where $\hat x\leqslant0\leqslant x$ satisfies $\int_{\text{0}}^{x}{g(s)}ds=\int_{\text{0}}^{\hat{x}}{g(s)}ds$. That is, the condition {\bf{(\romannumeral3)}}  of   {\rm Theorem \ref{thm1}} holds.
Moreover, it is easy to check that the conditions {\bf{(\romannumeral1)}} and  {\bf{(\romannumeral4)}}   of   {\rm Theorem \ref{thm1}} hold, too.
By {\rm Theorem \ref{thm1}}, system \eqref{bh} has no a closed orbit surrounding the origin.
 $\hfill\square$

	\subsection{Proof of Theorem \ref{thm3}}
	
As proved in Theorem \ref{thm1}, by contrary assume that  system \eqref{1} has a closed orbit $\Gamma'$ surrounding the origin in the strip $ \alpha<x<\beta$. Let   $A'$ be the intersection point of $\Gamma'$ with the negative  $y$--axis, and let $\widehat{A'C'}$ and $\widehat{A'D'}$, represented  by $y=y_3(x)$ and $y=y_4(x)$, be the orbit arcs of systems \eqref{3} and \eqref{4} starting from $A'$ and ending at respectively $C'$ and  $D'$, which  are  points on the $x$--axis.  Set $\varphi_1(x)=y_4(x)-y_3(x)$. By  a similar calculation as that in the proof of Theorem \ref{thm1},   we can obtain that
		\begin{equation}
		\label{phi1}
		\begin{aligned}
		\varphi_1(x)=\int_{0}^{x}\left(f(s,y_4(s))+f(-s,y_4(s))\right)\exp\left(\int_{s}^{x}M_1(\xi)d\xi\right)ds,
		\end{aligned}
		\end{equation}
	where 	
		{		\begin{equation*}
			M_1(x)=	\left\{\begin{aligned}
				&\frac{g(x)}{y_3(x)y_4(x)}-\frac{f(x,y_4(x))-f(x,y_3(x))}{y_4(x)-y_3(x)},\,\,\,{\rm if }\,y_3(x)\ne y_4(x),\\
				&	\frac{g(x)}{y_3(x)y_4(x)},\,\,\,\qquad \qquad \qquad \qquad \qquad \qquad {\rm if }\,y_3(x)=y_4(x).
			\end{aligned}
			\right.
		\end{equation*}
	}
	Our main goal is to obtain the sign of $\varphi_1(x)$. { Since $\exp\left(\int_{s}^{x}M_1(\xi)d\xi\right)>0
		$, the sign of $\varphi_1(x)$ in \eqref{phi1}    } is only associated with $f(x,y)+f(-x,y)$.
	Thus, $\varphi_1(x)>(\text{resp.} <)\ 0$ when $f(x,y)\geqslant (\text{resp.} \leqslant)-f(-x,y)$ with the equality not identically satisfied. Hence, system \eqref{1} has no a closed orbit surrounding the origin $O$.
	 $\hfill\square$
	
We remark that Theorem \ref{thm3} ensures that system \eqref{1} has no a closed orbit surrounding the origin $O$, but does not provide any information on existence of closed orbits around other equilibria in the  strip $\alpha<x<\beta$. The next example is an illustration on nonexistence of closed orbit around the origin and existence of closed orbit around other equilibrium.
	
	\begin{example}	
	Consider the system
		\begin{equation}\label{mm}
		\left\{\begin{aligned}
		\frac{dx}{dt }&=y,\\
		\frac{dy}{dt }&=-g(x)+\mu_1y+\mu_2xy+x^2y,
		\end{aligned}
		\right.
		\end{equation}
		where  
		\[
		g(x)=\left\{ \begin{aligned}
		& x-1,~&&{\rm if}~x\geqslant\frac{1}{2},
		\\
		& -3x+1,~&&{\rm if}~\frac{1}{4}<x<\frac{1}{2},
		\\
		& x,~&&{\rm if}~-\frac{1}{4}\leqslant x\leqslant \frac{1}{4},
		\\
		& -3x-1,~&&{\rm if}~-\frac{1}{2}<x<-\frac{1}{4},
		\\
		& x+1,~&&{\rm if}~x\leqslant -\frac{1}{2}, 	
		\end{aligned} \right.
		\]
with		$x,y\in\mathbb{R}$,  $\mu_1,\mu_2$ being parameters such that $-\mu_1-\mu_2-1>0$ sufficiently small  and $\mu_1\geqslant0$.
	\end{example}
	
	We claim that system \eqref{mm} has no a closed orbit around $O$, but has one surrounding $(1,0)$ for $-\mu_1-\mu_2-1>0$ sufficiently small  and $\mu_1\geqslant0$. Indeed, one can check that system \eqref{mm} satisfies  {\bf(\romannumeral2)}--{\bf(\romannumeral4)} of Theorem \ref{thm1}. The first argument of the claim follows from Theorem \ref{thm3}.
	
To prove the second argument of the claim, we restrict to a small neighborhood of the equilibrium $(x,y)=(1,0)$, then system \eqref{mm} is of form
	\begin{equation}\label{mn}
	\left\{\begin{aligned}
	\frac{dx}{dt }&=y,\\
	\frac{dy}{dt }&=1-x+\mu_1y+\mu_2xy+x^2y.
	\end{aligned}
	\right.
	\end{equation}
	After the transformation $(x,y)\to(x+1,y)$,  system \eqref{mn} is changed to
	\begin{equation}\label{mo}
	\left\{\begin{aligned}
	\frac{dx}{dt }&=y,\\
	\frac{dy}{dt }&=-x+(\mu_1+\mu_2+1)y+(\mu_2+2)x y+x^2y.
	\end{aligned}
	\right.
	\end{equation}	
	When $\mu_1+\mu_2+1=0$, it is easy to check that $(0,0)$ of \eqref{mo} is an unstable weak focus of order one, see \cite[p.203 ]{CLW}.
	Thus, by Hopf bifurcation an unstable limit cycle births from the origin of system \eqref{mo}
	when $-\mu_1-\mu_2-1>0$ is sufficiently small. The proof is done.

	\subsection{Proof of Theorem \ref{lem2}}
	We  only consider  the case $g(x)=-g(-x)$ for $x>0$
	since the other case $g(x)\not\equiv-g(-x)$  for {$x>0$} can be studied in the same way.  In fact, the transformation \eqref{tr} can sends system \eqref{1} in an equivalent way to system \eqref{bh}, which has the new $g(u)$ an odd function.
	
The condition $f(x,y)\equiv -f(-x,y)$  implies $b=0$. Furthermore, the assumptions {
$g(x)=-g(-x)$  and $f(x,y)\equiv-f(-x,y)$ in $(x,y)\in(0,+\infty)\times\R$
 mean that the vector field 
associated to system \eqref{1}
is symmetric with respect to the $y$--axis.

If $a=g'(0)> 0$, i.e., $(a,b)\in \mathscr{G}_4$,  by Lemma \ref{lem1}, $O$ is a center or focus of system \eqref{1}.
The symmetry of the vector field forces that $O$ must be a center, see Figure \ref{tu2}{\rm{(c)}}.

If $a=0$, i.e., $(a,b)\in \mathscr{G}_6$, Lemma \ref{lem1} { and its proof verify that  $O$ is  a nilpotent} equilibrium.
Again the condition $f(x,y)\equiv -f(-x,y)$ can be written as
$$
-b_nx^n-O(x^{n+1})-yp(x,y)-b_n(-x)^n-O((-x)^{n+1})-yp(-x,y)\equiv 0,
$$
which means that  $n$ is odd. Consequently, $\mathscr{G}_6=  \mathscr{G}_{61}\cup \mathscr{G}_{63}$.

When $(a,b)\in \mathscr{G}_{61}$, by Theorem \ref{thm7.2} of Appendix B and symmetry of the vector field it follows that  $O$ is a nilpotent center, see Figure \ref{tu2}{\rm{(c)}}.

When $(a,b)\in\mathscr{G}_{63}$,
it follows from  Lemmas 7.3 and 7.4 of \cite[Chapter 2]{ZDHD} that  system \eqref{1} has infinitely many orbits approaching $O$ along the negative $x$--axis as $t\to-\infty$,  and also infinitely many orbits approaching $O$ along the positive $x$--axis as $t\to+\infty$ but has no orbits approaching $O$ along other directions. {Next we prove the existence of one elliptic and one hyperbolic sectors under the assumptions of the theorem. As a first step, we consider $g(x)$ and $f(x,y)$ having only the leading terms.}

{	Taking $g(x)=a_kx^k$ and $f(x,y)=b_nx^n$, with $k$ and $n$ odd, such that the assumptions of the theorem hold.
	Let
	\[
	H(x,y):=\frac{a_k}{k+1}x^{k+1}+\frac{y^2}{2},
	\]		
	implying that
	\[
	\left.\frac{dH(x,y)}{dt}\right|_{\eqref{1}}=-b_nx^ny^2<0
	\]
	in the first and fourth quadrants of the $(x,y)$--plane.
	Therefore, { by symmetry of the vector field with respect to the $y$--axis one can conclude that the positive semi--orbit with an initial point $(0,y_0)$
		lies in the interior enclosed by the closed curve
		$H(x,y)={y_0}^2/2$ except the initial point, where $y_0>0$ is small. }
	Let $(0,y_1)$ be
	the first intersection point of the semi--orbit with the $y$--axis.
	Then, $-y_0<y_1\leqslant0$.
	Again by the symmetry of the vector field,  the orbit passing $(0,y_0)$ is a closed one if $y_1<0$, and
	is a homoclinic one if $y_1=0$.
	In conclusion,  $S_\delta(O)$ includes at least
	one elliptic sector and one hyperbolic sector.
	
	{Next we turn to the general $g(x)=a_kx^k+O(x^{k+1})$ and $f(x,y)=b_nx^n+O(x^{n+1})+yp(x,y)$, which satisfy the assumptions of the theorem.
		By symmetry of the vector field and the continuous dependence of solutions with respect to initial values and parameters, it follows that} $S_\delta(O)$ still includes at least
	one elliptic sector and one hyperbolic sector,
}
see Figures \ref{tu2}{\rm{(a)--(c)}}.

{  Finally we study the number of parabolic sectors at $O$. By symmetry of the vector field with respect to the $y$--axis, it follows clearly that the number of parabolic sectors at $O$ is even.
	If all orbits with their $\alpha$--limit set at $O$ have also their $\omega$--limit sets $O$, then $S_\delta(O)$ consists of  exactly	 one elliptic and one hyperbolic sectors, see Figure \ref{tu2}(a).
	If some orbits with their $\alpha$ (resp.  $\omega$)--limit set at $O$ have their $\omega$ (resp. $\alpha$)--limit sets not at $O$, then $S_\delta(O)$ includes at least two parabolic sectors. The previous proofs have shown that $\theta=0$ and $\theta=\pi$ are the only two exceptional directions, and that the orbits approaching $O$ along $\theta=0$ are in positive sense and along $\theta=\pi$ are in negative sense. These imply that there do not have other hyperbolic sections in $S_\delta(O)$. Consequently, there are exactly two parabolic sectors in $S_\delta(O)$, one repelling and the other attracting. We claim that if the elliptic sector is bounded, then system \eqref{1} cannot have parabolic sectors in a neighborhood of $O$. Indeed, by contrary we assume that system \eqref{1} has two parabolic sectors at $O$. Let $\gamma^*$ be the maximal homoclinic orbit to $O$ inside the elliptic sector. By existence of the two parabolic sectors and continuity of solutions with respect to initial values, it follows that there exist orbits inside each of the parabolic sectors which intersect the positive $y$--axis and are outside $\gamma$. By symmetry of the vector fields associated to system \eqref{1}, such kinds of orbits inside the parabolic sectors must be homoclinic to $O$, and so belong to the elliptic sector. This is in contradiction with the assumption that $\gamma^*$ is the maximal homoclinic orbit. The claim follows.
	
	In summary, the above arguments verify the next facts for system \eqref{1} at $O$. If the elliptic sector is bounded, $S_\delta(O)$ is composed of exactly one elliptic and one hyperbolic sectors provided that the elliptic sector is bounded. Figure \ref{fEHsymmetry} illustrates this case for system
	{\begin{equation}\label{eg3}
\dot x= y, \ \dot y = -2 x^5 - x  y,
\end{equation}
 by Matlab with the initial points $(0,-0.3)$,
	$(0,-0.2)$, $(0,-0.1)$, $(0,0.1)$, $(0,0.21)$, $(0,0.41)$.
	\begin{figure}[hpt]
		\centering
		{			\includegraphics[width=0.35\textwidth]{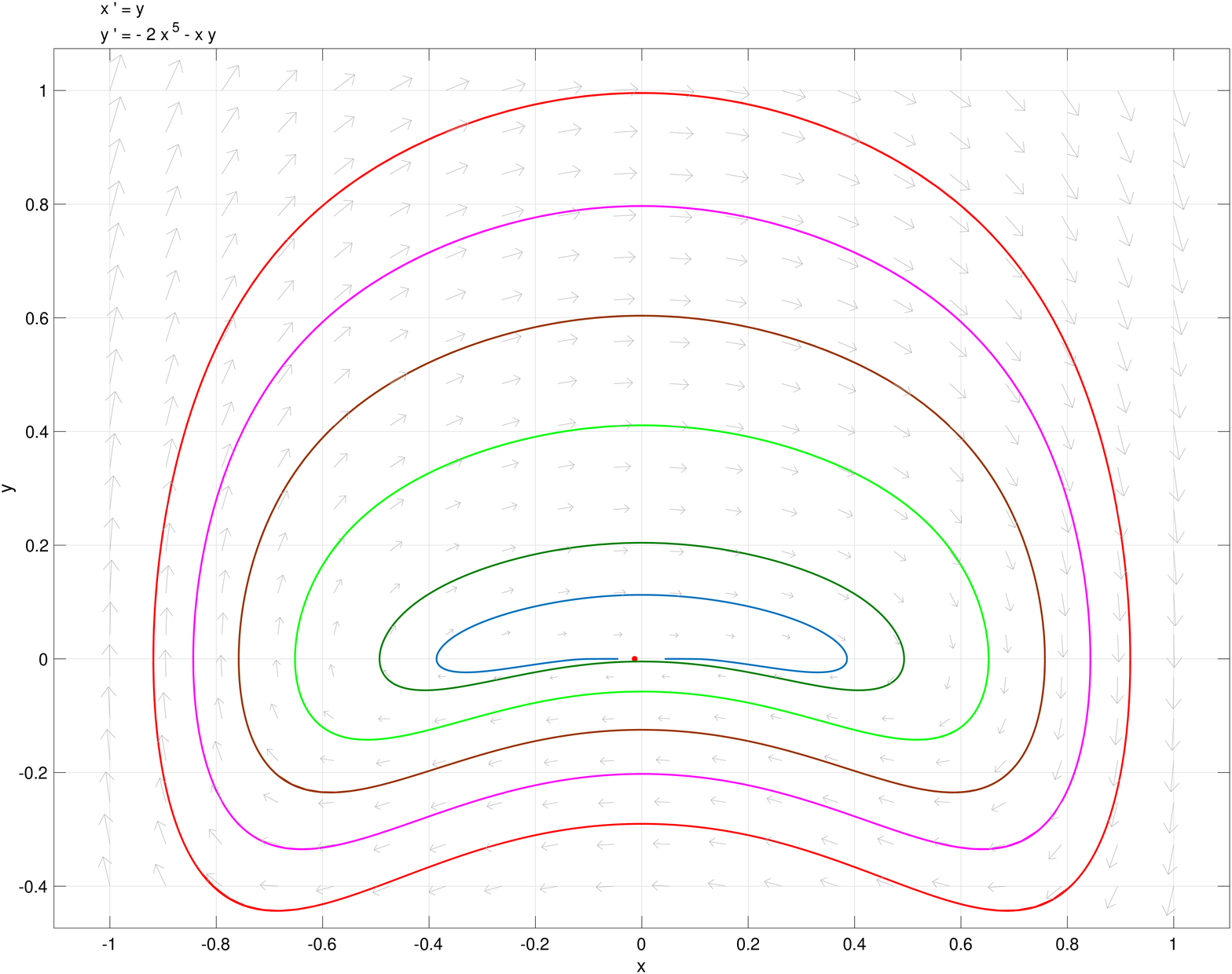}}
		\caption{Symmetric case with only one elliptic and one hyperbolic sectors near $O$.} \label{fEHsymmetry}
	\end{figure}
}

	If the elliptic sector is unbounded, $S_\delta(O)$ consists of either exactly one elliptic and one hyperbolic sectors (see Figure \ref{tu2}(c)), or one elliptic, one hyperbolic and two parabolic sectors (see Figure \ref{tu2}(b)). Example \ref{eg1} with $\varepsilon=0$ is in this symmetric case and it has two parabolic sectors in a neighborhood of $O$, and an unbounded elliptic and a hyperbolic sectors. } {At the moment we do not have an example without parabolic sectors, and also cannot prove its nonexistence.}
}

	\subsection{Proof of Theorem \ref{thm4}}
	\label{proofthm4}

To prove this theorem, we first analyse the equilibrium at the origin of system \eqref{1}.

 	\begin{lemma}\label{lem1}
		Suppose that all  conditions  of {\rm Theorem \ref{thm4}} hold. The unique equilibrium $O=(0,0)$ has the following  qualitative properties.
		\begin{itemize}
			\item[\bf (F1)]
		 $O$ is a  stable node for $(a,b)\in\mathscr{G}_1\cup \mathscr{G}_3$,  and is a stable focus for  $(a,b)\in \mathscr{G}_2$;
			
			\item[\bf (F2)] $O$ is a center or focus for $(a,b)\in \mathscr{G}_4$;

			\item[\bf (F3)] $O$ is a  stable node for $(a,b)\in\mathscr{G}_5$;
			
	\item[\bf (F4)]	$O$ is a degenerate equilibrium for $(a,b)\in \mathscr{G}_6$.
		\end{itemize}
	\end{lemma}
	\begin{proof}
It is easy to check that system \eqref{1} has the unique equilibrium $O=(0,0)$, at which the Jacobian matrix of system \eqref{1} is
\begin{equation*}
	\begin{aligned}
		\left.	J:= \begin{pmatrix}
			0 & 1  \\
			-g'(0)-{{f}_{x}}(x,y)y & -{{f}_{y}}(x,y)y-f(x,y)
		\end{pmatrix}\right| _{(0,0)}= \begin{pmatrix}
			0 & 1  \\
			-a & -b
		\end{pmatrix}
	\end{aligned}.
\end{equation*}
The eigenvalues of $J$ are
$$
\lambda_{1,2}=\frac{-b\pm\sqrt{b^2-4a}}{2}.
$$
Obviously, the equilibrium  $O$ is a stable node when $(a,b)\in\mathscr{G}_1\cup \mathscr{G}_3$, and is a stable focus when $(a,b)\in\mathscr{G}_2$.
For $(a,b)\in\mathscr{G}_4$, $\lambda_{1,2}$ is a pair of pure imaginary numbers, and so $O$ is a center or focus. For $(a,b)\in\mathscr{G}_5$, one of $\lambda_{1,2}$ is zero, and the other is nonzero.  In $S_\delta (O)$, $g(x)$ and $f(x,y)$ can be rewritten as
$$
g(x)=a_kx^k+O(x^{k+1}), \quad f(x,y)=b+\hat{f}(x,y),
$$
where $a_k>0$, $k$ is odd, $b>0$ and $\hat{f}(0,0)=0$.
With the transformation $ (x,y)\to \left((x-y)/b, y\right)$, system \eqref{1} is changed into
\begin{equation}\label{1change}
	\dot{x}=by-g\left(\frac{x-y}{b}
	\right)-f\left(\frac{x-y}{b},y\right)y,\quad \dot{y}=-g\left(\frac{x-y}{b}\right)-f\left(\frac{x-y}{b},y\right)y.
\end{equation}
By the implicit function theorem,
$$
-g\left(\frac{x-y}{b}\right)-f\left(\frac{x-y}{b},y\right)y=0
$$
has a unique root $y=\phi(x)=-a_k/b^{k+1}x^k+o(x^k)$ for small $|x|>0$.
Thus,  in system \eqref{1change}
$$
\dot{x}|_{y=\phi(x)}=b\phi(x)=-\frac{a_k}{b^k}x^k+o(x^k).
$$
By Theorem \ref{thm7.1} of Appendix B, the equilibrium $O$ is a stable node.
Moreover, by calculation, the origin of system \eqref{1} has four exceptional directions
$$
\theta_1:=\pi -\arctan \left(\frac{b-\sqrt{b^2-4a}}{2}\right),\,\,\,
\theta_2:=\pi -\arctan  \left(\frac{b+\sqrt{b^2-4a}}{2}\right),
$$
$$
\theta_3:=2\pi -\arctan \left(\frac{b-\sqrt{b^2-4a}}{2}\right),\,\,\,
\theta_4:=2\pi -\arctan \left(\frac{b+\sqrt{b^2-4a}}{2}\right)
$$
 when	$(a,b)\in \mathscr{G}_{1}\cup \mathscr{G}_{5}$,
or has two exceptional directions $\theta_5:=\pi-\arctan \left(b/2\right)$, $\theta_6:=2\pi-\arctan \left(b/2\right)$
when $(a,b)\in \mathscr{G}_{3}$.    {In the former the origin is a normal node, and in the latter the origin is an improper node.}

When $(a,b)\in\mathscr{G}_6$, $\lambda_1=\lambda_2=0$,  and the linearization of system \eqref{1} at $O$ {has its coefficient matrix being nilpotent}. This means that $O$ is a nilpotent equilibrium.
The proof is finished.
	\end{proof}

 Now we are back to the proof of Theorem \ref{thm4} and focus on  the case $g(x)=-g(-x)$ for { $x>0$}. The case $g(x)\not\equiv-g(-x)$  for { $x>0$} can be handled {by applying transformation \eqref{tr}}, and so is omitted.
Since all conditions of Theorem \ref{thm1} hold, system \eqref{1} has no a closed orbit around the origin.

If $b>0$ { (i.e. $(a,b)\in\mathscr{G}_{1}\cup\mathscr{G}_{2}\cup\mathscr{G}_{3}\cup\mathscr{G}_{5}$) }, by Lemma \ref{lem1},  it is easy to get the local phase portraits of system \eqref{1} at $O$, as those in Figures \ref{1pps2}{\rm{(a)--(c)}}.

For $b=0$, to classify the local phase portraits of system \eqref{1} at $O$ we distinguish  two cases: $a>0$ and $a=0$.

\noindent{\it Case }1. $a>0$ { (i.e. $(a,b)\in\mathscr{G}_{4}$)}.  Since $O$ is a center or a weak focus  by {\bf (F2)}  of Lemma \ref{lem1}, and there is no a closed orbit around it, the equilibrium $O$ must be a weak focus of  system \eqref{1}.

To characterize the stability of $O$, let $M$ be a point on the positive $y$--axis  in  $S_\delta(O)$, see Figure \ref{tu10}{\rm(a)}, and let $\varphi(M,I^+)$ be the positive orbit of system \eqref{1} having the initial point $M$. Obviously $\varphi(M,I^+)$ has to intersect the negative $y$--axis at a first time at a point $P$, and then return to the positive $y$--axis again at a first time at a point $N$. Comparing both the positive orbit arc and the negative one starting from $N$ of system \eqref{1} in the $x\geqslant 0$ half plane,
	\begin{figure}[hpt]
		\centering
		\subfigure[{ The orbit arc starting from $M$ for system \eqref{1}} ]{
			\includegraphics[width=0.3\textwidth]{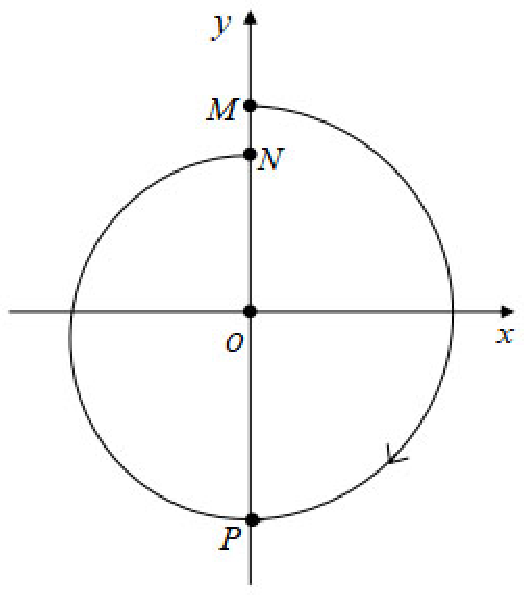}}
		\quad
		\subfigure[{ The orbit arc starting from $P$ for equations \eqref{3} and \eqref{4}}]{
			\includegraphics[width=0.3\textwidth]{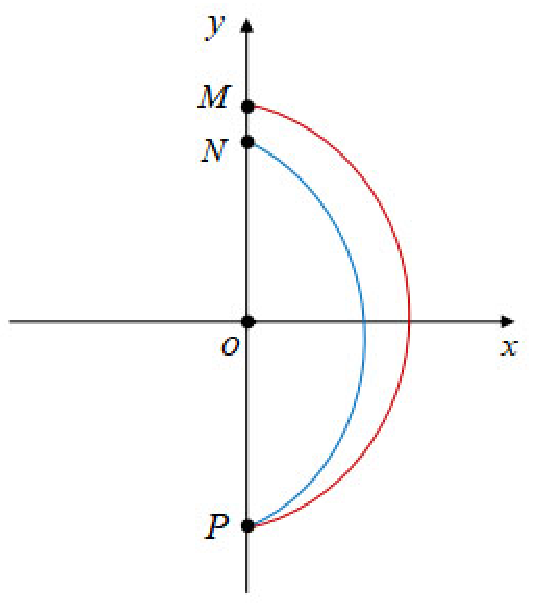}}
		\caption{{The orbit arcs for showing the stability of $O$.}} \label{tu10}
	\end{figure}	
one gets from $f(x,y)\geqslant-f(-x,y)$ that
	$$
	\left.	\frac{dy}{dx}\right |_\eqref{3}\leqslant \left. \frac{dy}{dx}\right |_\eqref{4}.
	$$
This implies that
	${\varphi}(N,I)|_\eqref{3}$ must be strictly located on the right hand--side of ${\varphi}(N,I)|_\eqref{4}$ except at $N$, and so $y_N<y_M$, see Figure \ref{tu10}{\rm (b)}, where $I$ is a suitable time interval corresponding to the orbit arc. Thus, the origin $O$ is a stable focus. As a result, we have the local phase portrait of system \eqref{1},  as shown in Figure \ref{1pps2}{\rm{(b)}}.
	
\noindent{\it Case }2. $a=0$ { (i.e. $(a,b)\in\mathscr{G}_{6})$}.  By {\bf (F4)} of Lemma \ref{lem1}
 the equilibrium  $O$ is {  a nilpotent one}. Recall that  system \eqref{1} has no a closed orbit around $O$, { and that $\mathscr{G}_{6}=\mathscr{G}_{61}\cup\mathscr{G}_{62}\cup\mathscr{G}_{63}$. According to this decomposition on $\mathscr{G}_{6}$, the next proof is divided in three subcases.}
	
\noindent{\it Subcase} 2.1. $(a,b)\in \mathscr{G}_{61}$. The arguments adopted in the proof of Case 1 work here and show that $O$ is a weak focus. Hence, the local phase portrait of system \eqref{1} at $O$ is also that as shown in Figure \ref{1pps2}{\rm{(b)}}.
	
\noindent{\it Subcase} 2.2. $(a,b)\in \mathscr{G}_{62}$. In this case, $n$ is even and $b_n\ne0$.
With the  polar transformation  $(x,y)=(r\cos\theta, r\sin\theta)$,
  system \eqref{1} is changed  to
	\begin{equation*}
	\frac{1}{r}\frac{dr}{d\theta}=\frac{H(\theta)+\widetilde{H}(\theta,r)}{G(\theta)+\widetilde{G}(\theta,r)},
	\end{equation*}
 where  $G(\theta)=-\sin^2\theta$,   $\widetilde{H}(\theta,r),  \widetilde{G}(\theta,r)\to 0$ as $r\to0$. Clearly,  $G(\theta)=0$ has the two roots $\theta=0,\,\,\pi$.
 Therefore, if an orbit approaches $O$ in the positive or negative limit, it must be along the $x$--axis. We now prove that $O$ is stable. Indeed,
set $\triangle\widehat{OAB}=\{(x,y):0<x\ll\delta,-x^{p+1}<y<x^{p+1}\}$ with $\delta>0$ sufficiently small.
  Let
	\[
	E(x,y)=\frac{y^2}{2}+\int_{\text{0}}^{x}{g(s)}ds.
	\]
Then along system \eqref{1}
	\[
	\begin{aligned}
\left.	\frac{dE }{dt}\right|_{(x,y)\in\triangle\widehat{OAB}}=-f(x,y){{y}^{2}}
	=-\left(b_nx^n+O(x^{n+1})+yp(x,y)\right)y^2
	<0 .
	\end{aligned}
	\]
This shows that when an orbit approaches $O$, it must be in the positive sense. Consequently, $O$ is a stable node and its local phase portrait is that as shown in Figure \ref{1pps2}{\rm{(c)}}.
	
\noindent{\it Subcase} 2.3.  $(a,b)\in \mathscr{G}_{63}$.
{ From the assumption of the theorem, the expressions \eqref{eab} and \eqref{eab1} hold. Then we are in the conditions in the last row of Table \ref{T7.2}. By Theorem \ref{thm7.2} (i.e. \cite[Theorem 7.2 of Chapter 2]{ZDHD}) system \eqref{1} has a neighborhood of $O$ which contains an elliptic sector and one hyperbolic sector. Moreover, from the proofs of Andreev \cite{Andreev} or of Zhang et al \cite{ZDHD}, system \eqref{1} has no other elliptic sectors and hyperbolic sectors.}

{
	By the facts $\dot x=y>0$ for $y>0$ and $\dot x=y<0$ for $y<0$, it follows that any homoclinic orbit inside the elliptic sector of system \eqref{1} at $O$ must intersect both the negative and positive $x$--axes.
Then, the natural question is that $S_\delta(O)$ includes any  parabolic sectors or not. It is clear that $S_\delta(O)$ has at most two  parabolic sectors, since system  \eqref{1} has exactly two exceptional directions $\theta=0,\pi$, { along each of which there are
		infinitely many orbits approaching the origin, and all orbits approaching $O$ are only along $\theta=0$  as $t\to+\infty$, and only along $\theta=\pi$ as $t\to-\infty$. }

 In what follows, we distinguish boundedness or not of the elliptic sector in $S_\delta(O)$
{ to answer the above question}.

{
	\noindent{\bf The elliptic sector is bounded.}}
{The similar arguments as those in the proof of Theorem \ref{lem2} verify that $S_\delta(O)$ cannot include two parabolic sectors. We next prove that in this asymmetric case, $S_\delta(O)$ consists of one parabolic sector locally in the fourth quadrant, together with the bounded elliptic sector and the hyperbolic sector.  {Recall that we have taken $b_n\ge 0$ in \eqref{eab}. If $b_n<0$ the location of the parabolic sector may vary locally from the fourth quadrant to some other one.}
}

{ To prove our results, we construct} a new  vector field
\begin{equation}
	\label{new}
	\left\{\begin{aligned}
		\dot{x}&=y,
		\\
		\dot{y}&=B(x,y),
	\end{aligned}
	\right.
\end{equation}
where
$$
B(x,y)=\left \{
\begin{aligned}&-g(x)-f(x,y)y,\,\,  &{\rm if }\,x<0,
	\\
	&-g(x)+f(-x,y)y,\,\,&{\rm if }\,x>0.
\end{aligned}
\right.
$$
{ 
Since systems \eqref{1} and \eqref{new} are the same in the $x<0$ half plane, in which if two orbits of the two systems intersect, they must coincide there. Of course, they could be different in the $x>0$ half plane.

\begin{figure*}[htp]
	\centering
	\includegraphics[ height=4cm,width=5cm]{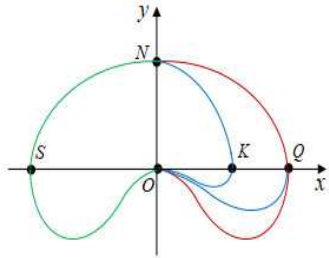}
	\caption{{ Relative positions of the maximal homoclinic orbits.}
	}
	\label{tu12}
\end{figure*}

	Let  the maximal homoclinic orbit for system \eqref{1}
	intersect the positive $x$--axis at $K$, the negative $x$--axis at $S$ and the positive $y$--axis at $N$ {(the green  one for $x<0$ and the blue for $x>0$ in Figure \ref{tu12})}.
{By contrary, we assume that $S_\delta(O)$ has no parabolic sectors.
Then} the orbit arc $\widehat{OSN}$ is the outermost orbit connecting $O$ in $x<0$. {Note that the vector fields of system \eqref{new} are symmetric with respect to the $y$--axis.} Since the vector fields of system \eqref{1} and \eqref{new} are same in the region
$x<0$ and
  the elliptic sector of system \eqref{1}  is bounded, it follows that
the elliptic sector of system \eqref{new} is also bounded, and its maximal homoclinic orbit has the same intersection point $N$ with the $y$--axis as that of the elliptic sector for system \eqref{1}. Let the maximal homoclinic orbit for system \eqref{new}
intersect the positive $x$--axis at $Q$ {(the green  one for $x<0$ and the red for $x>0$ in Figure \ref{tu12})}.

We claim that $K$ lies on the left hand side of $Q:=(x_Q,0)$, and that $\varphi(Q,I^+)$ for system \eqref{1} connects $O$ as $t\to+\infty$. For proving the first part,
 let the orbit arcs { $\widehat{NK}$ for system \eqref{1} and $\widehat{NQ}$ for system \eqref{new}} have  respectively the expressions $y=y_5(x)$ and $y=y_6(x)$. As shown in the proof of Theorem \ref{thm1}, we can  obtain that
$$
y_6(x)-y_5(x)=\int_{0}^{x}M_2(s)(y_6(s)-y_5(s))ds+N_2(x),
$$
	where
	\begin{equation*}
		M_2(x)=	\left\{\begin{aligned}
			&\frac{g(x)}{y_5(x)y_6(x)}-\frac{f(x,y_6(x))-f(x,y_5(x))}{y_6(x)-y_5(x)},\,\,\,&{\rm if }\,y_5(x)\ne y_6(x),\\
			&	\frac{g(x)}{y_5(x)y_6(x)},\,\ \  \ \ \ \ \ \ \  \ \ \ \ \  \ \ \  &{\rm if }\,y_5(x)=y_6(x)
		\end{aligned}
		\right.
	\end{equation*}
and $$ N_2(x)=\int_{0}^{x}f(s,y_6(s))+f(-s,y_6(s))ds.
$$
With a similar  calculation { as in the proof of Theorem \ref{thm1}, it follows that} the expression of $y_6(x)-y_5(x)$ has the following form
\begin{equation*}
	y_6(x)-y_5(x)=\int_{0}^{x}\left(f(s,y_6(s))+f(-s,y_6(s))\right)\exp\left(\int_{s}^{x}M_2(\xi)d\xi\right)ds.
\end{equation*}
Since $f(x,y)\geqslant-f(-x,y)$ for $x>0$ and $f(x,y)\not\equiv-f(-x,y)$ for $x\in(0,\zeta)$,  $\forall~   \zeta>0$, $y_6(x)>y_5(x)$ for $x\in(0,x_K)$,
where $x_K$ is the abscissa of $K$. Thus,
 $\widehat{NK}$ lies  on the left hand side of $\widehat{NQ}$. The first part of the claim follows.
To prove the second part,
 let the orbit arcs $\varphi(Q,I^+)$ for system \eqref{1}  and $\widehat{QO}$ for system \eqref{new}  have the expressions $y=y_7(x)$ and $y=y_8(x)$,
respectively. We can similarly obtain that
\begin{equation*}
y_8(x)-y_7(x)=\int_{x_Q}^{x}\left(f(s,y_8(s))+f(-s,y_8(s))\right)\exp\left(\int_{s}^{x}M_3(\xi)d\xi\right)ds<0
\end{equation*}
 for $x\in(0,x_Q)$,
where
	\begin{equation*}
	M_3(x)=	\left\{\begin{aligned}
		&\frac{g(x)}{y_7(x)y_8(x)}-\frac{f(x,y_7(x))-f(x,y_8(x))}{y_7(x)-y_8(x)},\,\,\,&{\rm if }\ y_7(x)\ne y_8(x),\\
		&	\frac{g(x)}{y_7(x)y_8(x)},\,\,\,&{\rm if }\ y_7(x)=y_8(x).
	\end{aligned}
	\right.
\end{equation*}
Thus, the orbit arc $\varphi(Q,I^+)$ for system \eqref{1} lies {above} the orbit arc $\widehat{QO}$ for system \eqref{new}.
{Consequently, the orbit arc $\varphi(Q,I^+)$ for system \eqref{1} cannot intersect the negative $y$--axis and must approach the origin as $t\rightarrow +\infty$. This proves the second part and so the claim.

Since $\widehat{OSNKO}$ is the maximal homoclinic orbit inside the elliptic sector of system \eqref{1} at $O$, by this last claim all orbits of system \eqref{1} with the initial points located in between $K$ and $Q$ on the positive $x$--axis will positively approach the origin, and they form a part of a parabolic sector at $O$ locally in the fourth quadrant. This proves that system \eqref{1} has a parabolic sector at $O$ locally in the fourth quadrant. {This is in contradiction with the contrary assumption. Consequently, system \eqref{1} must have at least one parabolic sector in $S_\delta(O)$.}

Because of the existence of the parabolic sector for system \eqref{1} at $O$ locally in the fourth quadrant, it forces that under the assumption of the theorem, system \eqref{1} in this asymmetric case cannot have a  parabolic sector at $O$ locally in the third quadrant. Otherwise we will be in contradiction with the fact that $\widehat{OSNKO}$ is the maximal homoclinic orbit inside the elliptic sector of system \eqref{1} at $O$.

In summary, we have proved that $S_\delta(O)$ of system \eqref{1} contains  one and only one parabolic sector, which are locally in the fourth quadrant. 
Example \ref{eg2} provides a concrete verification on this case via Figure \ref{fEHasymOneparabolic}. }
 }

{\noindent{\bf The elliptic sector is unbounded.} 
{Beside the elliptic sector (unbounded by the assumption) and the hyperbolic sector, $S_\delta(O)$ for system \eqref{1} possibly contains also either zero, one, or two parabolic sectors.
 }

Example \ref{eg1} when $\varepsilon>0$ is in the asymmetric case, and its origin has a neighborhood consisting of one elliptic, one hyperbolic and two parabolic sectors, where the elliptic sector is unbounded.} 
{
When an elliptic sector is unbounded, we cannot numerically simulate its existence.
For completing the proof of the theorem, we now theoretically prove the conclusion in Example \ref{eg1}, and so by a concrete example shows existence of a nilpotent equilibrium whose neighborhood can contain two parabolic sectors besides one elliptic and one hyperbolic sectors.}

System \eqref{ub} when $\varepsilon=0$ can be reduced to
\begin{equation}\label{ubo}	
\begin{aligned}
\dot{x}=y-x^2,\quad
\dot{y}=-2xy,
\end{aligned}
\end{equation}
by the transformation $(x,y)\to(x,y-x^2)$.
It is easy to check that system \eqref{ubo} is symmetry on the $y$--axis, and has a unique equilibrium
and two invariant orbits along the $x$--axis. Since system \eqref{ubo} can be written in a Bernoulli equation of $x$ with respect to $y$, { or in a linear homogeneous equation via $z=x^2$, by their solutions} one can get the global phase portrait in the Poincar\'e disc of  system  \eqref{ubo}, as shown in {\rm Figure \ref{gpps4}(a)}
\footnote{System \eqref{ubo} and its global phase portrait in the Poincar\'e disc were originally obtained by professor Feng Li. }.
Therefore, it follows from the transformation that system \eqref{ub} when $\varepsilon=0$ has the global phase portrait as shown in {\rm Figure \ref{gpps34}}.
We remark that system \eqref{ubo} has two pairs of equilibria at the infinity, whereas system \eqref{ub} $($when $\varepsilon=0$ or not$)$ has a unique pair of equilibria, and that the two invariant lines on the $x$--axis of system  \eqref{ubo} have been deformed to two orbits connecting the origin and the equilibrium at infinity in the
negative $y$--axis of system \eqref{ub} when $\varepsilon=0$, see the red orbits in {\rm Figures \ref{gpps34} and \ref{gpps4}(a)}.

 	\begin{figure}[hpt]
 		\centering
 		\subfigure[{ for system \eqref{ubo} }]
 { 			\includegraphics[width=0.35\textwidth]{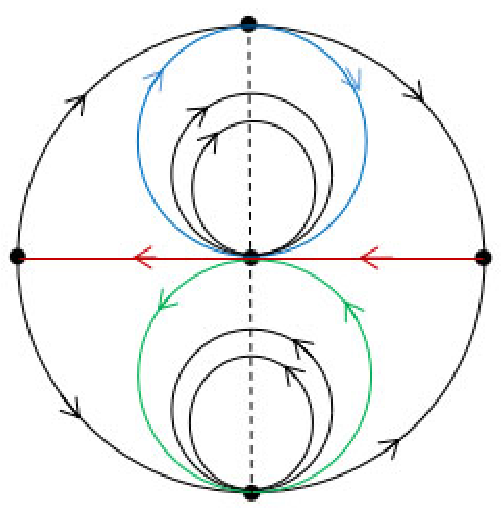}}
 		\quad
 		\subfigure[{ for system \eqref{ub}  with $\varepsilon>0$ }]
 { 	\includegraphics[ width=0.35\textwidth]{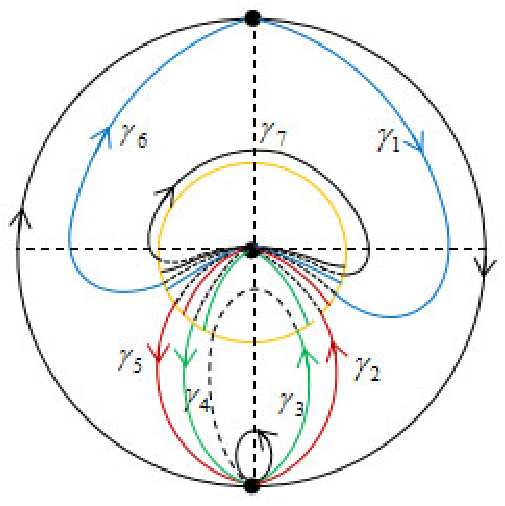}}
 		\caption{The global phase portraits and variations of orbits with $\varepsilon$.  }
 		\label{gpps4}
 	\end{figure}


{ We next show that system \eqref{ub} when $\varepsilon>0$ sufficiently small has the same topological phase portrait as that shown in Figure \ref{gpps34}. Indeed,
on the one hand, we can check that the vector fields of system \eqref{ub} in the region $r\geqslant2$ are independent of $\varepsilon>0$ or not, as shown in Figure \ref{gpps4}(b), and that }
when $\varepsilon>0$ is sufficiently small,
 the origin of
 system \eqref{ub} still has only two exceptional directions $\theta=0$ and $\theta=\pi$, infinitely many orbits
 connect the origin along $\theta=0$ as $t\to+\infty$ and infinitely many orbits
 connect the origin along $\theta=\pi$ as $t\to-\infty$.
 On the other hand,
{ the vector field associated to system  \eqref{ub} rotates counterclockwise in the disc $r<2$ for $\varepsilon>$ increasing, Figure \ref{gpps4}(b) illustrates this variation by the representative orbits $\gamma_1,\ldots,\gamma_7$ and their positive and negative limits. This fact together with dependence of solutions with respect to parameters shows that the elliptic and hyperbolic sectors of system \eqref{ub} at $O$ will both persist for $\varepsilon>0$ sufficiently small. In addition, parabolic sectors are structurally stable. These verify that system  \eqref{ub} when $\varepsilon>0$ sufficiently small has still the global phase portrait in the Poincar\'e disc as that shown in  {\rm Figure \ref{gpps34}}.

Since system \eqref{1} cannot have more than two parabolic sectors, we can finish the proof of the theorem.
\qed
}

As commented in Remark \ref{rem3}, at the moment we could not find examples in this asymmetric case with an unbounded elliptic sector and either zero or one parabolic sector, and also cannot prove nonexistence of systems of form \eqref{1} which are in asymmetric case and have an unbounded elliptic sector and zero or one parabolic sector.

 \subsection{Proof of Proposition \ref{P7}}

	It suffices to consider the case $x_0>0$. The case $x_0<0$ can be handled in a similar way via the transformation $(x,t)\to(-x,-t)$.
	Without loss of generality, we consider $x_0>0$ to be the minimal one
 satisfying $F(x_0)=0$.

Set $A:=(x_0,0)$, and let $\varphi(A,I^+)$ and $\varphi(A,I^-)$ be respectively the positive and negative semi--orbits with their initial points at $A$.
 We claim that $\varphi(A,I^-)$ and $\varphi(A,I^+)$  must intersect the positive and negative $y$--axes respectively, see Figure \ref{P1}. {Indeed, since $dx/dt=0$ on $y=F(x)$ and $dy/dt=-g(x)<0$ for $x>0$, it follows that when the orbit arc of $\varphi(A,I^-)$ is in the first quadrant, it must be located above the curve $y=F(x)$, and that $\varphi(A,I^-)$ either intersects the positive
 $y$--axis, or always keep in the first quadrant and its $\alpha$--limit set is} the equilibrium at infinity in the positive $y$-axis. We now prove that the latter cannot happen.
 Set $E(x,y):=y^2/2+\int_{0}^{x}g(s)ds$,
   then
   \[
\left.   \frac{dE}{dt}\right|_{\eqref{lienard}}=-g(x)F(x).
   \]
    Let $y=y(x)$ be the expression of $\varphi(A,I^-)$, and $x=\hat x(y)$ be its inverse.
  One can check that
\begin{equation}
	\label{A}
E(x_0,0)-E(0,+\infty)=\int_{0}^{x_1}\frac{-g(x)F(x)}{y(x)-F(x)}dx+\int_0^{y(x_1)}F(\hat x(y))dy
\end{equation}
for any fixed $x_1\in(0,x_0)$. Note that the left hand of \eqref{A} is equal to $-\infty$, whereas the right hand of \eqref{A} is finite, a contradiction.  Hence $\varphi(A,I^-)$ must intersect the positive
 $y$--axis, denote by $B:=(0,y_1)$ this intersection point. The same arguments verify also that $\varphi(A,I^+)$ must intersect the negative $y$--axis, saying at $C:=(0,y_2)$. This proves the claim.
	  	\begin{figure*}[htp]
	  \centering
	  \subfigure[{\quad}]{
	  	\includegraphics[width=0.33\textwidth]{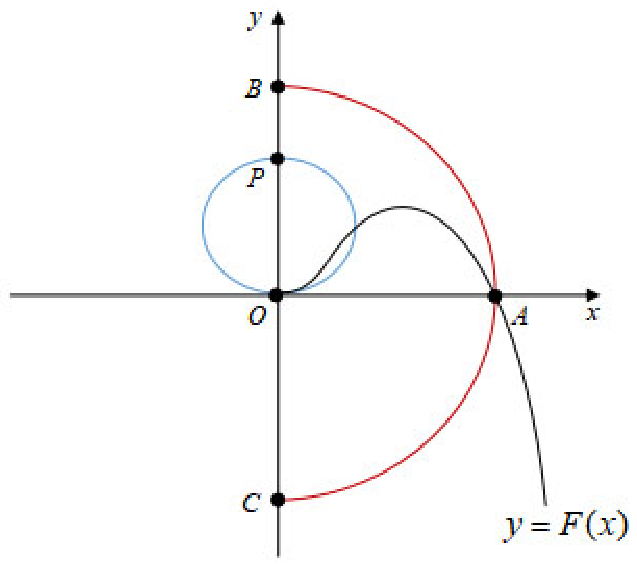}}
	  \quad
	  \subfigure[{\quad}]{
	  	\includegraphics[width=0.33\textwidth]{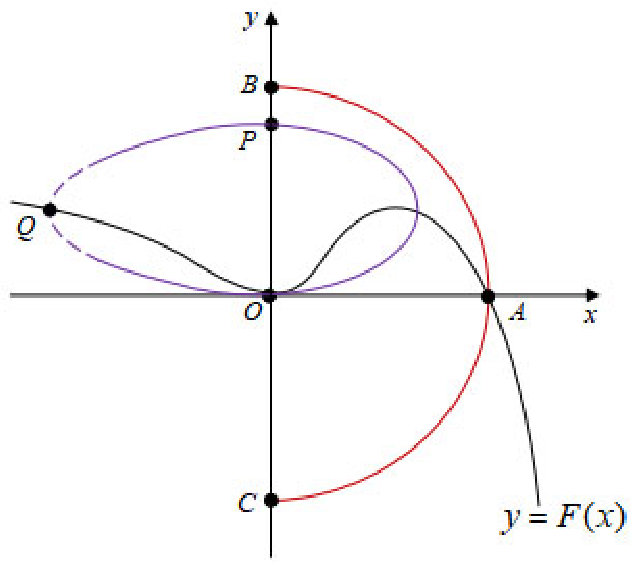}}
	  	\caption{Location of the orbit arcs for showing boundedness of elliptic sector.}
	  	\label{P1}
	  \end{figure*}

We claim that any homoclinic orbit inside the elliptic sector at the origin must positively approach the origin from the first quadrant and negatively approach the origin from the second quadrant.
{ Since	the condition $(a,b)\in\mathscr{G}_{63}$ implies that $b_n>0$, $n$ is odd, $F(x)>0$ in $(-\epsilon,0)\cup(0,\epsilon)$ for some small $\epsilon>0$ and $F(x)=b_nx^{n+1}/(n+1)+O(x^{n+2})$ for small $|x|$, 
  $dy/dx=g(x)/(F(x)-y)>0$ near the origin and in the fourth quadrant, and $dy/dx=g(x)/(F(x)-y)<0$ near the origin and in the third quadrant.
}
 Indeed, by contrary, assume that there is an orbit connecting the origin in the third quadrant.
Let   $(-\epsilon, y(-\epsilon))$ be a point on the orbit and be closed to the origin. It implies $dy/dx>0$ at this point, a contradiction. The claim hold.
Hence, all orbits connecting $O$  in a small neighborhood of the origin lie in either the first or the second quadrants and are tangent to the $x$--axis at the origin.

We claim that $\varphi(A,I)$ cannot form a homoclinic orbit to the origin $O$. Otherwise, $\varphi(A,I^+)$ must positively pass the segment of $y=F(x)$ in between $x=0$ and $x=x_0$, and approach the origin in between the positive $x$--axis and $y=F(x)$. Then $\varphi(A,I^-)$ will cannot connect $O$, a contradiction. The claim follows.
	
Finally we prove that the elliptic sector (if exists) at $O$ must be bounded. By contrary, we assume that the  elliptic sector is unbounded.
	  Notice that $x=0$ is the horizontal isoclinic and $y=F(x)$ is  the vertical isoclinic, and that $dy/dt<0$ for $x>0$ and $dy/dt>0$ for $x<0$. These facts imply that any homoclinic orbit inside the elliptic sector at $O$ has its highest point located on the positive $y$--axis,
	  and its leftmost point and rightmost point on $y=F(x)$.
By the existence of	$\varphi(A,I)$  and its properties, it is clear that
 the highest point and  the rightmost point of the elliptic sector are bounded.
By the contrary assumption, it is only possible that the leftmost point of the elliptic sector is unbounded.
We claim that if it is the case, then $F(x)>0$  for $x<0$. Otherwise, there is a largest negative value $x_2$ satisfying $F(x_2)=0$. Then, $\varphi((x_2,0),I^-)$ and $\varphi((x_2,0),I^+)$  must intersect the positive and the negative $y$--axes respectively, denote these intersection points by $C'$ and $B'$. Thus, the elliptic sector lies inside of the region limited by $\widehat{CB}\cup\overline{BB'}\cup\widehat{B'C'}\cup\overline{C'C}$,
implying that the elliptic sector is bounded, a contradiction.
The claim is proved.

	 Let $P:=(0,y_0)$ be a point on the positive $y$--axis for which $\varphi(P,I)$ is a homoclinic orbit inside the elliptic sector, and let $Q:= (x_Q, y_Q)$ be the leftmost point of $\varphi(P,I^-)$, see Figure \ref{P1}(b).

	 Since the leftmost point of the elliptic sector is unbounded, when the homoclinic orbit approaches the outer boundary of the elliptic sector, one has  $x_Q\to-\infty$.
On the one hand, $$E(Q)=\int_{0}^{x_Q}g(x)dx+\frac{y_Q^2}{2}\to+\infty,\quad \text{as}\,\,\, x_Q\to-\infty$$ and $E(P)={y_0^2}/{2} $
is a finite value.
On the other hand,
\[
E(P)-E(Q)=\int_{\widehat{QP}}dE=\int_{\widehat{QP}}-g(x)F(x)dt>0,
\]
where we have used the fact that $dE/dt|_{\eqref{lienard}}=-g(x)F(x)>0$ in the second quadrant.
Again a contradiction happens. Hence, the elliptic sector must be bounded.

It completes the proof of the proposition.

	\section{Applications}
	\label{ap}

This section provides an application of our aforementioned theoretic results to a concrete planar differential system for obtaining its global phase portraits.
	
For the cubic family
		\begin{equation}
	\label{gasull}
	\left\{\begin{aligned}
	\dot{x}&=ax+by,
	\\
	\dot{y}&=cx^3+dx^2y+exy^2+fy^3,
	\end{aligned}
	\right.
	\end{equation}
where $a,\dots,f$ are real parameters, Gasull \cite{Gasull} posed an open question:
\textit{Is 2 its maximum number of limit cycles of system \eqref{gasull}?}
This is the first half of the third problem of his list of 33 open problems in \cite{Gasull}. Here we characterize its global phase portraits under the conditions such that system \eqref{gasull} has no a limit cycle.

If $b=0$, system \eqref{gasull} has an invariant line $x=0$, and its dynamics is simple. The details are omitted.
For $b\neq0$, the linear change of variables
$$
(x,y)\to\left(x,\frac{y-ax}{b}\right)
$$
sends system
\eqref{gasull} to
\begin{equation}
\label{cgu}
\left\{\begin{aligned}
\dot{x}&=y,
\\
\dot{y}&=ay+\nu x^3+\left(d-\frac{2ae}{b}+\frac{3a^2f}{b^2}\right)x^2y+\left(\frac{e}{b}-\frac{3af}{b^2}\right)xy^2+\frac{f}{b^2}y^3,
\end{aligned}
\right.
\end{equation}
where $\nu=bc-ad+a^2e/b-a^3f/b^2$.
When $\nu>0$,  system \eqref{cgu} has a unique equilibrium, which is a saddle.
When $\nu=0$, system \eqref{cgu} has an invariant line $y=0$ and its dynamics is simple.
The next focuses on $\nu<0$.

Related to system \eqref{cgu} is the next system
	\begin{equation}
	\label{cer}
	\left\{\begin{aligned}
	\dot{x}&=y,
	\\
	\dot{y}&=-\omega_0^2x+2\mu\omega_0y\left(1-\beta x^2-\frac{\gamma}{\omega_0} xy-\frac{\delta}{\omega_0^2}y^2\right),
	\end{aligned}
	\right.
	\end{equation}
which is obtained from the equation
	\begin{equation}
	\notag
	\ddot{u}_y(t)-2\mu\omega_0\dot{u}_y(t)\left(1-\beta u_y^2(t)-\frac{\gamma}{\omega_0}\dot{u}_y(t)u_y(t)-\frac{\delta}{\omega_0^2}\dot{u}_y^2(t)\right)+\omega_0^2u_y(t)=0,
	\end{equation}
by Erlicher et al. \cite{ETA} for modelling a hybrid Van der Pol-Rayleigh oscillator with an additional $\gamma$-term
(see	\cite[Section 4.2]{ETA}).

Notice that both system \eqref{cgu} with $\nu\leqslant0$ and  system \eqref{cer} are a sub-family of the system
	\begin{equation}
	\label{app}
	\left\{\begin{aligned}
	\frac{dx}{dt}&=y,
	\\
	\frac{dy}{dt}&=-\lambda x-\mu y-\kappa x^3-ax^2y-bxy^2-cy^3.
	\end{aligned}
	\right.
	\end{equation}
We now apply our main results to system \eqref{app} for obtaining its global  phase portraits in the Poincar\'e  disc when the parameters $\lambda$, $\mu$, $\kappa$, $a$, $c$ are non--negative.
Note that
	when $b=0$, the divergence of the system is non--positive, and so the Dulac criterion can be applied directly.
For $b\neq0$,
	when $\lambda=\kappa=0$, system \eqref{app} has the invariant line $y=0$, and its dynamics is simple. In what follows, our study is under the conditions: $b\ne 0$, and $\lambda+\kappa\ne0$.
	
		When $\kappa>0$, the scaling $(x,y,t)\to(x/\sqrt{\kappa} , y/\sqrt{\kappa} ,t)$
sends		system \eqref{app} to
		\begin{equation}
		\label{61}
		\left\{\begin{aligned}
		\frac{dx}{dt}&=y,
		\\
		\frac{dy}{dt}&=-\lambda x-\mu y-x^3-ax^2y-bxy^2-cy^3,
		\end{aligned}
		\right.
		\end{equation}
with the parameters belonging to the region
			\begin{align*}
			\mathscr{R}_1:=\{(\lambda,\mu,a,b,c)\in\mathbb{R}^5: \lambda\geqslant0,\mu\geqslant0, a\geqslant0, c\geqslant0,b\ne0\}.
			\end{align*}
The global topological phase portraits of system \eqref{61} will be summarized in Theorem \ref{theo} via Figure \ref{qjxt}.

	When $\kappa=0$ and $\lambda>0$, the rescaling $(x,y,t)\to (x,\sqrt{\lambda} y, t/\sqrt{\lambda})$
sends	system \eqref{app} to
	\begin{equation}
	\label{71}
	\left\{\begin{aligned}
	\frac{dx}{dt}&=y,
	\\
	\frac{dy}{dt}&=-x-\mu y-ax^2y-bxy^2-cy^3,
	\end{aligned}
	\right.
	\end{equation}		
with the parameters belonging to the region
			\begin{align*}
			\mathscr{R}_2:=\{(\mu,a,b,c)\in\mathbb{R}^4: \mu\geqslant0, a\geqslant0, c\geqslant0,b\ne0\}.
			\end{align*}
The global topological phase portraits of system \eqref{71} will be summarized in Theorem \ref{thm8} via Figure \ref{qjxt1}.

		\subsection{Global dynamics of system \eqref{61}}
		\subsubsection{Nonexistence of closed orbits of system \eqref{61}}
		\label{sec411}
		In this subsection, we  study the nonexistence of closed orbits of system \eqref{61}.
		\begin{lemma}
			\label{lemma1}
			When $\mu$, $a$, $c$ are non-negative and $\mu^2+a^2+c^2\ne0$, system \eqref{61}  has no closed orbit.
		\end{lemma}
		\begin{proof}
			Set $g(x)=\lambda x+x^3$ and $f(x,y)=\mu+ax^2+bxy+cy^2$. Obviously $g(x)$ is an odd function. One can check that the conditions {\bf{(\romannumeral1)}} and  {\bf{(\romannumeral2)}} of Theorem \ref{thm1} hold for $(x,y)\in \mathbb{R}^2$ and  $\lambda\geqslant0$. Besides, $f(x,y)+f(-x,y)=2\mu+2ax^2+2cy^2\geqslant0$ when
			$\mu$, $a$, $c$ are non-negative. Then the condition {\bf{(\romannumeral3)}} of Theorem \ref{thm1} holds, and so does the condition {\bf{(\romannumeral4)}}, because $\mu^2+a^2+c^2\ne0$. By Theorem \ref{thm1} system \eqref{61} has no a closed orbit around the origin.
		\end{proof}

		\begin{remark}
Lemma \ref{lemma1} indicates that our criterion in Theorem \ref{thm1} is more applicable than the classical Dulac one for system
			\eqref{61}. The divergence of system \eqref{61} is
			$$
			div(X,Y)=-\mu-ax^2-2bxy-3cy^2.
			$$
If $\mu$, $a$ and $c$  are of the same  sign and $b^2-3ac\leqslant0$, the Dulac criterion shows that system
			\eqref{app} has no a closed orbit around $O$. However,
			the Dulac criterion is invalid for $b^2-3ac>0$.  While, the application of the criterion in Theorem \ref{thm1} is independent on $b$.
		\end{remark}	

		\subsubsection{Equilibria  of system \eqref{61}}
		\label{sec412}
		Applying Lemma  \ref{lem1} and Theorem \ref{lem2} to system \eqref{61} directly provides
		the qualitative properties of the unique equilibrium $O$, which are shown in Table \ref{qpeq}.
		
		\begin{table}[htp]
			\renewcommand\arraystretch{1.5}
			\setlength{\tabcolsep}{5mm}{
				\caption{	\label{qpeq} The qualitative property of $O$ of system \eqref{app}.}
				\begin{tabular}{c|c|c|c}
					\hline
					\multicolumn{3}{c|}
					{Cases of parameters} & Type of $O$
					\\
					\hline
					\multirow{2}{*}
					{\begin{tabular}[c]{@{}c@{}} $\mu=0$\end{tabular}} & \multicolumn{2}{c|}{$a=c=0$} & center
					\\
					\cline{2-4}
					& \multicolumn{2}{c|}{ $a^2+c^2\ne0$} &  stable focus
					\\
					\hline
					\multirow{4}{*}{\begin{tabular}[c]{@{}c@{}}
							$\mu>0$\end{tabular}} & \multicolumn{2}{c|}{$\lambda=0$} & stable  improper  node
					\\
					\cline{2-4}
					& \multirow{3}{*}{\begin{tabular}[c]{@{}c@{}} $\lambda> 0$\end{tabular}} & $\mu^2-4\lambda>0$ & stable node
					\\
					\cline{3-4}
					&  & $\mu^2-4\lambda=0$ & stable improper node
					\\
					\cline{3-4}
					&  & $\mu^2-4\lambda<0$ &  stable focus
					\\
					\hline
				\end{tabular}}
			\end{table}
			
			On the qualitative properties of the equilibria  at infinity,
			the proof is not difficult but is too long to give here, see Appendix C.

		\subsubsection{Global phase portraits in the Poincar\'e disc of system \eqref{61}}
		By the nonexistence of closed orbits, qualitative properties of equilibria (including at infinity) of  system \eqref{61}, 
		we can obtain 		all global phase portraits in the Poincar\'e disc in the following theorem.
	\begin{theorem}
		\label{theo}
		All global phase portraits in the Poincar\'e disc
		are shown in
		{\rm Figure \ref{qjxt}} for system \eqref{61},
		where
		\begin{align*}
				S_1&=\{(\lambda,\mu,a,b,c)\in\mathscr{R}_1: a=\mu=c=0, b<0\},\\
				S_2&=\{(\lambda,\mu,a,b,c)\in\mathscr{R}_1: a=\mu=c=0, b>0\},\\
			S_3&=\{(\lambda,\mu,a,b,c)\in\mathscr{R}_1: c=0, a^2+\mu^2\ne0, b>a^2/4\},\\
			S_4&=\{(\lambda,\mu,a,b,c)\in\mathscr{R}_1:  c=0, 0<b<a^2/4\},\\
			S_5&=\{(\lambda,\mu,a,b,c)\in\mathscr{R}_1: c=0, a^2+\mu^2\ne0, b<0\},\\
				S_6&=\{(\lambda,\mu,a,b,c)\in\mathscr{R}_1: c=0, b=a^2/4, u_0^2+
				\mu u_0+\lambda<0\},\\
						S_7&=\{(\lambda,\mu,a,b,c)\in\mathscr{R}_1: c=0, b=a^2/4, u_0^2+
				\mu u_0+\lambda=0\},\\
						S_8&=\{(\lambda,\mu,a,b,c)\in\mathscr{R}_1: c=0, b=a^2/4, u_0^2+
				\mu u_0+\lambda>0\},\\
						S_9&=\{(\lambda,\mu,a,b,c)\in\mathscr{R}_1: c>0, -\sqrt{3ac}\leqslant b\leqslant \sqrt{3ac} \} \\
					&\ \ \cup \{(\lambda,\mu,a,b,c)\in\mathscr{R}_1: c>0, b<-\sqrt{3ac}, \Phi(\varrho_2)>0\}\\
					&\ \ \cup \{(\lambda,\mu,a,b,c)\in\mathscr{R}_1: c>0, b>\sqrt{3ac},  \Phi(\varrho_1)<0\}\\
					&\ \ \cup \{(\lambda,\mu,a,b,c)\in\mathscr{R}_1: c>0, b>\sqrt{3ac}, \Phi(\varrho_1)>0, \Phi(\varrho_2)>0\},\\
	S_{10}&=\{(\lambda,\mu,a,b,c)\in\mathscr{R}_1: c>0, b>\sqrt{3ac}, \Phi(\varrho_1)=0,
 \varrho_1^2+\mu \varrho_1+\lambda<0\},\\					
	S_{11}&=\{(\lambda,\mu,a,b,c)\in\mathscr{R}_1: c>0, b>\sqrt{3ac}, \Phi(\varrho_1)=0,  \varrho_1^2+\mu \varrho_1+\lambda=0\},\\					
	S_{12}&=\{(\lambda,\mu,a,b,c)\in\mathscr{R}_1: c>0, b>\sqrt{3ac}, \Phi(\varrho_1)=0,  \varrho_1^2+\mu \varrho_1+\lambda>0\},\\	
	S_{13}&=\{(\lambda,\mu,a,b,c)\in\mathscr{R}_1: c>0, b<-\sqrt{3ac}, \Phi(\varrho_2)=0\},\\			
			S_{14}&=\{(\lambda,\mu,a,b,c)\in\mathscr{R}_1: c>0, b>\sqrt{3ac}, \Phi(\varrho_1)>0,  \Phi(\varrho_2)=0,  \varrho_2^2+\mu \varrho_2+\lambda<0\},\\		
		S_{15}&=\{(\lambda,\mu,a,b,c)\in\mathscr{R}_1: c>0, b>\sqrt{3ac}, \Phi(\varrho_1)>0,  \Phi(\varrho_2)=0,  \varrho_2^2+\mu \varrho_2+\lambda=0\},\\				
			S_{16}&=\{(\lambda,\mu,a,b,c)\in\mathscr{R}_1: c>0, b>\sqrt{3ac}, \Phi(\varrho_1)>0, \Phi(\varrho_2)=0,  \varrho_2^2+\mu \varrho_2+\lambda>0\},\\
			S_{17}&=\{(\lambda,\mu,a,b,c)\in\mathscr{R}_1: c>0, b>\sqrt{3ac}, \Phi(\varrho_1)>0, \Phi(\varrho_2)<0\},\\				
				S_{18}&=\{(\lambda,\mu,a,b,c)\in\mathscr{R}_1: c>0, b<-\sqrt{3ac},\Phi(\varrho_2)<0\},
			\end{align*}
$\Phi(u):=cu^3+bu^2+au+1$, $u_0=-a/2b$, and $\varrho_1=(-b-\sqrt{b^2-3ac})/3c$,  $\varrho_2=(-b+\sqrt{b^2-3ac})/3c$ for $c>0$.
\end{theorem}
		\begin{figure}[htp]
			\centering
			\subfigure[{ $	S_{1}$}]{
				\includegraphics[width=0.2\textwidth]{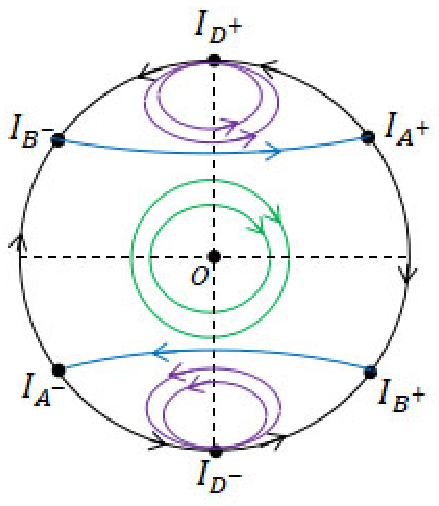}}
			\quad
			\subfigure[{ 	$	S_{2}$ }]{
				\includegraphics[width=0.2\textwidth]{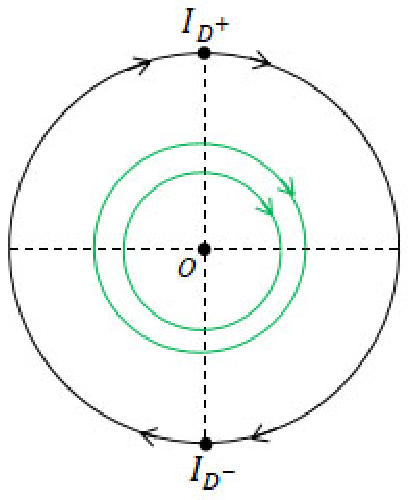}}
			\quad
			\subfigure[{ $	S_{3}$ }]{
				\includegraphics[width=0.2\textwidth]{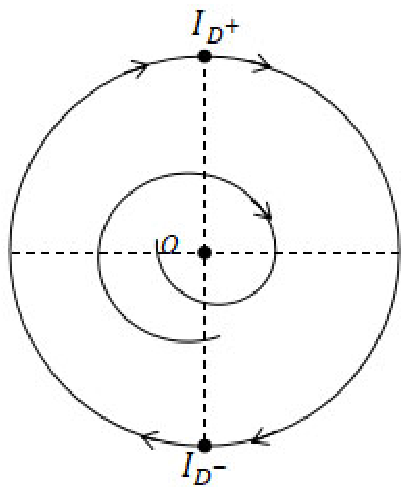}}
			\quad
			\subfigure[{ $	S_{4}$ }]{
				\includegraphics[width=0.2\textwidth]{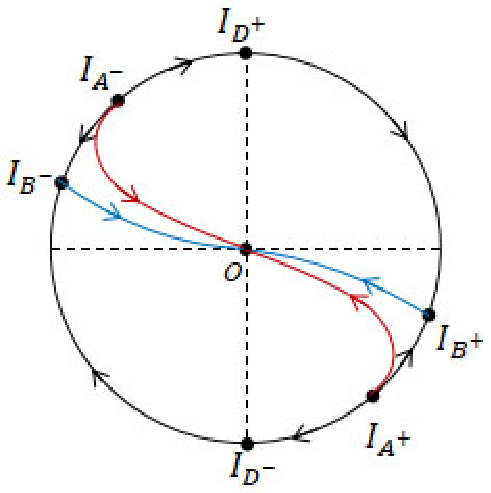}}
			\quad
			\subfigure[{ $	S_{5}$ }]{
				\includegraphics[width=0.2\textwidth]{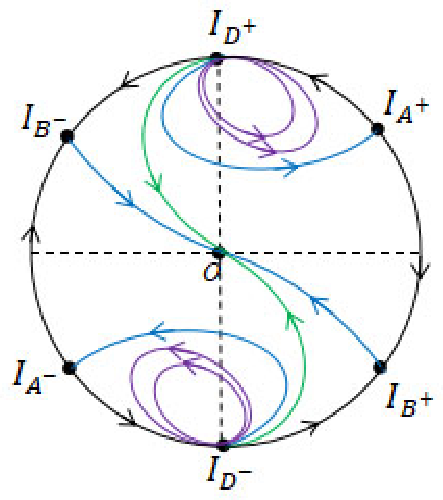}}
			\quad
			\subfigure[{ $	S_{6}$ }]{
				\includegraphics[width=0.2\textwidth]{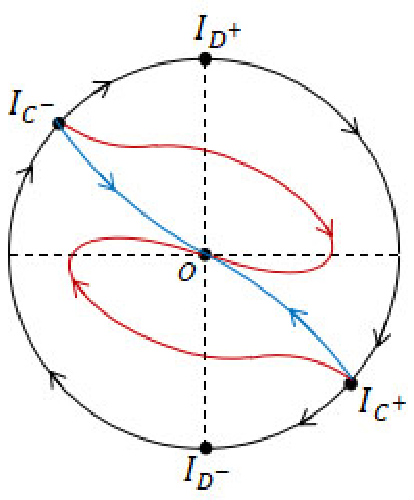}}
			\quad
			\subfigure[{ $	S_{7}$ }]{
				\includegraphics[width=0.2\textwidth]{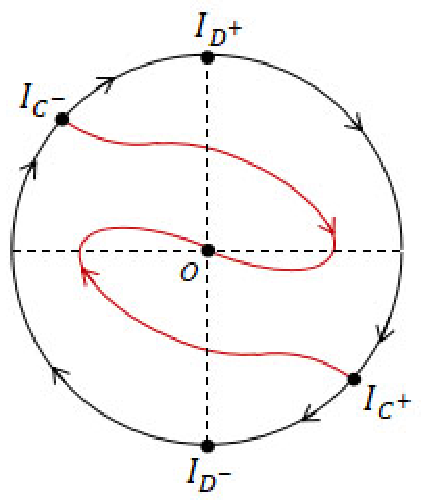}}
			\quad
			\subfigure[{ $	S_{8}$ }]{
				\includegraphics[width=0.2\textwidth]{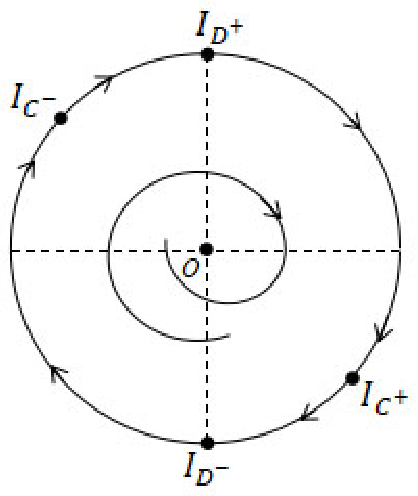}}
			\quad
			\subfigure[{ $	S_{9}$ }]{
				\includegraphics[width=0.2\textwidth]{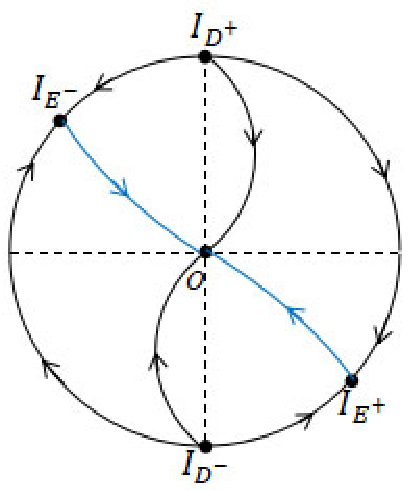}}
			\quad	
			\subfigure[{$	S_{10}$}]{
				\includegraphics[width=0.2\textwidth]{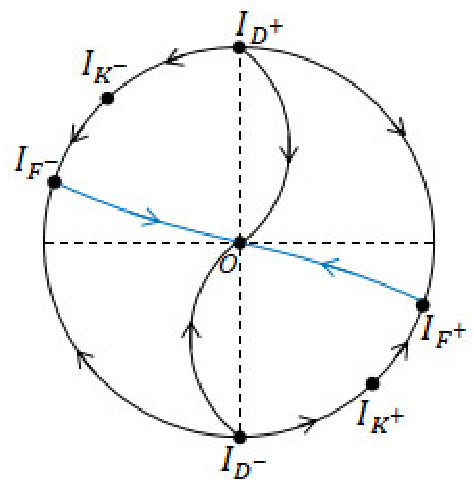}}
			\quad		
			\subfigure[{$	S_{11}$}]{
				\includegraphics[width=0.2\textwidth]{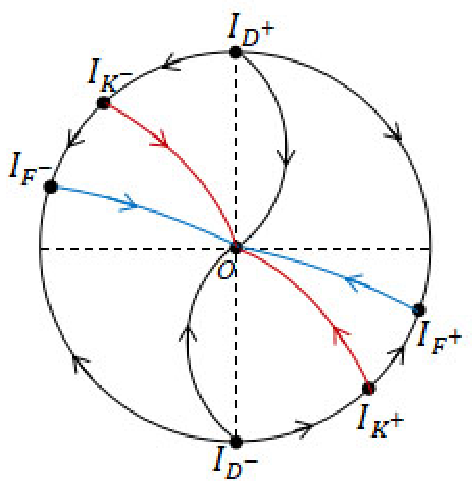}}
			\quad	
			\subfigure[{$	S_{12}$}]{
				\includegraphics[width=0.2\textwidth]{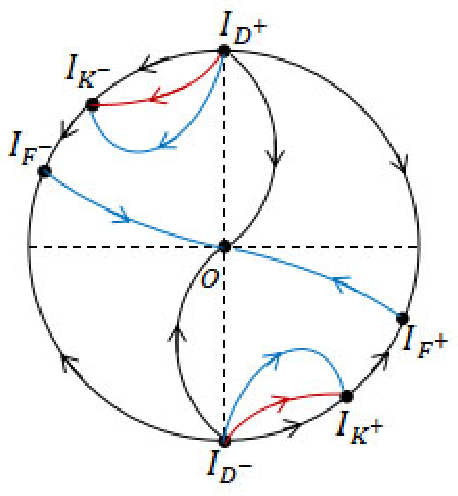}}
			\quad
			\subfigure[{$	S_{13}$}]{
				\includegraphics[width=0.2\textwidth]{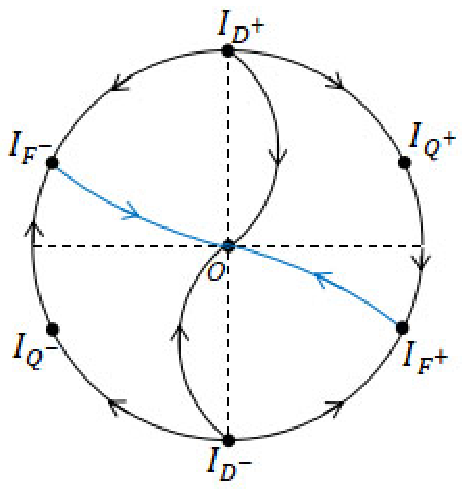}}
			\quad
			\subfigure[{$	S_{14}$}]{
				\includegraphics[width=0.2\textwidth]{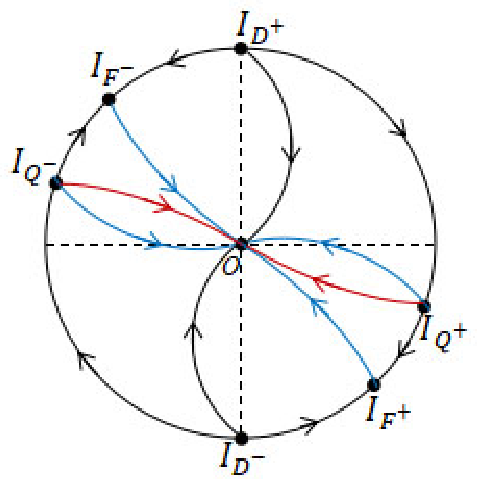}}
			\quad	
			\subfigure[{$	S_{15}$}]{
				\includegraphics[width=0.2\textwidth]{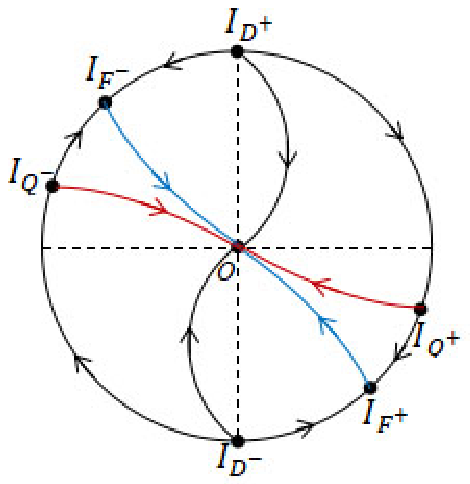}}
			\quad
			\subfigure[{$	S_{16}$}]{
				\includegraphics[width=0.2\textwidth]{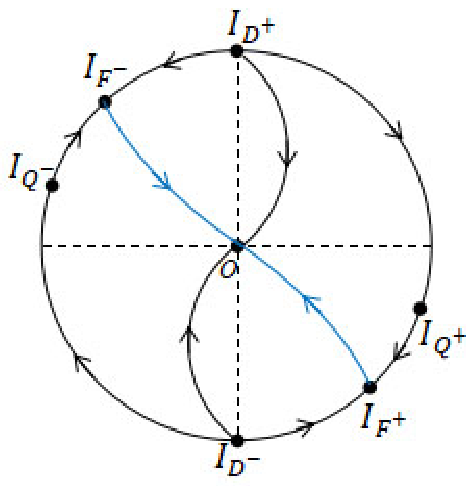}}
			\quad
			\subfigure[{$	S_{17}$}]{
				\includegraphics[width=0.2\textwidth]{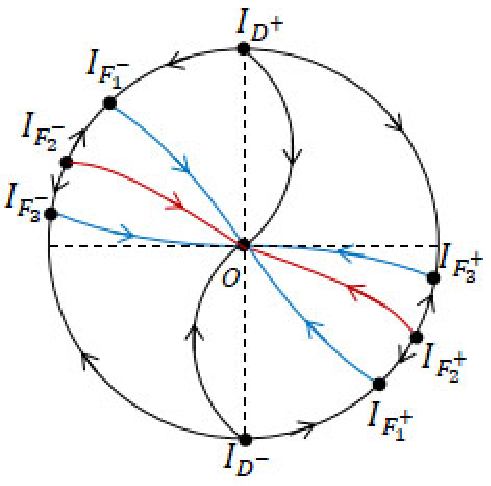}}
			\quad
			\subfigure[{$S_{18}$ }]{
				\includegraphics[width=0.2\textwidth]{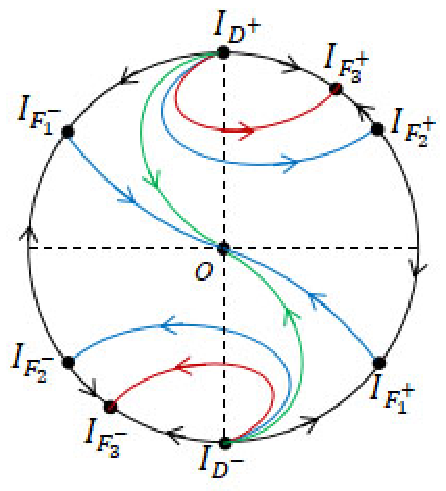}}
			\quad
			\caption{ Global phase portraits of system \eqref{61}. }\label{qjxt}
		\end{figure}
		\begin{proof}
			When $\mu=a=c=0$, the unique equilibrium $O$ is a center for system \eqref{61} by Table \ref{qpeq}. When $\mu$, $a$, $c$ are non-negative and $\mu^2+a^2+c^2\ne0$, systems \eqref{61}  has no closed orbits and $O$ is a sink by Lemma \ref{lemma1} and Table \ref{qpeq}.
			
		Combining Lemmas \ref{lemy1}--\ref{lemx} of Appendix C, we can obtain all  global phase portraits of systems \eqref{61}.
		It is worth to notice that the aquirment of Figure \ref{qjxt}(e, l, r)  need  more  derivation. In fact,
		consider the region $S_{5}$. Incorporated with  Figure \ref{tu11}(b) and  Figure \ref{tu19}(b) of Appendix C, it is easy to obtain that {the $\omega$--limit sets of $\theta$ is probably $O$, $I_{D^+}$, $I_{A^+}$, which can be seen} 
		 in Figure \ref{tu21}, where $\theta$ stands for the orbit leaving $I_{B^-}$.
		
		\begin{figure}[htp]
			\centering
			\subfigure[{ $\Omega_\theta=O$  }]{
				\includegraphics[width=0.3\textwidth]{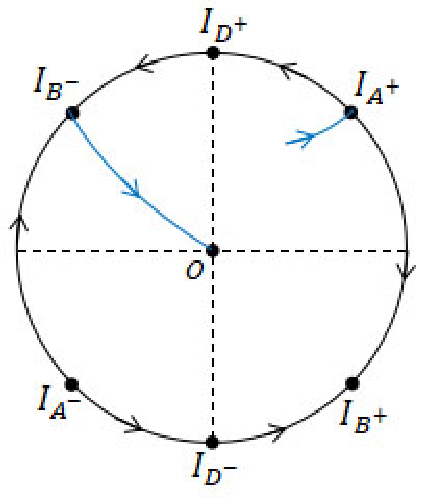}}
			\quad
			\subfigure[{$\Omega_\theta=I_{D^+}$ }]{
				\includegraphics[width=0.3\textwidth]{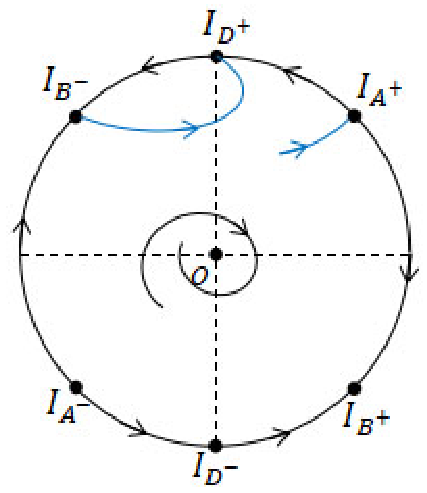}}
			\quad
			\subfigure[{ $\Omega_\theta=I_{A^+}$}]{
				\includegraphics[width=0.3\textwidth]{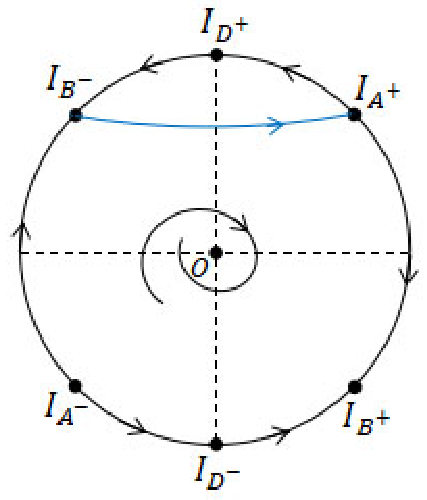}}
			\caption{ The possibilities of the $\omega$--limit sets of $\theta$.
			}\label{tu21}
		\end{figure}
		
		We claim that the $\omega$--limit sets of $\theta$ is $O$.
		In fact, if the $\omega$--limit sets of $\theta$ is $I_{D^+}$.
		 Because of the stability of $O$ and a orbit approaching $I_{A^+}$,  $\widehat{I_{B^-}I_{D^-}I_{D^+}}$ and $O$ can be used as the inner and  outer boundaries. Thus, there is a limit cycle surrounding $O$ which contradicts the nonexistence of closed orbits.
		 If the
 $\omega$--limit sets of $\theta$ is $I_{A^+}$.
		 The fixed $b$ and $a$ (resp. $\mu$) make it easy to check that system \eqref{61} is a generalized rotated vector field on $\mu$ (resp. $a$). When $\mu$ (resp. $a$) decreases,  in Figure \ref{tu21}(b) will happen which is a conflict. Thus, the global portrait of  system \eqref{61} for $c=0$, $b<0$ and $a^2+\mu^2\ne0$ is shown in Figure \ref{qjxt}(e).
		
		Next, we can give similarly the rest of global phase portraits in the Poincar\'e disc of system \eqref{61} .
		\end{proof}
		
		\subsection{Global dynamics of system \eqref{71} }
		
		As proven in subsections \ref{sec411} and \ref{sec412}, we can obtain similarly
		that  system \eqref{71}  has no closed orbits when $\mu$, $a$, $c$ are non-negative and $\mu^2+a^2+c^2\ne0$, and the qualitative properties of
	 the unique equilibrium $O$ is also as shown in   Table \ref{qpeq}.
			Similarly,	on the qualitative properties of the equilibria  at infinity of system \eqref{71},
	the proof is also not difficult but is too long to give here, see Appendix D.
	Then, we obtain all global phase portraits in the Poincar\'e disc in  the following theorem.

		\begin{theorem}
			\label{thm8}
			All global phase portraits in the Poincar\'e disc
			are shown in
		 {\rm Figure \ref{qjxt1}} for system \eqref{71},
			where
			\begin{align*}
				G_{1}&=\{(\mu,a,b,c)\in\mathscr{R}_2:\mu=a=c=0, b>0\},\\
				G_{2}&=\{(\mu,a,b,c)\in\mathscr{R}_2:\mu=a=c=0, b<0\},\\
				G_{3}&=\{(\mu,a,b,c)\in\mathscr{R}_2:\mu>0, a=c=0, b<0\},\\
				G_{4}&=\{(\mu,a,b,c)\in\mathscr{R}_2:\mu>0, a=c=0, b>0\},\\
				G_{5}&=\{(\mu,a,b,c)\in\mathscr{R}_2:a>0, b>0, c=0\},\\
				G_{6}&=\{(\mu,a,b,c)\in\mathscr{R}_2: a>0, b<0, c=0\},\\
				G_{7}&=\{(\mu,a,b,c)\in\mathscr{R}_2: a=0, b>0, c>0\},\\
			G_{8}&=\{(\mu,a,b,c)\in\mathscr{R}_2: a=0, b<0, c>0\},\\	
			G_{9}&=\{(\mu,a,b,c)\in\mathscr{R}_2: a>0, c>0, -2\sqrt{ac}<b<2\sqrt{ac}\},\\
				G_{10}&=\{(\mu,a,b,c)\in\mathscr{R}_2: a>0, c>0, b>2\sqrt{ac}\},\\
				G_{11}&=\{(\mu,a,b,c)\in\mathscr{R}_2: a>0,  c>0, b<-2\sqrt{ac}\},\\
				G_{12}&=\{(\mu,a,b,c)\in\mathscr{R}_2: a>0, c>0, b=2\sqrt{ac}, a-\mu\sqrt{ac}+c<0\},\\
			G_{13}&=\{(\mu,a,b,c)\in\mathscr{R}_2: a>0, c>0,  b=-2\sqrt{ac}\},\\
			G_{14}&=\{(\mu,a,b,c)\in\mathscr{R}_2: a>0, c>0,  b=2\sqrt{ac}, a-\mu\sqrt{ac}+c=0\},\\
				G_{15}&=\{(\mu,a,b,c)\in\mathscr{R}_2: a>0, c>0, b=2\sqrt{ac}, a-\mu\sqrt{ac}+c>0\},\\
			\end{align*}
		\end{theorem}
		\begin{figure}[htp]
		\centering
		\subfigure[{ $	G_{1}$}]{
			\includegraphics[width=0.2\textwidth]{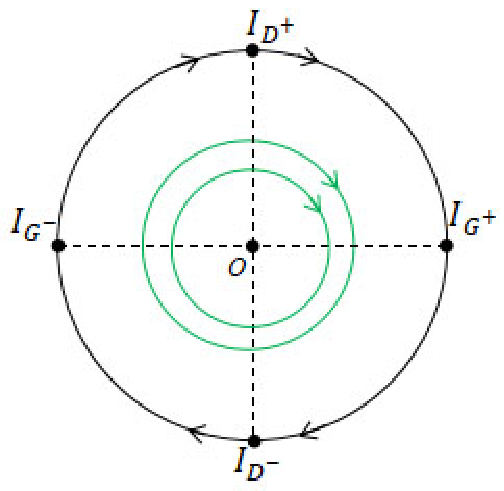}}
		\quad
		\subfigure[{ 	$	G_{2}$ }]{
			\includegraphics[width=0.2\textwidth]{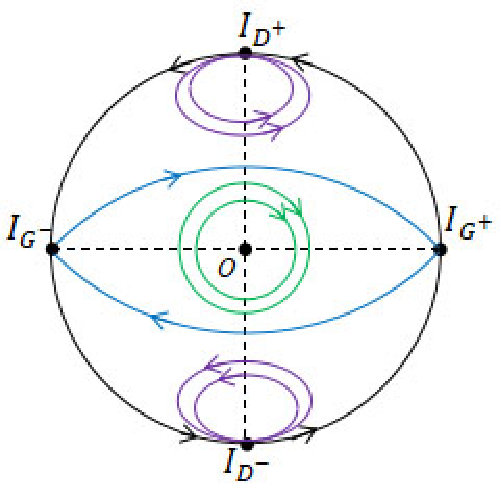}}
		\quad
		\subfigure[{ $G_{3}$ }]{
			\includegraphics[width=0.2\textwidth]{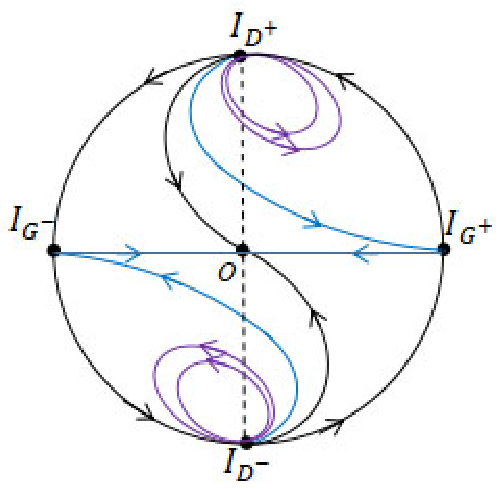}}
		\quad
		\subfigure[{ $G_{4}$ }]{
			\includegraphics[width=0.2\textwidth]{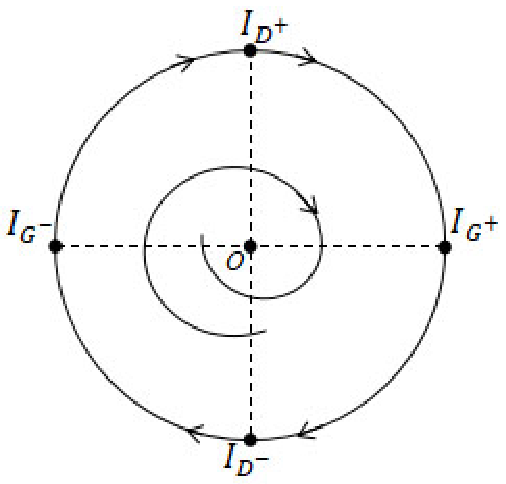}}
		\quad
		\subfigure[{ $G_{5}$ }]{
			\includegraphics[width=0.2\textwidth]{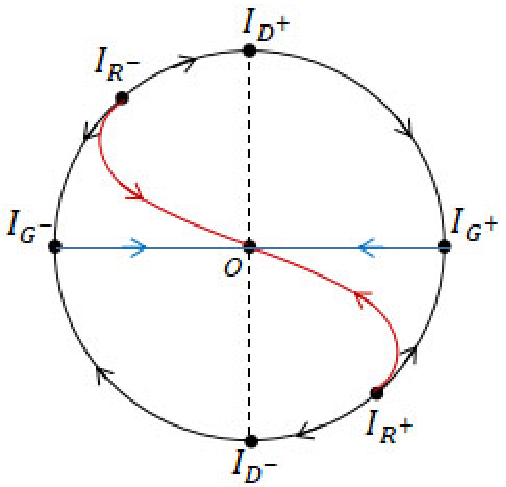}}
		\quad
		\subfigure[{ $G_{6}$ }]{
			\includegraphics[width=0.2\textwidth]{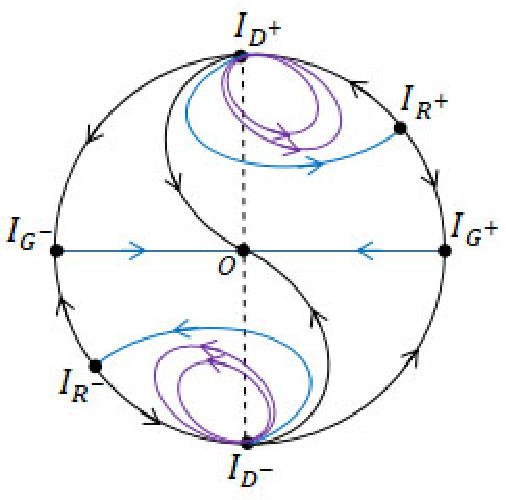}}
		\quad
		\subfigure[{ $G_{7}$ }]{
			\includegraphics[width=0.2\textwidth]{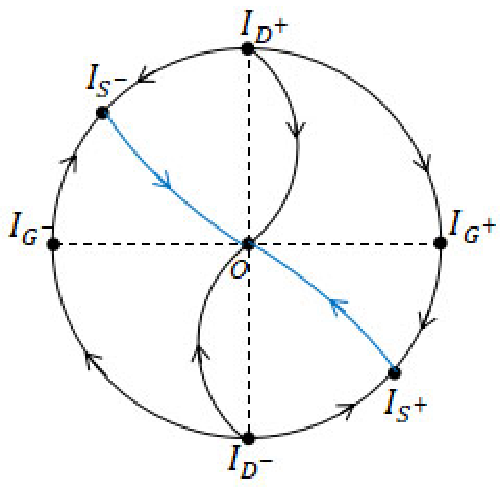}}
		\quad
		\subfigure[{ $G_{8}$ }]{
			\includegraphics[width=0.2\textwidth]{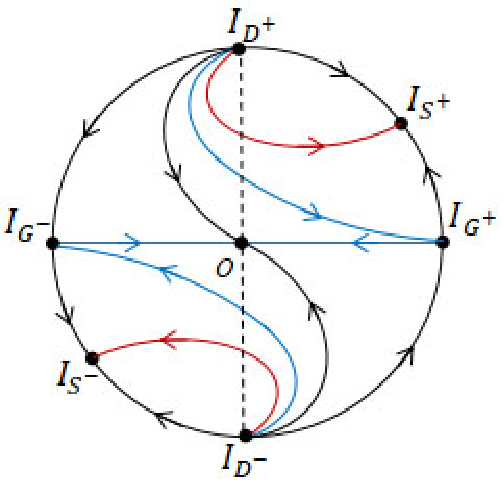}}
		\quad
		\subfigure[{ $G_{9}$ }]{
			\includegraphics[width=0.2\textwidth]{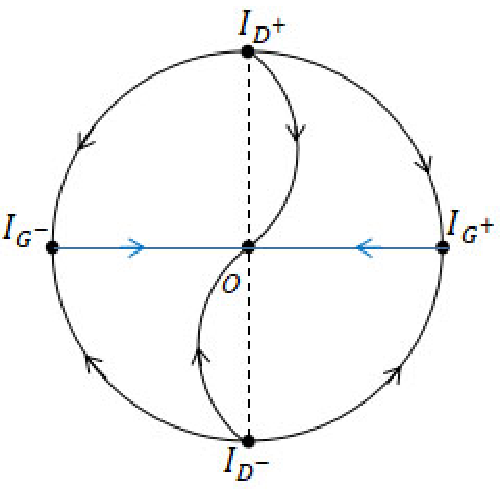}}
		\quad	
		\subfigure[{$G_{10}$}]{
			\includegraphics[width=0.2\textwidth]{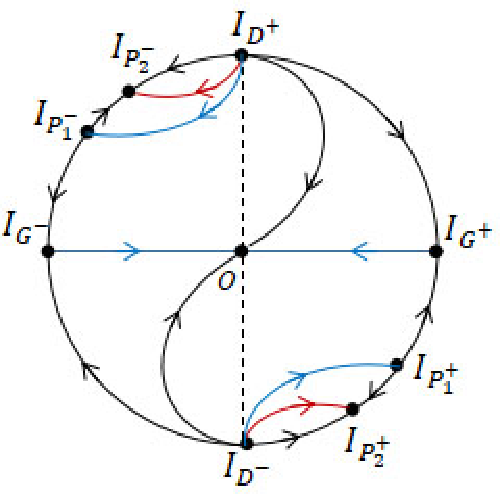}}
		\quad		
		\subfigure[{$G_{11}$}]{
			\includegraphics[width=0.2\textwidth]{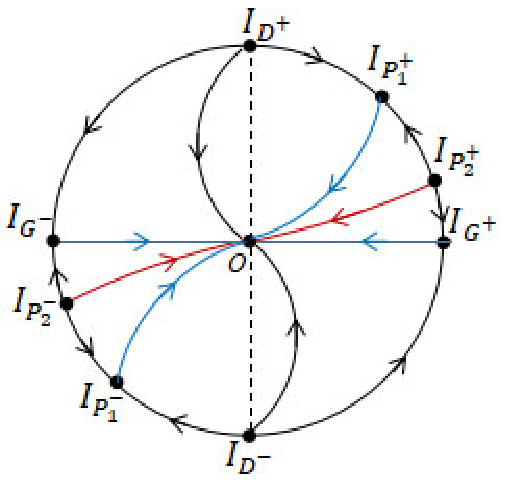}}
		\quad	
		\subfigure[{$G_{12}$}]{
			\includegraphics[width=0.2\textwidth]{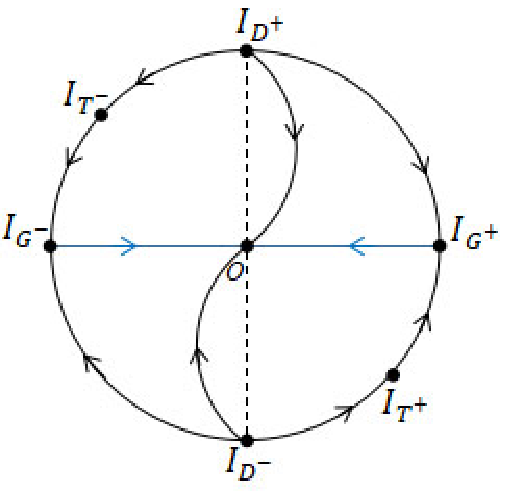}}
		\quad
		\subfigure[{$G_{13}$}]{
			\includegraphics[width=0.2\textwidth]{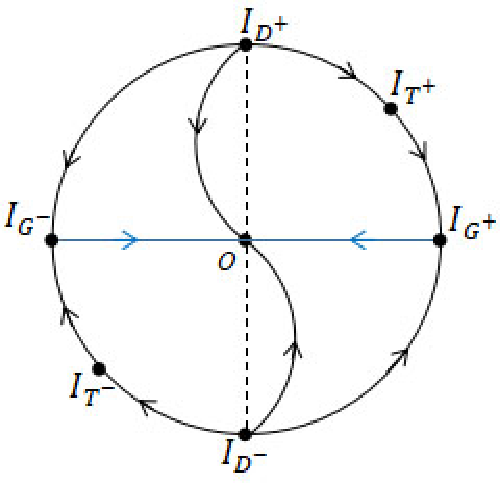}}
		\quad
		\subfigure[{$G_{14}$}]{
			\includegraphics[width=0.2\textwidth]{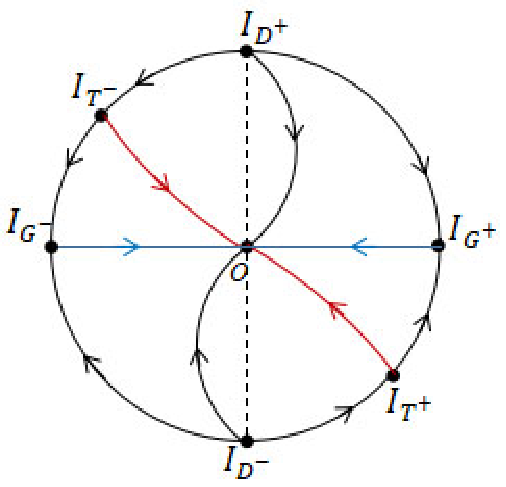}}
		\quad	
		\subfigure[{$G_{15}$}]{
			\includegraphics[width=0.2\textwidth]{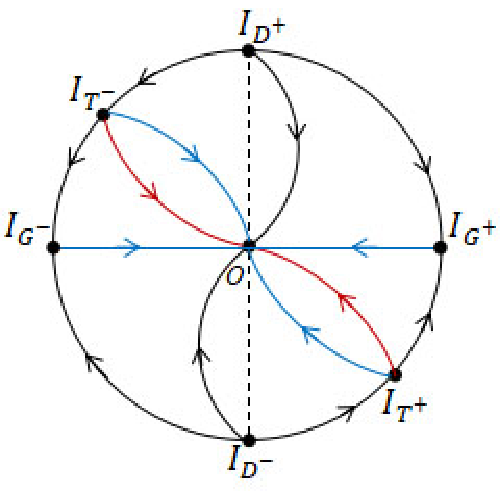}}
		\caption{ Global phase portraits of system \eqref{71}. }\label{qjxt1}
	\end{figure}

	\section*{Appendix A}

We will recall here the classical results on nonexistence of closed orbits of planar dynamical systems, which are used for comparing with our criterions.

	Consider the following planar dynamical system
	\begin{equation}\label{A1}
	\frac{dx}{dt}=X(x,y), \quad
	\frac{dy}{dt}=Y(x,y),
	\end{equation}
	where $X(x,y)$, $Y(x,y)$ are defined on $\mathbb{R}^2$.
	We first state the Poincar\'e's method of tangential curves.

	\begin{theorem}
		\cite[Theorem 1.6 of Chapter 4]{ZDHD} Let $F(x, y) = C$
		be a family of curves, where $F(x , y) \in C^1(G)$. Suppose that
		$$
		\frac{dF}{dt}=X\frac{\partial F}{\partial x}+Y\frac{\partial F}{\partial y}
		$$
		has a fixed sign on $G$ $($i.e., $\geqslant0$ or $\leqslant0)$, and the equality
		$$X\frac{\partial F}{\partial x}+Y\frac{\partial F}{\partial y}=0$$
		cannot be satisfied on an entire orbit of \eqref{A1}. Then system \eqref{A1}
		has no closed orbits in $G$.
	\end{theorem}

Secondly, we recall the Bendixson-Dulac criterion.
	\begin{theorem}
		\cite[Theorem 1.7 of Chapter 4]{ZDHD}
		Suppose that in the simply connected
		region $G$, the vector field $(X(x,y), Y(x,y))$  associated to system
		\eqref{A1} is   $C^1(G)$. Further, there is a function $B(x, y)\in C^1(G)$ such that
		$$\frac{\partial (BX)}{\partial x}+\frac{\partial (BY)}{\partial y}$$ is of the fixed sign, and is never identically zero
		in any subregion. Then system \eqref{A1} has no a closed orbit in $G$.
	\end{theorem}
	
	Thirdly, we recall the result by Lins et al   \cite{LMP} on nonexistence of closed orbits of the classical Li\'enard system.
	
	\begin{theorem}
		\cite[Proposition 1]{LMP}
		Consider the Li\'enard differential system
		\begin{equation}\label{A2}
		\frac{dx}{dt}=y-F(x), \quad
		\frac{dy}{dt}=-x,
		\end{equation}
		where $F(x)=a_dx^d+\cdots+a_1x$. Set $F(x) = c(x) + e(x)$, with $c(x)$ even and
		$e(x)$ odd functions.  If $0$ is the unique zero of $e(x)$,
		then system \eqref{A2} has no closed orbits around the origin.
		\label{LMP}
	\end{theorem}

	Fourthly, we recall the result by Dumortier and Rousseau \cite{DR} on nonexistence of closed orbits of the generalized Li\'enard system.
	\begin{theorem}
		\cite[Proposition 2.3]{DR}
			Consider the generalized Li\'enard differential system
		\begin{equation}\label{A3}
		\frac{dx}{dt}=y-F(x), \quad
		\frac{dy}{dt}=-g(x),
		\end{equation}
 with $F$ of class $C^2$ and $g$ of class
		$C^1$ on  $( \alpha,\beta)$ $(\alpha$, $\beta$ can be $\pm\infty)$, and satisfying
\begin{itemize}
			\item[(\romannumeral1)]$f(x) = F'(x)$ has a unique zero $x_0< 0$; $f(x) < 0 \ (\text{resp.} > 0)$ as $\alpha<x<x_0\ (\text{resp.} \,\, x_0<x<\beta)$;
			\item[(\romannumeral2)] $F(0)=0$, $F(\xi_0)=0$ for $\alpha<\xi_0<x_0$;
			\item[(\romannumeral3)] $xg(x)>0$ for  $x \in (\alpha,\beta)$ and $x\ne0$.
		\end{itemize}
		Then system \eqref{A3}  has no a limit cycle in the strip $\xi_0<x<\beta$.
	\end{theorem}

	Fifthly, we recall the result by Sugie \cite{Sugie} on nonexistence of closed orbits of the generalized Li\'enard system.

	\begin{theorem}\cite[Theorems 3.1 and 3.2]{Sugie}
		Consider system \eqref{A3}, with  $F(x)$ and $g(x)$ being continuous functions in $\mathbb{R}$ and satisfying
		$F(0) = 0$ and $xg(x) > 0$ for $x \ne 0$. Assume that the initial value problem for system \eqref{A3} has always a unique solution. Let $M^+=\int_{0}^{+\infty} g(x)dx$, $M^-=\int_{0}^{-\infty} g(x)dx$ and $M=\min\{M^+,M^-\}$. Define
		$$w = G(x) := \int_0^x\left|g(s)\right|ds.$$
		\begin{itemize}
\item[{\rm(a)}] If
		$$F\left(G^{-1}(-w)\right)\ne F\left(G^{-1}(w)\right), \quad 0<w<M, $$
		then  system \eqref{A3} has no periodic solutions except for the origin.
		
\item[{\rm(b)}] Define $$H(w)=F\left(G^{-1}(-w)\right)- F\left(G^{-1}(w)\right)\quad\text{for} \quad0\leqslant w<M.$$Suppose that
		$$H(w)\geqslant0 \quad\text{or} \quad H(w)\leqslant0 \quad\text{for} \quad0\leqslant w<M,$$
		and that there exists a sequence $\{w_n\}$ such that
		$$w_n\to0^+ \quad	\text{as}\quad n\to+\infty \quad\text{ and}\quad  H(w_n) \ne 0.$$
		Then  system \eqref{A3} has no periodic solutions except for the origin.
\end{itemize}
	\end{theorem}

	Sixthly, we recall the result from   \cite{CC} on nonexistence of closed orbits of the Li\'enard system \eqref{A3}.

	\begin{theorem}
		\cite[Proposition 2.1]{CC}
		Consider system \eqref{A3} with $g$ continuous and $F$ smooth on $(a_1,b_1 )$ for
		some given $a_1<0<b_1$, and $F(0) = 0$. Assume that
 \begin{itemize}
 \item $g(x)$ has n +1 zeros $0$,  $x _1$, $\cdots$, $x_n$ in $(a_1,b_1 )$ such that $xg(x) > 0$
		for all $x\in(a_1,b_1 )\setminus\{0,x_1,\cdots, x_n\}$.
\end{itemize}
Set $z(x):= \int_0^xg(s)ds$, $z_1:=z(b_1)>0$ and $z_2:=z(a_1)>0$ and $x_1(z)$, and let $x_2(z)$ be the branches of the inverse of $z(x)$ for $x \geqslant 0$ and $x \leqslant 0$, respectively. Set
\[
 F_1(z):=F(x_1(z)),\qquad F_2(z):=F(x_2(z)).
 \]
	If
		\begin{itemize}
			\item[(\romannumeral1)]
$F_ 1(z)\not\equiv F_ 2(z)$ for all $0 < z \ll 1$, and
			\item[(\romannumeral2)]
either $F_ 1(z)\geqslant F_ 2(z)$  or $F_ 1(z)\leqslant F_ 2(z)$ for all $z\in(0,\min\{z_1,z_2\})$,
		\end{itemize}	
	then system \eqref{A3} has no closed orbits in the strip $a_1<x<b_1$.
	\end{theorem}

	Seventhly, we recall the result  by Chen and Tang \cite{CT} on nonexistence of closed orbits of the generalized Li\'enard system \eqref{A3}.

	\begin{theorem}\cite[Theorem 2.1]{CT}
		Consider system \eqref{A3}
		with $F$ and $g$ of class $C^1$ in $(\beta_1,\beta_2)$, where $\beta_1 < 0$ and $\beta_2>0$ $(\beta_1$, $\beta_2$ could be $\pm\infty)$. Assume that $F(x)$ and $g(x)$  satisfy the conditions: 
		\begin{itemize}
			\item[(\romannumeral1)]$xg(x)>0$ for $x \in(\beta_1, 0)\cup(0,\beta_2)$;
			\item[(\romannumeral2)] $F (x)$ has at most three zeros $x_1$, $x_2$, $0 \in(\beta_1,\beta_2)$ with $x_1 < x_2 \leqslant 0$, and $F (x)>0\ (\text{resp.}
		\,\,	< 0)$ for $x \in(x_1, x_2) \cup (0,+\infty)\ (\text{resp.} \,\,x \in(\beta_1, x_1) \cup (x_2,0))$;
			\item[(\romannumeral3)] $f (x) = F'(x)$ has a unique zero $\xi_0$ in $(x_1, x_2)$, and $f (x) < 0 \ (\text{resp.}\,\, > 0)$ for $x \in (\xi_0,x_2) \
			(\text{resp.}\,\, x\in(x_1,\xi_0) \cup (0,\beta_2))$;
			\item[(\romannumeral4)] the simultaneous equations
			$$F(z_1) = F(z_2) \quad \text{and} \quad \frac{g(z_1)}{f(z_1)} = \frac{g(z_2)}{f(z_2)} $$
			have no common solutions in $(\beta_1,\beta_2)$ satisfying $x_1<z_1<x_2<0<z_2$.
		\end{itemize}
		Then system \eqref{A3} has no closed orbits in the strip $\beta_1<x<\beta_2$.
	\end{theorem}

	\section*{Appendix B}

In this part we recall Theorems 7.1 and 7.2 of \cite [Chapter 2]{ZDHD} for reader's convenience, which are frequently used to characterize local topological structure of the equilibrium having at least one eigenvalue vanishing and its linearization nonvanishing.

	Consider the planar differential system
	\begin{equation}\label{T7.1}
	\frac{d x}{d t}=P_2(x,y), \quad \frac{d y}{d t}=y+Q_2(x,y).
	\end{equation}

	\begin{theorem}\label{thm7.1}\cite[Theorem 7.1 of Chapter 2]{ZDHD}
		Suppose that $O=(0,0)$ is an isolated equilibrium of system \eqref{T7.1} and that $P_{2}$ and $Q_{2}$ are analytic functions in a small neighborhood $S_{\delta}(O)$ of $O$ without constant and linear terms. Let $y=\phi(x)$, $|x|<\delta$,  be the unique analytic solution of the equation
		$$
		y+Q_{2}(x, y) \equiv 0, \quad \mbox{in} \ \ \ S_{\delta}(O),
		$$
		and set
		$$
		\psi(x)=P_{2}(x, \phi(x))=a_{m} x^{m}+o(x^m), \quad |x|<\delta,
		$$
		with $a_{m} \neq 0, m \geqslant 2 .$ Then the following properties are satisfied.
		
		\begin{itemize}
			\item[(i)]
			If $m$ is odd and $a_{m}>0$, then $O $ is an unstable node.
			\item[(ii)]If $m$ is odd and $a_{m}<0$, then $O $ is a saddle  with its four separatrices tending to $O(0,0)$ along the directions $\theta=0, \pi/2, \pi$ and $3 \pi/2$, respectively.
			\item[(iii)]If $m$ is even, then $O $ is a saddle--node, and $S_{\delta}(O)$ is divided by two separatrices,  tangent to respectively the positive and negative $y$--axes at $O$, into two parts: one is a parabolic sector, and the other consists of two hyperbolic sectors.
		\end{itemize}	
	\end{theorem}

A system with the nilpotent equilibrium at the origin can be transformed to
	\begin{equation}\label{AA}
	\frac{d x}{d t}=y, \quad \frac{d y}{d t}=a_{k} x^{k}\left(1+h(x)\right)+b_{n} x^{n} y\left(1+g(x)\right)+y^{2} p(x, y),
	\end{equation}	
	where $h(x), ~g(x), ~p(x, y)$ are analytic functions in $S_{\delta}(O)$. Moreover, $h(O)=g(O)=0, ~a_{k} \neq 0, ~k \geqslant 2 ; ~b_{n}$ can be zero, and when $b_{n} \neq 0,~ n \geqslant 1 .$

	\begin{theorem}\label{thm7.2}
		\cite[Theorem 7.2 of Chapter 2]{ZDHD}
		For system \eqref{AA} with  $k=2 m+1, m\geqslant 1$,  the equilibrium $O$ has the local property as that in {\rm Table \ref{T7.2}}, where $\lambda=b_{n}^{2}+4(m+1) a_{2 m+1}$.

	\begin{table}
		\renewcommand\arraystretch{1.5}
		\setlength{\tabcolsep}{2mm}{
		\caption{\label{T7.2}}
		\centering
		\begin{tabular}{c|c|c|c}
			\hline
			\multicolumn{3}{c|}
			{Relations between $a_{2 m+1}, b_{n}, \lambda, m, n$}
			
			& Type of equilibrium $O$
			\\
			\hline
			\multicolumn{3}{c|}
			{$a_{2 m+1}>0$}
			
			& saddle
			\\
			\hline
			\multirow{4}{*}
			{$
				\begin{aligned}
				&  \\
				&  \\
				& {{a}_{2m+1}}<0
				\end{aligned}
				$}
			
			&\multicolumn{2}{c|}{$b_n=0$}
			
			& center or focus
			\\
			\cline{2-4}
			
			& \multirow{3}{*}{$
				\begin{aligned}
				&  \\
				&  \\
				& b_n\ne 0
				\end{aligned}
				$}
			
			& $n>m$; or $m=n$ and $\lambda<0$
			
			& center or focus
			\\
			\cline{3-4}
			& &
			\begin{tabular}[c]{@{}c@{}}
				$n$ is even
				$\left\{
				\begin{aligned}	
				& n<m,~ or
				\\ 	
				& n=m ~and~ \lambda \geqslant 0	
				\end{aligned}
				\right.$
			\end{tabular}
			
			& node
			\\
			
			\cline{3-4}
			& &
			\begin{tabular}[c]{@{}c@{}}
				$n$ is odd
				$\left\{
				\begin{aligned}
				& n<m,~ or
				\\ 	
				& n=m~ and ~\lambda \geqslant 0 	
				\end{aligned}
				\right.$
			\end{tabular}
			
			&
			\begin{tabular}[c]{@{}c@{}}
				$S_{\delta}(O)$ consists of one hyperbolic \\ sector and one elliptic sector
			\end{tabular}
			\\
			\hline
		\end{tabular}}
	\end{table}
		\end{theorem}

 Notice that
	the	
	results of Theorem \ref{thm7.1} were initially obtained by Lyapunov in \cite[p. 301]{Lya}
	and
	 the	
	 results of Theorem \ref{thm7.2} were initially obtained by Andreev in \cite{Andreev}.

	 \section*{Appendix C}
	 
	 	 By a Poincar\'e transformation
	 $$
	 x=\frac{1}{z},\quad y=\frac{u}{z},
	 $$
	 system \eqref{61} is changed to
	 \begin{equation}\label{y}
	 \left\{\begin{aligned}
	 \frac{du}{d\tau}&=-z^2(u^2+\mu u+\lambda)-(cu^3+bu^2+au+1),\\
	 \frac{dz}{d\tau}&=-z^3u,
	 \end{aligned}
	 \right.
	 \end{equation}
	 where $d\tau=dt/z^2$. Notice that the abscissaes of the equilibria of system \eqref{y}  on $z = 0$ are the zeros   of the polynomial $\Phi(u):=cu^3+bu^2+au+1$.
	 
	 \begin{lemma}
	 	\label{lemy1}
	 	If $c=0$, system \eqref{y} has two equilibria $A=\left((-a-\sqrt{a^2-4b})/2b,0\right)$ and $B= \left((-a+\sqrt{a^2-4b})/2b,0\right)$  for $a^2-4b>0$,  only one equilibrium $C=(-a/2b,0)$ for   $a^2-4b=0$, and no an equilibrium for $a^2-4b<0$.
	 	Moreover, 	$A$ is an  unstable node for $b>0$, a saddle for $b<0$; $B$ is a saddle
	 	and $C$ is a degenerate equilibrium.
	 \end{lemma}
	 
	 \begin{proof}
	 	The number of equilibria follows directly from the roots of $\Phi(u)=0$. On the properties of these equilibria, computing the Jacobian matrices of the system at $A$, $B$ and $C$ gives
	 	$$
	 	J_A:=	
	 	\begin{pmatrix}
	 	\sqrt{a^2-4b}&0\\
	 	0&0\\
	 	\end{pmatrix},
	 	\quad
	 	J_B:=	
	 	\begin{pmatrix}
	 	-\sqrt{a^2-4b}&0\\
	 	0&0\\
	 	\end{pmatrix},
	 	\quad
	 	J_C:=	
	 	\begin{pmatrix}
	 	0&0\\
	 	0&0\\
	 	\end{pmatrix},
	 	$$
	 	respectively. Thus, $A$ and $B$ are semi-hyperbolic equilibria and $C$ is a degenerate equilibrium.
	 	
	 	Set $u_1=-a-\sqrt{a^2-4b})/2b$ and $u_2=(-a+\sqrt{a^2-4b})/2b$.		 	By the transformation $(u,z)\to(u+u_1,z)$
	 	system \eqref{y} is changed to
	 	\begin{equation*}
	 	\left\{\begin{aligned}
	 	\frac{du}{d\tau}&=-z^2\left(u^2+(\mu+2u_1) u+u_1^2+\mu u_1+\lambda\right)-bu^2-(2bu_1+a)u=:P_1(u,z),\\
	 	\frac{dz}{d\tau}&=-z^3u-u_1z^3=:Q_1(u,z).
	 	\end{aligned}
	 	\right.
	 	\end{equation*}
	 	The implicit function theorem shows that $P_1(u,z)=0$
	 	has a unique root $u=\phi_1(z)$ for small $|z|$.
	 	Thus,
	 	$$
	 	Q_1(\phi_1(z),z)=-u_1 z^3+o(z^3).
	 	$$
	 	Notice that $u_1<u_2<0$ for $b>0$ and $u_2<0<u_1$ for $b<0$.
	 	By Theorem \ref{thm7.1} in Appendix B, it follows that $A$ is an unstable node for $b>0$, and is a saddle for $b<0$.   Similarly, $B$ is a saddle of system \eqref{y}.
	 	The proof is finished.
	 \end{proof}
	 
	 Next we give  the qualitative property of $C$.
	 \begin{lemma}	
	 	\label{lem10}
	 	The qualitative property of $C$ is shown  in  {\rm Tables \ref{C1}--\ref{C3}}.
	 	
	 	\begin{table}[htp]
	 		\renewcommand\arraystretch{2}
	 		\setlength{\tabcolsep}{2.5mm}{
	 			\caption{\label{C1} Numbers of orbits connecting $C$ for  $\gamma>0$.}
	 			\begin{tabular}{c|c}
	 				\hline
	 				Exceptional directions & Numbers of orbits
	 				\\
	 				\hline
	 				$\theta=0$ & one  \ $(+)$
	 				\\
	 				\hline
	 				$\theta=\pi$ & one  \ $(-)$
	 				\\
	 				\hline
	 		\end{tabular}}
	 		\begin{center}
	 			$(-)($resp.$(+))$ means that the orbits approaching $C$ as $\tau\to -\infty($resp.  $\tau\to +\infty)$.
	 		\end{center}
	 	\end{table}
	 	\begin{table}[htp]
	 		\renewcommand\arraystretch{2}
	 		\setlength{\tabcolsep}{2.5mm}{
	 			\caption{\label{C2} Numbers of orbits connecting $C$ for  $\gamma=0$.}
	 			\begin{tabular}{c|c}
	 				\hline
	 				Exceptional directions & Numbers of orbits
	 				\\
	 				\hline
	 				$\theta=0$ & one  \ $(+)$
	 				\\
	 				\hline
	 				$\theta=\frac{\pi}{2}$ & $\infty$  \ $(-)$
	 				\\
	 				\hline
	 				$\theta=\pi$ & one  \ $(-)$
	 				\\
	 				\hline
	 				$\theta=\frac{3\pi}{2}$ &  $\infty$  \ $(-)$
	 				\\
	 				\hline
	 		\end{tabular}}
	 	\end{table}
	 	\begin{table}[htp]
	 		\renewcommand\arraystretch{2}
	 		\setlength{\tabcolsep}{2.5mm}{
	 			\caption{\label{C3} Numbers of orbits connecting $C$ for  $\gamma<0$.}
	 			\begin{tabular}{c|c}
	 				\hline
	 				Exceptional directions & Numbers of orbits
	 				\\
	 				\hline
	 				$\theta=0$ & one  \ $(+)$
	 				\\
	 				\hline
	 				$\theta=\arctan\sqrt{\frac{-b}{\gamma}}$ &  one  \ $(-)$
	 				\\
	 				\hline
	 				$\theta=\pi-\arctan\sqrt{\frac{-b}{\gamma}}$ &  $\infty$  \ $(-)$
	 				\\
	 				\hline
	 				$\theta=\pi$ & one  \ $(-)$
	 				\\
	 				\hline	
	 				$\theta=\pi+\arctan\sqrt{\frac{-b}{\gamma}}$ & $\infty$ \ $(-)$
	 				\\
	 				\hline
	 				$\theta=2\pi-\arctan\sqrt{\frac{-b}{\gamma}}$ &  one \ $(-)$
	 				\\
	 				\hline
	 		\end{tabular}}
	 	\end{table}
	 \end{lemma}
	 
	 \begin{proof}
	 	In this case, one has $b=a^2/4>0$ and $u_0:=-a/2b<0$. With the transformation $(u,z)\to (u+u_0, z)$,
	 	system \eqref{y} becomes
	 	\begin{equation}
	 	\label{y1}\left\{\begin{aligned}
	 	\frac{du}{d\tau}&=-z^2\left(u^2+(\mu+2u_0) u+u_0^2+\mu u_0+\lambda\right)-bu^2,\\
	 	\frac{dz}{d\tau}&=-z^3u-u_0z^3.
	 	\end{aligned}
	 	\right.
	 	\end{equation}
	 	With a polar transformation $(u,z)=(r\cos\theta,r\sin\theta)$,   system \eqref{y1} can be written as
	 	\begin{equation}
	 	\notag
	 	\frac{1}{r}\frac{dr}{d\theta}=\frac{H_1(\theta)+\widetilde{H_1}(\theta,r)}{G_1(\theta)+\widetilde{G_1}(\theta,r)},
	 	\end{equation}
	 	where
	 	$G_1(\theta)=\sin\theta((u_0^2+\mu u_0+\lambda)\sin^2\theta+b\cos^2\theta)$, $H_1(\theta)=-\cos\theta((u_0^2+\mu u_0+\lambda)\sin^2\theta+b\cos^2\theta)$, and $\widetilde{H_1}(\theta,r),  \widetilde{G_1}(\theta,r)\to 0$ as $r\to0$.
	 	A necessary condition on existence of exceptional directions is $G_1(\theta)=0$ by \cite[Chapter 2]{ZDHD}.
	 	Whereas, the zeros of $G_1(\theta)$ is strongly related to the sign of $\gamma:= u_0^2+\mu u_0+\lambda$.
	 	
	 	For $\gamma>0$, it is easy to check that $G_1(\theta)=0$   has exactly two roots $0$, $\pi$ in $\theta\in[0,2\pi)$. Moreover, easy calculation gives  $G_1'(0)H_1(0)=G_1'(\pi)H_1(\pi)=-b^2<0$.  By
	 	{ $H_1(0)<0$, $H_1(\pi)>0$ }
	 	and \cite[Theorem 3.7 of  Chapter 2]{ZDHD}, system \eqref{y1} has
	 	a unique orbit approaching  $(0, 0)$ in the direction $\theta = \pi$ as $\tau\to-\infty$,  and 	a unique orbit approaching $(0, 0)$  in the direction $\theta = 0$  as $\tau\to+\infty$. So is $C$.

	 	For $\gamma<0$, it is easy to check that  $G_1(\theta)$ has six zeros $\theta=0$, $\arctan\left(\sqrt{-b/\gamma}\right)$, $\pi-\arctan \left(\sqrt{-b/\gamma}\right)$, $\pi$, $ \pi+\arctan\left(\sqrt{-b/\gamma}\right)$, $2\pi-\arctan\left(\sqrt{-b/\gamma}\right)$ in $\theta\in[0,2\pi)$.
	 	Since $G_1'(0)H_1(0)=G_1'(\pi)H_1(\pi)=-b^2<0$, it follows from
	 	{$H_1(0)<0$, $H_1(\pi)>0$  and}
	 	\cite[Theorem 3.7 of  Chapter 2]{ZDHD} that system \eqref{y1} has
	 	a unique orbit approaching  $(0, 0)$ in the direction $\theta = \pi$ as $\tau\to-\infty$,  and 	a unique orbit approaching $(0, 0)$  in the direction $\theta = 0$  as $\tau\to+\infty$. So is $C$.
	 	Since the other four zeros of $G_1(\theta)$ are also those of $H_1(\theta)$, one cannot apply  the normal sector method  (see \cite[Chapter 2]{ZDHD}) to analyze the four exceptional directions
	 	$\theta=\arctan\left(\sqrt{-b/\gamma}\right)$, $\pi-\arctan \left(\sqrt{-b/\gamma}\right)$,   $ \pi+\arctan\left(\sqrt{-b/\gamma}\right)$, $2\pi-\arctan\left(\sqrt{-b/\gamma}\right)$  of the origin for system \eqref{y1}.  Instead, we adopt Briot--Bouquet transformations to blow up the four directions.
	 	
	 	With the Briot--Bouquet transformation
	 	$	z=\widetilde{z} u$,
	 	system \eqref{y1} is changed to
	 	\begin{equation}
	 	\label{bb3}
	 	\left\{\begin{aligned}
	 	\frac{du}{d\delta}&=-\widetilde{z} ^2u\left(u^2+(\mu+2u_0) u+\gamma\right)-bu,\\
	 	\frac{d\widetilde{z} }{d\delta}&=(\mu+u_0)\widetilde{z} ^3u+\gamma\widetilde{z}^3+b\widetilde{z},
	 	\end{aligned}
	 	\right.
	 	\end{equation}
	 	where $d\delta=ud\tau$.  System \eqref{bb3}  has three  equilibria $(0,0)$, $\left(0,\sqrt{-b/\gamma}\right)$ and $\left(0,-\sqrt{-b/\gamma}\right)$.
	 	The equilibrium $(0,0)$ is a saddle. For the other two  equilibria,  	
	 	taking transformation $(u,\widetilde{z})\to (u, \widetilde{z}+z_1)$ with $z_1=\sqrt{-b/\gamma}$, system \eqref{bb3} becomes
	 	\begin{equation}
	 	\label{bbb3}
	 	\left\{\begin{aligned}
	 	\frac{du}{d\delta}=&-(\widetilde{z}^2+2z_1\widetilde{z})\left(u^3+(\mu+2u_0) u^2+\gamma u\right)-z_1^2\left(u^3+(\mu+2u_0) u^2\right),\\
	 	\frac{d\widetilde{z} }{d\delta}=&(\mu+u_0)z_1^3u+2\gamma z_1^2\widetilde{z}+(\mu+u_0)(u\widetilde{z}^3+3z_1u\widetilde{z}^2+3z_1^2u\widetilde{z})\\
	 	&+\gamma \widetilde{z}^3+3\gamma z_1\widetilde{z}^2.
	 	\end{aligned}
	 	\right.
	 	\end{equation}
	 	A further transformation $(u,\widetilde{z}) \to\left(u, \left(\widetilde{z}-(\mu+u_0)z_1^3u\right)/2\gamma z_1^2  \right) $ sends system \eqref{bbb3} to
	 	\begin{equation}
	 	\label{bbb4}
	 	\left\{\begin{aligned}
	 	\frac{du}{d\delta}&=-u_0u^2-\frac{1}{z_1}u\widetilde{z}+h.o.t.=:P_2(u,\widetilde{z}),\\
	 	\frac{d\widetilde{z} }{d\delta}&=2\gamma z_1^2\widetilde{z}+h.o.t.=:Q_2(u,\widetilde{z}).
	 	\end{aligned}
	 	\right.
	 	\end{equation}
	 	By the implicit function theorem, $Q_2(u,\widetilde{z})=0$
	 	has a unique root $\widetilde{z}=\phi_2(u)$ for small $|u|$.
	 	Thus,
	 	$$
	 	P_2(u,\phi_2(u)=-u_0 u^2+o(u^2).
	 	$$
	 	Theorem \ref{thm7.1} in Appendix B together with $\gamma<0$ and $u_0<0$ verifies that the origin of system \eqref{bbb4} is a saddle--node, so is $\left(0,\sqrt{-b/\gamma}\right)$   of system \eqref{bb3}. Similarly,  $\left(0,-\sqrt{-b/\gamma}\right)$ is also a saddle--node.
	 	Figure \ref{tu9}$(a)$ illustrates the qualitative properties of the equilibria $(0, 0)$, $\left(0,\sqrt{-b/\gamma}\right)$ and $\left(0,-\sqrt{-b/\gamma}\right)$ of system  \eqref{bb3} in the $(u,\widetilde{z})$ plane.
	 	In fact, when $u\ne0$,
	 	the Briot--Bouquet transformation $z=\widetilde{z} u$ is a topological transformation from  $(u, z)$ plane to $(u, \widetilde{z})$ plane, mapping the first, second, third and fourth quadrants respectively
	 	into the first, third, second and fourth quadrants. Moreover,
	 	when $u=0$, the transformation makes the whole $\widetilde{z}$--axis {shrinking to} the
	 	origin in the $(u, z)$ plane.
	 	{ 	In other words, all orbit segments in first, second, third and fourth quadrants in the { $(u, z)$  plane}
	 		correspond first, third, second  and fourth quadrants in the {  $(u, \widetilde{z})$   plane}, respectively.}
	 	Further, an orbit with the initial point $(0,\tilde z_0)$ in
	 	the $(u , \widetilde{z})$ plane becomes an orbit connecting $O$ along $\theta=\tilde\theta_0$ in the $(u, z)$ plane.
	 	
	 	Therefore,
	 	blowing down these equilibria to $C$ yields that system \eqref{y} has infinitely many orbits approaching $C$ in respectively the directions $\theta=\pi-\arctan\left(\sqrt{-b/\gamma}\right)$ and $\pi+\arctan\left(\sqrt{-b/\gamma}\right)$   as $\tau\to -\infty$,  a unique orbit approaching  $C$ in respectively the directions $\theta=\arctan\left(\sqrt{-b/\gamma}\right)$ and  $2\pi-\arctan\left(\sqrt{-b/\gamma}\right)$  as $\tau\to -\infty$. Figure \ref{tu9}$(b)$ exhibits the local structure of  $C$ in the $(u, z)$ plane.
	 	
	 	\begin{figure}[htp]
	 		\centering
	 		\subfigure[{$(u,\widetilde{z})$ plane for system \eqref{bb3}}]{
	 			\includegraphics[width=0.35\textwidth]{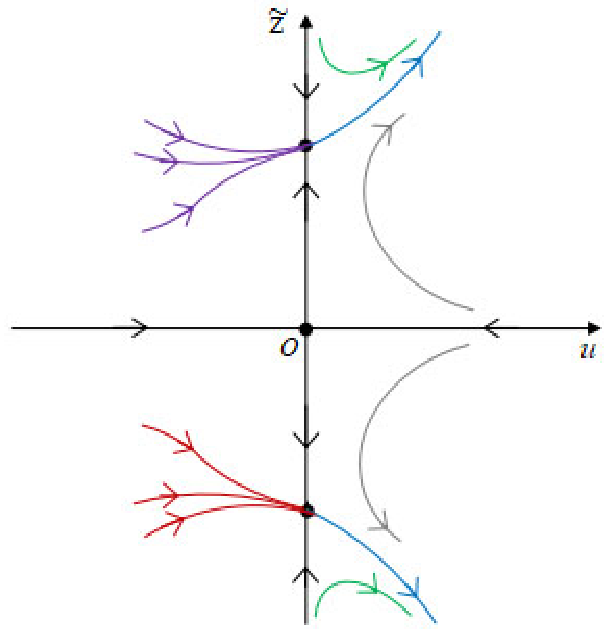}}
	 		\quad
	 		\subfigure[{$(u, z)$ plane for system \eqref{y1}}]{
	 			\includegraphics[width=0.35\textwidth]{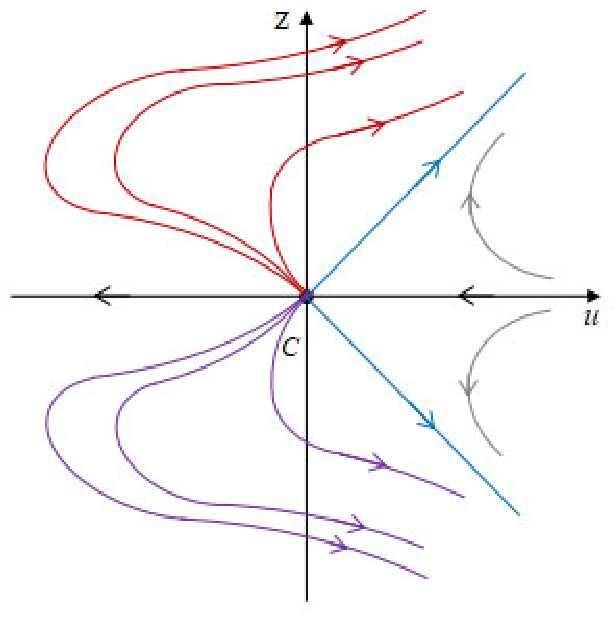}}
	 		\caption{ Orbits  changing under the Briot--Bouquet transformation for $\gamma<0$. }
	 		\label{tu9}
	 	\end{figure}

	 	For $\gamma=0$, the equation $G_1(\theta)=0$ has four roots $\theta=0, \pi/2$, $\pi$,  $3\pi/2$ in $\theta\in[0,2\pi)$. Moreover, $G_1'(0)H_1(0)=G_1'(\pi)H_1(\pi)=-b^2<0$.
	 	By
	 	{ $H_1(0)<0$, $H_1(\pi)>0$  and}
	 	\cite[Theorem 3.7 of  Chapter 2]{ZDHD}, system \eqref{y1} has
	 	a unique orbit approaching  $(0, 0)$ in the direction $\theta = \pi$ as $\tau\to-\infty$,  and 	a unique orbit approaching $(0, 0)$  in the direction $\theta = 0$  as $\tau\to+\infty$. So is $C$.
	 	However,
	 	$H_1(\pi/2)=H_1(3\pi/2)=0$.
	 	We need  Briot--Bouquet transformations to blow up the   directions $\theta=\pi/2$ and $\theta= 3\pi/2$.
	 	
	 	With the Briot--Bouquet transformation $		 				u=\widetilde{u}z,
	 	$
	 	system \eqref{y1} becomes
	 	\begin{equation}
	 	\label{bb4}
	 	\left\{\begin{aligned}
	 	\frac{d\widetilde{u}}{ds}&=(-u_0-\mu)\widetilde{u}z-b\widetilde{u}^2,\\
	 	\frac{dz }{ds}&=-\widetilde{u}z^3-u_0z^2 ,
	 	\end{aligned}
	 	\right.
	 	\end{equation}
	 	where $ds=zd\tau$.
	 	The polar change of variables $(\widetilde{u},z)=(r \cos\theta,r\sin\theta)$ sends system \eqref{bb4} to
	 	\begin{equation}
	 	\notag
	 	\frac{1}{r}\frac{dr}{d\theta}=\frac{H_2(\theta)+\widetilde{H_2}(\theta,r)}{G_2(\theta)+\widetilde{G_2}(\theta,r)},
	 	\end{equation}
	 	where
	 	$$
	 	G_2(\theta)=\sin\theta\cos\theta(b\cos\theta+\mu\sin\theta)
	 	$$
	 	and
	 	$$
	 	H_2(\theta)=-u_0\sin^3\theta+(-u_0-\mu)\sin\theta\cos^2\theta-b\cos^3\theta.
	 	$$
	 	The condition $\gamma=0$ implies $\mu>0$, and so  $G_2(\theta)$ has six zeros $\theta=0$,  $\pi/2$,  $\pi-\arctan(b/\mu)$, $\pi$, $3\pi/2$, $2\pi-\arctan(b/\mu)$ in $[0,2\pi)$.
	 	Moreover, $G_2'(0)H_2(0)=G_2'(\pi)H_2(\pi)=-b^2<0$,  $G_2'(\pi/2)H_2(\pi/2)=G_2'(3\pi/2)H_2(3\pi/2)=u_0\mu<0$, $G_2'(\theta_0)H_2(\theta_0)=G_2'(\pi+\theta_0)H_2(\pi+\theta_0)=-\mu u_0\sin^2\theta_0>0$, where $\theta_0=\pi-\arctan(b/\mu)$. By  \cite[Theorems 3.7 and 3.8  of Chapter 2]{ZDHD}, system \eqref{bb4} has
	 	a unique orbit approaching $(0, 0)$ in respectively the directions $\theta = 0$ and $3\pi/2$  as $s\to+\infty$,
	 	a unique orbit approaching $(0, 0)$ in respectively the
	 	directions $\theta= \pi$ and $\pi/2$  as $s\to-\infty$,
	 	infinitely many orbits approaching $(0, 0)$ in the direction $2\pi-\arctan(b/\mu)$ as $s\to+\infty$, and
	 	infinitely many orbits approaching $(0, 0)$ in the direction $\theta =\pi- \arctan(b/\mu)$ as $s\to-\infty$.
	 	Figure \ref{tu18}$(a)$ illustrates the local structure of $(0, 0)$ for system \eqref{bb4}  in the $(\widetilde{u}, z)$ plane.
	 	It follows that  system \eqref{y} has infinitely many orbits approaching $C$ in respectively the directions $\pi/2$ and $3\pi/2$ as $\tau\to-\infty$. Figure \ref{tu18}$(b)$ exhibits the local qualitative structure of system \eqref{y} at $C$ in the $(u, z)$ plane.
	 	\begin{figure}[htp]
	 		\centering
	 		\subfigure[{$(\widetilde{u}, z)$ plane for system \eqref{bb4}}]{
	 			\includegraphics[width=0.35\textwidth]{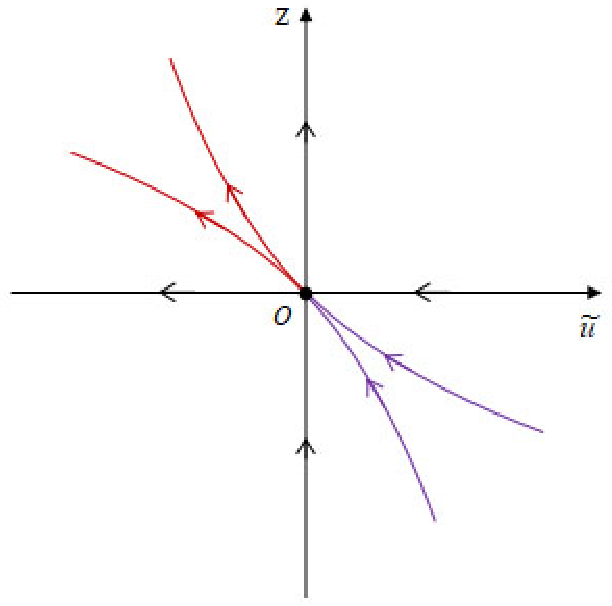}}
	 		\quad
	 		\subfigure[{$(u,z)$ plane for system \eqref{y}}]{
	 			\includegraphics[width=0.35\textwidth]{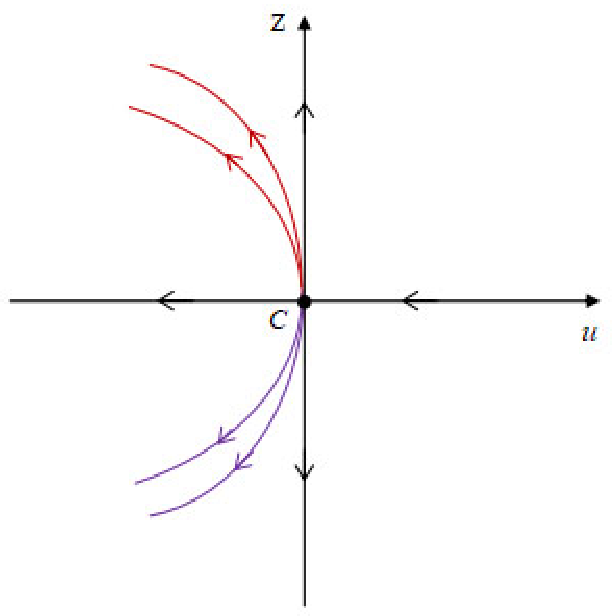}}
	 		\caption{ Orbits  changing under the  Briot--Bouquet transformation for $\gamma=0$.}
	 		\label{tu18}
	 	\end{figure}
	 \end{proof}
	 
	 	By Lemmas
	 \ref{lemy1} and \ref{lem10}, we obtain the local phase portraits of system \eqref{61} at infinity of the Poincar\'e disc, as those shown Figure \ref{tu11}, where $I_{A^\pm}$, $I_{B^\pm}$ and $I_{C^\pm}$ correspond respectively to  equilibria $A$, $B$ and $C$  of system \eqref{y}.
	 \begin{figure}[htp]
	 	\centering
	 	\subfigure[{ $b<0$}]{
	 		\includegraphics[width=0.28\textwidth]{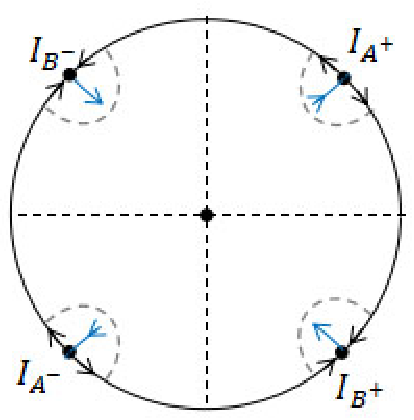}}
	 	\quad
	 	\subfigure[{   $0<b<a^2/4$}]{
	 		\includegraphics[width=0.28\textwidth]{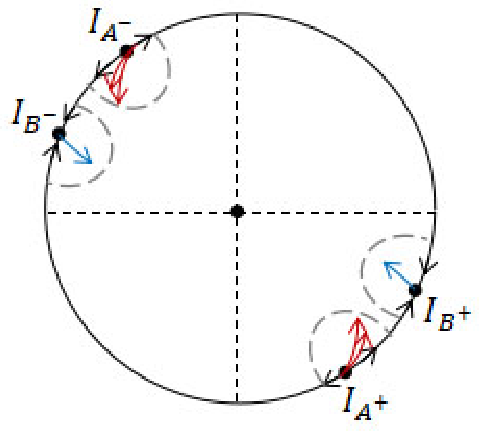}}
	 	\\
	 	\centering
	 	\subfigure[{$b=a^2/4$, $\gamma>0$}]{
	 		\includegraphics[width=0.28\textwidth]{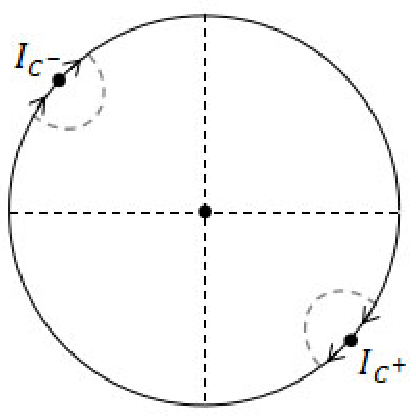}}
	 	\quad
	 	\subfigure[{$b=a^2/4$, $\gamma<0$}]{
	 		\includegraphics[width=0.28\textwidth]{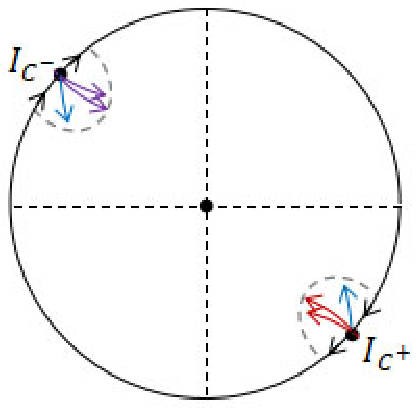}}
	 	\quad
	 	\subfigure[{ $b=a^2/4$, $\gamma=0$}]{
	 		\includegraphics[width=0.28\textwidth]{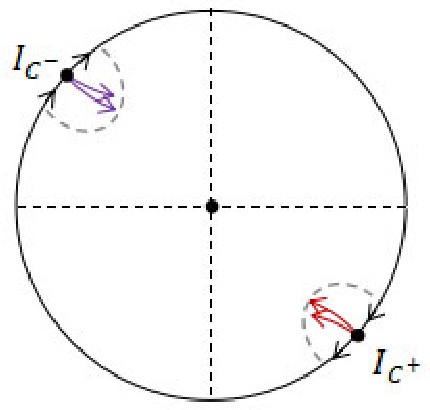}}
	 	\caption{{Locally qualitative property of} equilibria at infinity $I_{A^{\pm}}$, $I_{B^{\pm}}$ and $I_{C^{\pm}}$   in the Poincar\'e disc for $c=0$.}
	 	\label{tu11}
	 \end{figure}

	 
	 \begin{lemma}\label{lemma}
	 	When $c>0$, the  equilibria of system \eqref{y} is shown in the following Table \ref{NE}.
	 	\begin{table}[htp]
	 		\renewcommand\arraystretch{1.5}
	 		\setlength{\tabcolsep}{5mm}{
	 			\caption{\label{NE}The equilibria of system \eqref{y} for $c>0$.}
	 			
	 			\begin{tabular}{c|c|c|c}
	 				\hline
	 				\multicolumn{3}{c|}
	 				{Cases of parameters} & Equilibria
	 				\\
	 				\hline
	 				\multicolumn{3}{c|}
	 				{$b\in\left[-\sqrt{3ac},0\right)\cup\left(0,\sqrt{3ac}\right ]$} & $E:(u_3,0)$
	 				\\
	 				\hline
	 				\multirow{3}{*}
	 				{\begin{tabular}[c]{@{}c@{}}
	 						$b\in\left(-\infty,-\sqrt{3ac}\right)$
	 				\end{tabular}}
	 				& \multicolumn{2}{c|}{$\Phi(\varrho_2)<0$} & $F_1:(u_4,0)$, $F_2:(u_5,0)$, $F_3:(u_6,0)$
	 				\\
	 				\cline{2-4}
	 				& \multicolumn{2}{c|}{$\Phi(\varrho_2)=0$} & $F:(u_7,0)$, $Q:(\varrho_2,0)$
	 				\\
	 				\cline{2-4}
	 				& \multicolumn{2}{c|}{$\Phi(\varrho_2)>0$} & $E:(u_3,0)$
	 				\\
	 				\hline
	 				\multirow{5}{*}
	 				{\begin{tabular}[c]{@{}c@{}}
	 						$b\in\left(\sqrt{3ac},+\infty\right)$
	 				\end{tabular}}
	 				& \multicolumn{2}{c|}{$\Phi(\varrho_1)<0$} & $E:(u_3,0)$
	 				\\
	 				\cline{2-4}
	 				&\multicolumn {2}{c|}{$\Phi(\varrho_1)=0$} & $K:(\varrho_1,0)$,  $F:(u_7,0)$
	 				\\
	 				\cline{2-4}
	 				& \multirow{3}{*}{\begin{tabular}[c]{@{}c@{}} $\Phi(\varrho_1)>0$\end{tabular}} & $\Phi(\varrho_2)>0$ & $E:(u_3,0)$
	 				\\
	 				\cline{3-4}
	 				&  & $\Phi(\varrho_2)=0$ & $F:(u_7,0)$, $Q:(\varrho_2,0)$
	 				\\
	 				\cline{3-4}
	 				&  & $\Phi(\varrho_2)<0$ & $F_1:(u_4,0)$, $F_2:(u_5,0)$, $F_3:(u_6,0)$
	 				\\
	 				\hline
	 		\end{tabular}}
	 		\begin{center}
	 			Remark :  $\varrho_1:=(-b-\sqrt{b^2-3ac})/3c$, $\varrho_2:=(-b+\sqrt{b^2-3ac})/3c$, and $u_3,\cdots, u_7$ are probable zeros of $\Phi(u)=cu^3+bu^2+au+1$.
	 		\end{center}
	 	\end{table}
	 \end{lemma}
	 
	 \begin{proof}
	 	When $b\in\left[-\sqrt{3ac},0\right)\cup\left(0,\sqrt{3ac}\right]$, $\Phi'(u)=3cu^2+2bu+a\geqslant0$ holds identically. Thus,  $\Phi(u)$  is  increasing and has a unique zero.   When $b\in\left(-\infty,-\sqrt{3ac}\right)\cup\left(\sqrt{3ac},+\infty\right)$, $\Phi'(u)=3cu^2+2bu+a=0$ has two roots
	 	$$
	 	\varrho_1=\frac{-b-\sqrt{b^2-3ac}}{3c}, \quad \varrho_2=\frac{-b+\sqrt{b^2-3ac}}{3c}.
	 	$$
	 	Therefore,
	 	$\Phi(u)$  is  increasing in $(-\infty,\varrho_1)$, $(\varrho_2,+\infty)$ and decreasing in $(\varrho_1,\varrho_2)$. The zeros of $\Phi(u)$ are determined by discussing the signs of $\Phi(\varrho_1)$ and $\Phi(\varrho_2)$.  The proof is finished.
	 \end{proof}
	 
	 We study further the qualitative properties of the equilibria  in Lemma \ref{lemma}.

	 \begin{lemma}\label{E}
	 	When $b\in[-\sqrt{3ac},0)\cup(0,\sqrt{3ac}]$,  $E$ is a saddle.
	 \end{lemma}
	 \begin{proof}\rm
	 	Clearly,	the  Jacobian matrix at $E$ is
	 	$$
	 	J_E:=	\begin{pmatrix}
	 	-3cu_3^2-2bu_3-a&0\\
	 	0&0\\
	 	\end{pmatrix}.
	 	$$
	 	When $3cu_3^2+2bu_3+a>0$, $E$ is a semi-hyperbolic equilibrium.   With a transformation $(u,z)\to (u+u_3,z)$,
	 	system \eqref{y} is changed into
	 	\begin{equation}
	 	\label{yyy}\left\{\begin{aligned}
	 	\frac{du}{d\tau}=&-z^2\left(u^2+(\mu+2u_3) u+u_3^2+\mu u_3+\lambda\right)-cu^3-(3cu_3+b)u^2
	 	\\
	 	&-(3cu_3^2+2bu_3+a)u=:P_3(u,z),
	 	\\
	 	\frac{dz}{d\tau}=&-z^3u-u_3z^3=:Q_3(u,z).
	 	\end{aligned}
	 	\right.
	 	\end{equation}
	 	By the implicit function theorem, $P_3(u,z)=0$
	 	has a unique root $u=\phi_3(z)=-(u_3^2+\mu u_3+\lambda)/(3cu_3^2+2bu_3+a)z^2+o(z^2)$ for small $|z|$.
	 	Thus,
	 	$$
	 	Q_3(\phi_3(z),z)=-u_3 z^3+o(z^3).
	 	$$
	 	By Theorem \ref{thm7.1} of Appendix B and $u_3<0$, $E$ is a saddle.
	 	
	 	When $3cu_3^2+2bu_3+a=0$, $E$ is a degenerate equilibrium and
	 	system \eqref{yyy} can be simplified as
	 	\begin{equation}
	 	\label{y2}
	 	\left\{\begin{aligned}
	 	\frac{du}{d\tau}&=-z^2(u^2+(\mu+2u_3) u+u_3^2+\mu u_3+\lambda)-cu^3,\\
	 	\frac{dz}{d\tau}&=-z^3u-u_3z^3.
	 	\end{aligned}
	 	\right.
	 	\end{equation}

	 	Firstly,
	 	consider $u_3^2+\mu u_3+\lambda=0$.
	 	Then,	system \eqref{y2} is simplified as
	 	\begin{equation}
	 	\label{y7}
	 	\left\{\begin{aligned}
	 	\frac{du}{d\tau}&=-z^2(u^2+(\mu+2u_3) u)-cu^3,\\
	 	\frac{dz}{d\tau}&=-z^3u-u_3z^3.
	 	\end{aligned}
	 	\right.
	 	\end{equation}
	 	Using a polar coordinate $(u,z)=(r\cos\theta, z=r\sin\theta)$,
	 	system \eqref{y7} is transformed into the polar form
	 	\begin{equation}
	 	\notag
	 	\frac{1}{r}\frac{dr}{d\theta}=\frac{H_3(\theta)+\widetilde{H_3}(\theta,r)}{G_3(\theta)+\widetilde{G_3}(\theta,r)},
	 	\end{equation}
	 	where
	 	$$
	 	G_3(\theta)=\sin\theta\cos\theta\left((\mu+u_3)\sin^2\theta+c\cos^2\theta\right)$$
	 	and
	 	$$
	 	H_3(\theta)=-u_3\sin^4\theta-(\mu+2u_3)\cos^2\theta\sin^2\theta-c\cos^4\theta.
	 	$$
	 	It follows from $u_3^2+\mu u_3+\lambda=0$ that $\mu+u_3>0$ (resp. $=0$) if $\lambda>0$ (resp. $=0$). Thus, $G_3(\theta)$ has four zeros $\theta=0$, $\pi/2$, $\pi$, $3\pi/2$ in $[0, 2\pi)$,
	 	where $0,\pi$ are simple zeros, $\pi/2, 3\pi/2$ are simple zeros (resp. zeros of three-multiple) for $\mu+u_3>0$
	 	(resp. $=0$).	
	 	We can check that
	 	$G_3'(0)H_3(0)=G_3'(\pi)H_3(\pi)=-c^2<0$.
	 	Moreover, we can obtain
	 	\[
	 	G_3'(\pi/2)H_3(\pi/2)=G_3'(3\pi/2)H_3(3\pi/2)=u_3(\mu+u_3)<0	
	 	\] for $\mu+u_3>0$
	 	and
	 	\[
	 	G_3'''(\pi/2)H_3(\pi/2)=G_3'''(3\pi/2)H_3(3\pi/2)=6cu_3<0
	 	\] for $\mu+u_3=0$.
	 	By \cite[ Theorem 3.7 of  Chapter 2 ]{ZDHD},
	 	there is a unique orbit  approaching $ (0,0)$ in respectively  the directions $\theta=\pi/2$, $3\pi/2$  as $\tau\to-\infty$,  and a unique orbit approaching $ (0,0)$ in respectively  the directions $\theta=0$, $\pi$ as $\tau\to+\infty$
	 	in system \eqref{y2}. In other words, $E$ is a degenerate saddle of system \eqref{y}.
	 	
	 	Secondly, consider $u_3^2+\mu u_3+\lambda\ne0$. Then, applying the polar transformation $(u,z)=(r\cos\theta, r\sin\theta)$ to system \eqref{y2}, we still get	
	 	\begin{equation}
	 	\notag
	 	\frac{1}{r}\frac{dr}{d\theta}=\frac{H_4(\theta)+\widetilde{H_4}(\theta,r)}{G_4(\theta)+\widetilde{G_4}(\theta,r)},
	 	\end{equation}
	 	where $G_4(\theta)=(u_3^2+\mu u_3+\lambda)\sin^3\theta$ and $H_4(\theta)=-(u_3^2+\mu u_3+\lambda)\cos\theta\sin^2\theta$.
	 	It is easy to check that $G_4(\theta)$ has two zeros $\theta=0$, $\pi$ in $[0,2\pi)$ and  $H_4(0)=H_4(\pi)=0$.
	 	Therefore, we need to desingularize further the degenerate equilibrium.		
	 	With
	 	the Briot--Bouquet transformation $
	 	z=\widetilde{z} u,
	 	$
	 	system \eqref{y2}
	 	becomes
	 	\begin{equation}
	 	\label{bb5}
	 	\left\{\begin{aligned}
	 	\frac{du}{d\delta}&=-\widetilde{z}^2u(u^2+(\mu+2u_3) u+u_3^2+\mu u_3+\lambda)-cu^2,\\
	 	\frac{d\widetilde{z} }{d\delta}&=(\mu+u_3)\widetilde{z}^3u+(u_3^2+\mu u_3+\lambda)\widetilde{z}^3+c\widetilde{z}u,
	 	\end{aligned}
	 	\right.
	 	\end{equation}
	 	where $d\delta=ud\tau$.   Since the origin  of system \eqref{bb5} is still degenerate, we repeat
	 	the aforementioned analysis steps.
	 	With a polar transformation $(u,\widetilde{z})=(r\cos\theta,r\sin\theta)$,
	 	system \eqref{bb5} is written as
	 	\begin{equation}
	 	\notag
	 	\frac{1}{r}\frac{dr}{d\theta}=\frac{H_5(\theta)+\widetilde{H_5}(\theta,r)}{G_5(\theta)+\widetilde{G_5}(\theta,r)},
	 	\end{equation}
	 	where $G_5(\theta)=2c\sin\theta\cos^2\theta$ and $H_5(\theta)=c\sin^2\theta\cos\theta-c\cos^3\theta$. It is easy to check that $\theta= 0$, $\pi/2$, $\pi$, $3\pi/2$ are four roots of $G_5(\theta)=0$ in $[0,2\pi)$. Moreover, $G_5'(0)H_5(0)=G_5'(\pi)H_5(\pi)=-2c^2<0$. Thus, there is a unique orbit approaching $ (0,0)$ in  respectively  the direction $\theta= 0$, $\pi$
	 	as $\delta\to+\infty$, and a unique orbit approaching $ (0,0)$ in   respectively the direction $\theta=\pi$
	 	as $\delta\to-\infty$ in system \eqref{bb5} by    applying  \cite[Theorem 3.7 of Chapter 2 ]{ZDHD}. For $\theta= \pi/2$, $3\pi/2$, we need to do more desingularisations because $H_5(\pi/2)=H_5(3\pi/2)=0$.
	 	
	 	With
	 	the second  Briot--Bouquet transformation
	 	$ u=\widetilde{u}\widetilde{z}$ ,
	 	system \eqref{bb5} is rewritten as
	 	\begin{equation}
	 	\label{bb6}
	 	\left\{\begin{aligned}
	 	\frac{d\widetilde{u}}{d\sigma}&=-\widetilde{u}^3\widetilde{z}^3-(2\mu+3u_3) \widetilde{u}^2\widetilde{z}^2-2(u_3^2+\mu u_3+\lambda)\widetilde{u}\widetilde{z} -2c\widetilde{u}^2,\\
	 	\frac{d\widetilde{z} }{d\sigma}&=(\mu+u_3)\widetilde{z}^3\widetilde{u}+(u_3^2+\mu u_3+\lambda)\widetilde{z}^2+c\widetilde{z}\widetilde{u},
	 	\end{aligned}
	 	\right.
	 	\end{equation}
	 	where $d\sigma=\widetilde{z}d\delta$. Similarly, transforming system \eqref{bb6} into  equation
	 	\begin{equation}
	 	\notag
	 	\frac{1}{r}\frac{dr}{d\theta}=\frac{H_6(\theta)+\widetilde{H_6}(\theta,r)}{G_6(\theta)+\widetilde{G_6}(\theta,r)},
	 	\end{equation}
	 	by a polar transformation $(\widetilde{u},\widetilde{z}) =(r\cos\theta,r\sin\theta)$, we obtain that
	 	$$
	 	G_6(\theta)= 3\sin\theta\cos\theta\left((u_3^2+\mu u_3+\lambda)\sin\theta+c\cos\theta\right) $$
	 	and
	 	$$
	 	H_6(\theta)=(u_3^2+\mu u_3+\lambda)\sin^3\theta+c\sin^2\theta\cos\theta-2(u_3^2+\mu u_3+\lambda)\sin\theta\cos^2\theta-2c\cos^3\theta.
	 	$$
	 	Thus, $G_6(\theta)=0$ has six roots  $\theta=0$, $\pi/2$, $\pi$, $3\pi/2$, $\arctan(-c/(u_3^2+\mu u_3+\lambda))$ ,  $\pi+\arctan(-c/(u_3^2+\mu u_3+\lambda))$ for $u_3^2+\mu u_3+\lambda<0$  and six roots
	 	$\theta=0$, $\pi/2$, $\pi$, $3\pi/2$,  $\pi-\arctan(c/(u_3^2+\mu u_3+\lambda))$,   $2\pi-\arctan(c/(u_3^2+\mu u_3+\lambda))$ for $u_3^2+\mu u_3+\lambda>0$ in $[0,2\pi)$.
	 	Then, we can check that  $G_6'(0)H_6(0)=G_6'(\pi)H_6(\pi)=-6c^2<0$, $G_6'(\pi/2)H_6(\pi/2)=G_6'(3\pi/2)H_6(3\pi/2)=-3(u_3^2+\mu u_3+\lambda)^2<0$.
	 	However,
	 	\[
	 	H_6\left(\arctan\left(-\frac{c}{u_3^2+\mu u_3+\lambda}\right)\right)=H_6\left(\pi+\arctan\left(-\frac{c}{u_3^2+\mu u_3+\lambda}\right)\right)=0
	 	\]
	 	for $u_3^2+\mu u_3+\lambda<0$
	 	and
	 	\[
	 	H_6\left(\pi-\arctan\left(\frac{c}{u_3^2+\mu u_3+\lambda}\right)\right)= H_6\left(2\pi-\arctan\left(\frac{c}{u_3^2+\mu u_3+\lambda}\right)\right)=0
	 	\]
	 	for $u_3^2+\mu u_3+\lambda>0$.
	 	Therefore, we need to desingularize further the degenerate equilibrium.
	 	
	 	Using the third transformation
	 	$
	 	\widetilde{z}=\left(\widetilde{\widetilde{z}}-z_2\right)\widetilde{u}
	 	$
	 	with $z_2=c/(u_3^2+\mu u_3+\lambda)$, we change system \eqref{bb6} into	
	 	\begin{equation}
	 	\label{bb7}
	 	\left\{\begin{aligned}
	 	\frac{d\widetilde{u}}{d\upsilon}&=-2(u_3^2+\mu u_3+\lambda)\widetilde{u}\widetilde{\widetilde{z}}-(\widetilde{\widetilde{z}}-z_2)^3\widetilde{u}^5-(2\mu+3u_3)(\widetilde{\widetilde{z}}-z_2)^2\widetilde{u}^3=:P_4(\widetilde{u},\widetilde{\widetilde{z}}),
	 	\\
	 	\frac{d\widetilde{\widetilde{z}} }{d\upsilon}&=-3c\widetilde{\widetilde{z}}
	 	+3(u_3^2+\mu u_3+\lambda)\widetilde{\widetilde{z}}^2+(3\mu+4u_3)(\widetilde{\widetilde{z}}-z_2)^3\widetilde{u}^2+(\widetilde{\widetilde{z}}-z_2)^4\widetilde{u}^4=:Q_4(\widetilde{u},\widetilde{\widetilde{z}}),
	 	\end{aligned}
	 	\right.
	 	\end{equation}
	 	where $d\upsilon=\widetilde{u}d\sigma$.
	 	On the one hand,   By the implicit function theorem,  $Q_4(\widetilde{u},\widetilde{\widetilde{z}})=0$ has a unique root $\widetilde{\widetilde{z}}=\phi_4(\widetilde{u})=-(3\mu+4u_3)z_2^3/(3c)\widetilde{u}^2+o(\widetilde{u}^2)$  for small $|\widetilde{u}|$. Thus,
	 	\begin{equation}
	 	\notag
	 	\begin{aligned}
	 	P_4(\widetilde{u},\phi_4(\widetilde{u})) =-\frac{u_3}{3}z_2^2\widetilde{u}^3+o(\widetilde{u}^3).
	 	\end{aligned}
	 	\end{equation}
	 	By Theorem \ref{thm7.1}  of Appendix B,  the orgin of system \eqref{bb7} is a saddle.  On the other hand, it is easy to check $(0,z_2)$ of  system \eqref{bb7} is a hyperbolic saddle. Thus,
	 	we can obtain  the qualitative properties of system \eqref{bb7}, as shown in Figure \ref{tu7} (a) (resp. Figure \ref{tu8} (a)) when $u_3^2+\mu u_3+\lambda>0$ (resp. $u_3^2+\mu u_3+\lambda<0$). Further, we obtain the qualitative properties of $(0, 0)$ in the $(\widetilde{u},\widetilde{z})$
	 	plane for system \eqref{bb6} and in the  $(u, \widetilde{z})$ plane
	 	for  system \eqref{bb5}, respectively.
	 	See Figures \ref{tu7} (b) and (c)  and Figures  \ref{tu8} (b) and (c).
	 	Finally, $E$ of  system \eqref{y} is a degenerate saddle, as shown in Figures \ref{tu7} (d) and \ref{tu8} (d).
	 	\begin{figure}[htp]
	 		\centering
	 		\subfigure[{$(\widetilde{u},\widetilde{\widetilde{z}})$ plane for system \eqref{bb7}}]{
	 			\includegraphics[width=0.2\textwidth]{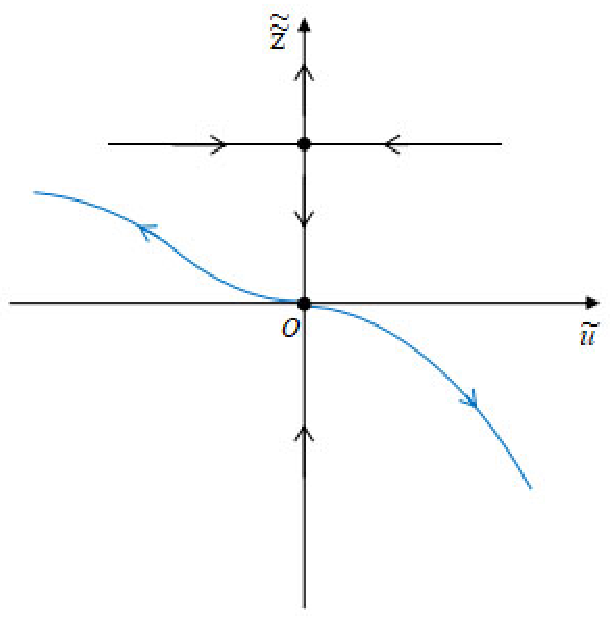}}
	 		\quad
	 		\subfigure[{$(\widetilde{u},\widetilde{z})$  plane for system \eqref{bb6}}]{
	 			\includegraphics[width=0.2\textwidth]{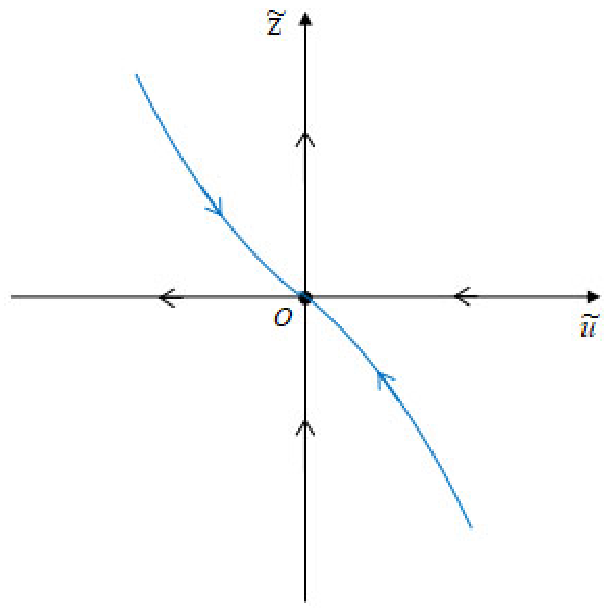}}
	 		\quad
	 		\subfigure[{$(u,\widetilde{z})$ plane for system \eqref{bb5}}]{
	 			\includegraphics[width=0.2\textwidth]{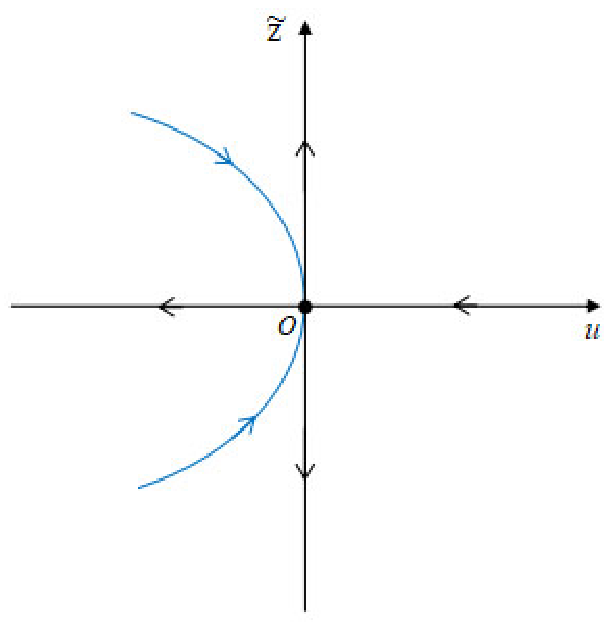}}
	 		\quad
	 		\subfigure[{$(u,z)$ plane for system \eqref{y}}]{
	 			\includegraphics[width=0.2\textwidth]{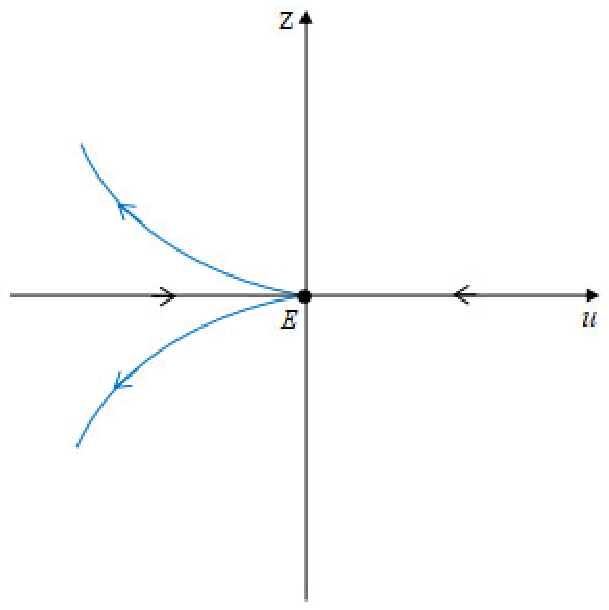}}
	 		\caption{ Orbits changing under Briot--Bouquet transformations when $u_3^2+\mu u_3+\lambda>0$. }\label{tu7}
	 	\end{figure}
	 	\begin{figure}[htp]
	 		\centering
	 		\subfigure[{$(\widetilde{u},\widetilde{\widetilde{z}})$ plane for system \eqref{bb7}}]{
	 			\includegraphics[width=0.2\textwidth]{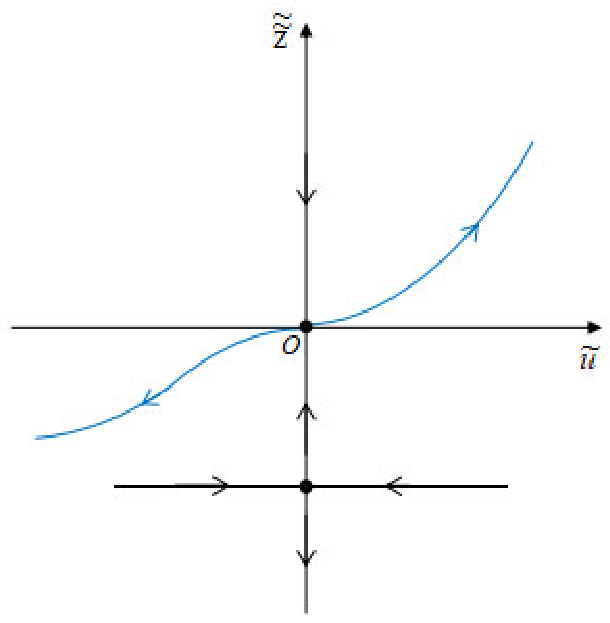}}
	 		\quad
	 		\subfigure[{$(\widetilde{u},\widetilde{z})$ plane for system \eqref{bb6}}]{
	 			\includegraphics[width=0.2\textwidth]{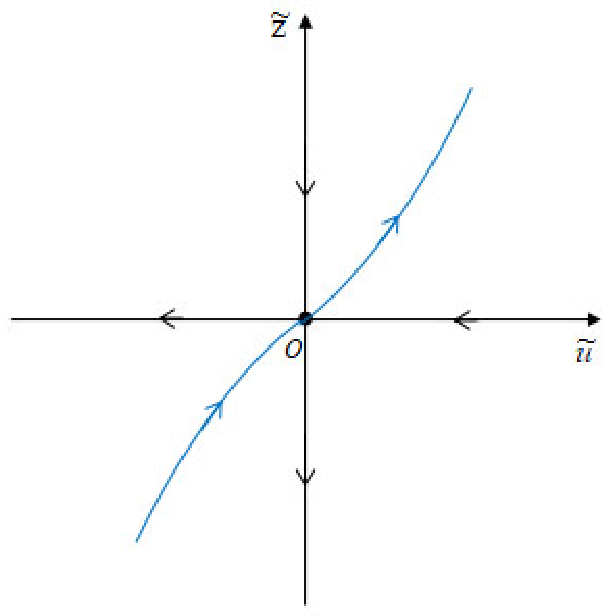}}
	 		\quad
	 		\subfigure[{$(u,\widetilde{z})$ plane for system \eqref{bb5}}]{
	 			\includegraphics[width=0.2\textwidth]{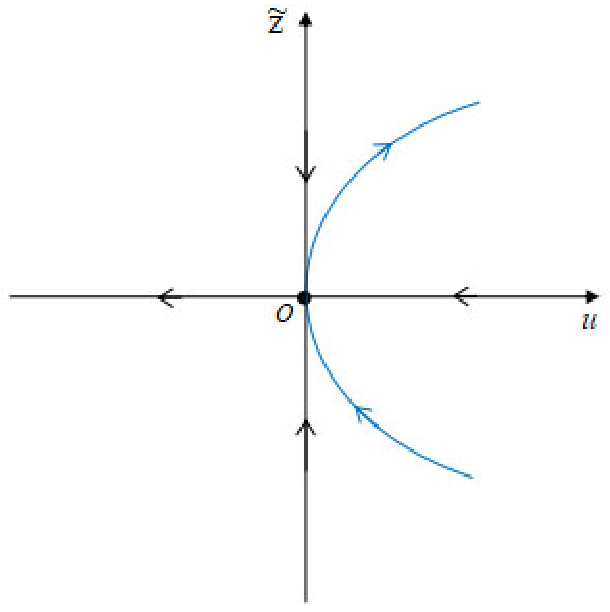}}
	 		\quad
	 		\subfigure[{$(u,z)$ plane for system \eqref{y}}]{
	 			\includegraphics[width=0.2\textwidth]{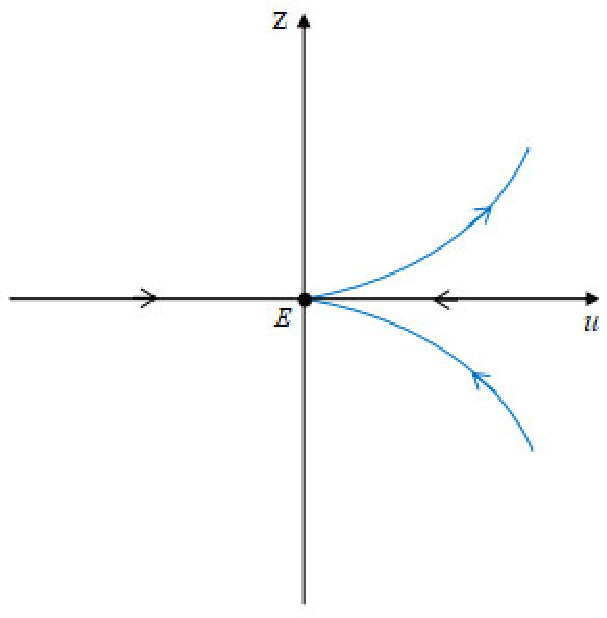}}
	 		\caption{ Orbits changing under Briot--Bouquet transformations when $u_3^2+\mu u_3+\lambda<0$. }\label{tu8}
	 	\end{figure}
	 \end{proof}
	 By Lemma \ref{E},  when $b\in[-\sqrt{3ac},0)\cup(0,\sqrt{3ac}]$, the qualitative properties of equilibria $I_{E^\pm}$  at infinity in the Poincar\'e disc of system \eqref{61}, which correspond the equilibrium $E$ of system \eqref{y},  are as shown in Figure \ref{tu14} {(which illustrate only the local topological structure at $I_{E^+}$ and $I_{E^-}$, there have other equilibria between them, without marked there)}.
	 \begin{figure*}[htp]
	 	\centering
	 	\includegraphics[ height=5cm,width=5cm]{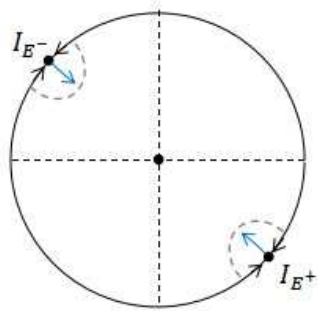}
	 	\caption{{ Locally qualitative property of} equilibria  at  infinity $I_{E^{\pm}}$  in the Poincar\'e disc for $b\in[-\sqrt{3ac},0)\cup(0,\sqrt{3ac}]$ {(other equilibria between them are not marked here).} }
	 	\label{tu14}
	 \end{figure*}

 \begin{lemma}
 	\label{FQ}
 	Consider $b\in\left(-\infty,-\sqrt{3ac}\right)$. Then, $F_1$ and $F_2$ are saddles, and $F_3$ is a stable node  for $\Phi(\varrho_2)<0$; $F$ is a saddle, and $Q$ is a degenerate equilibrium  for $\Phi(\varrho_2)=0$; $E$ is a saddle for $\Phi(\varrho_2)>0$. Moreover,  there is a unique orbit approaching $Q$ along $\theta=0$ as $\tau\to+\infty$,
 	a unique orbit approaching $Q$ along $\theta=\pi$ as $\tau\to-\infty$, and no orbit connecting $Q$ along other directions when  $\Phi(\varrho_2)=0$.
 \end{lemma}

 \begin{proof}	
 	Firstly,	
 	consider $\Phi(\varrho_2)<0$.
 	With the transformation $(u,z)\to (u+u_4,z)$,
 	system \eqref{y} is changed into
 	\begin{equation}
 	\label{u4}\left\{\begin{aligned}
 	\frac{du}{d\tau}=&-z^2\left(u^2+(\mu+2u_4) u+u_4^2+\mu u_4+\lambda\right)-cu^3-(3cu_4+b)u^2
 	\\
 	&	-(3cu_4^2+2bu_4+a)u=:P_5(u,z),
 	\\
 	\frac{dz}{d\tau}=&-z^3u-u_4z^3=:Q_5(u,z).
 	\end{aligned}
 	\right.
 	\end{equation}
 	It is easy to obtain $3cu_4^2+2bu_4+c>0$ by $u_4<\varrho_1$.
 	By the implicit function theorem, $P_5(u,z)=0$
 	has a unique root $u=\phi_5(z)=-(u_4^2+\mu u_4+\lambda)/(3cu_4^2+2bu_4+a)z^2+o(z^2)$ for small $|z|$.
 	Thus,
 	$$
 	Q_5(\phi_5(z),z)=-u_4 z^3+o(z^3).
 	$$
 	By Theorem \ref{thm7.1} of Appendix B and $u_4<0$, $F_1$ is a saddle. It is similar to prove that $F_2$ is a saddle and $F_3$ is a stable node.
 	
 	Secondly,
 	consider  $\Phi(\varrho_2)>0$. It is easy to check $u_3<0$ and $3cu_3^2+2bu_3+a>0$. Considering system \eqref{yyy},
 	by the implicit function theorem, $P_3(u,z)=0$
 	has a unique root $u=\phi_3(z)$ for small $|z|$.
 	Thus,
 	$$
 	Q_3(\phi_3(z),z)=-u_3 z^3+o(z^3).
 	$$
 	By Theorem \ref{thm7.1} of Appendix B, $E$ is a saddle.
 	
 	Finally, consider $\Phi(\varrho_2)=0$. It is obvious that $u_7<0<\varrho_1<\varrho_2$  and the Jacobian matrices at $F$ and $Q$ are
 	$$
 	J_{F}:=	\begin{pmatrix}
 	-3cu_7^2-2bu_7-a&0\\
 	0&0\\
 	\end{pmatrix},
 	\quad	J_{Q}:=	\begin{pmatrix}
 	-3c\varrho_2^2-2b\varrho_2-a&0\\
 	0&0\\
 	\end{pmatrix},
 	$$	
 	respectively.
 	Since   $3cu_7^2+2bu_7+a>0$ and $3c\varrho_2^2+2b\varrho_2+a=0$ by Lemma \ref{lemma}, $F$ is a semi-hyperbolic equilibrium and $Q$ is a degenerate equilibrium.  For $F$, taking the transformation $(u,z)\to (u+u_7, z)$,
 	system \eqref{y} becomes
 	\begin{equation}
 	\label{u7}
 	\left\{\begin{aligned}
 	\frac{du}{d\tau}=&-z^2\left(u^2+(\mu+2u_7) u+u_4^2+\mu u_4+\lambda\right)-cu^3-(3cu_7+b)u^2
 	\\
 	&-(3cu_7^2+2bu_7+a)u=:P_6(u,z),
 	\\
 	\frac{dz}{d\tau}=&-z^3u-u_7z^3=:Q_6(u,z).
 	\end{aligned}
 	\right.
 	\end{equation}
 	By the implicit function theorem, $P_6(u,z)=0$
 	has a unique root $u=\phi_6(z)=-(u_7^2+\mu u_7+\lambda)/(3cu_7^2+2bu_7+a)z^2+o(z^2)$ for small $|z|$.
 	Thus,
 	$$
 	Q_6(\phi_6(z),z)=-u_7 z^3+o(z^3).
 	$$
 	By Theorem \ref{thm7.1} of Appendix B and $u_7<0$, we can check that $F$ is a saddle.
 	
 	Concerning the equilibrium $Q$, a transformation
 	$(u,z) \to(u+\varrho_2, z)$ sends system \eqref{y} to
 	\begin{equation}\label{y4}\left\{\begin{aligned}
 	\frac{du}{d\tau}&=-z^2(u^2+(\mu+2\varrho_2) u+\varrho_2^2+\mu \varrho_2+\lambda)-cu^3-\left(\sqrt{b^2-3ac}\right)u^2,\\
 	\frac{dz}{d\tau}&=-z^3u-\varrho_2z^3.
 	\end{aligned}
 	\right.
 	\end{equation}
 	With a polar transformation $(u,z)=(r\cos\theta,  r\sin\theta)$,  system \eqref{y4} is rewritten as
 	\begin{equation}
 	\label{G}
 	\frac{1}{r}\frac{dr}{d\theta}=\frac{H_7(\theta)+\widetilde{H_7}(\theta,r)}{G_7(\theta)+\widetilde{G_7}(\theta,r)},
 	\end{equation}
 	where
 	$$
 	G_7(\theta)=\sin\theta\left((\varrho_2^2+\mu \varrho_2+\lambda)\sin^2\theta+\left(\sqrt{b^2-3ac}\right)\cos^2\theta\right)
 	$$
 	and $$
 	H_7(\theta)=-\cos\theta\left((\varrho_2^2+\mu\varrho_2+\lambda)\sin^2\theta+\left(\sqrt{b^2-3ac}\right)\cos^2\theta\right).
 	$$
 	It follows from $\varrho_2>0$  that  $\varrho_2^2+\mu \varrho_2+\lambda>0$. Then, we can check that $G_7(\theta)$ has two zeros $\theta=0$, $\pi$ in $[0,2\pi)$ and \[G_7'(0)H_7(0)=G_7'(\pi)H_7(\pi)
 	=3ac-b^2<0.\]
 	Therefore,  there is a unique orbit approaching  $Q$  in the direction $\theta=0$ as $\tau\to+\infty$, and a unique orbit approaching  $Q$  in the direction $\theta=\pi$ as $\tau\to-\infty$ in  system \eqref{y} by   \cite[Theorem 3.7 of Chapter 2 ]{ZDHD},  $H_7(0)<0$ and $H_7(\pi)>0$.
 \end{proof}
 
 By Lemma \ref{FQ},  when $b\in(-\infty,-\sqrt{3ac})$, the qualitative properties of equilibria at infinity  $I_{F_1^\pm}$, $I_{F_2^\pm}$ and $I_{F_3^\pm}$ for $\Phi(\varrho_2)<0$, $I_{F^\pm}$ and $I_{Q^\pm}$ for $\Phi(\varrho_2)=0$, or $I_{E^\pm}$  for $\Phi(\varrho_2)>0$ in the Poincar\'e disc of system \eqref{61}, which correspond respectively the equilibria $F_1$, $F_2$, $F_3$, $F$, $Q$, $E$ of system \eqref{y},  are as shown in Figure \ref{tuFQ} {(there may probably other equilibria at the infinity, which are not marked there)}.
 \begin{figure}[htp]
 	\centering
 	\subfigure[{ $\Phi(\varrho_2)<0$}]{
 		\includegraphics[width=0.28\textwidth]{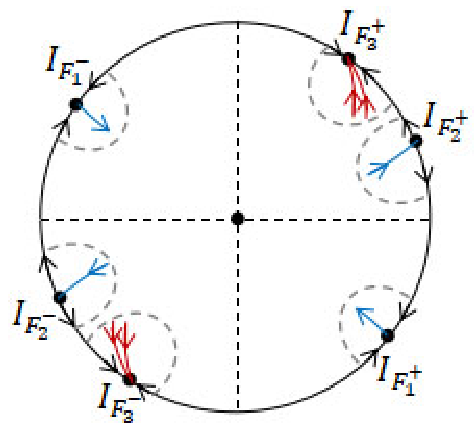}}	
 	\quad
 	\subfigure[{$\Phi(\varrho_2)=0$}]{
 		\includegraphics[width=0.28\textwidth]{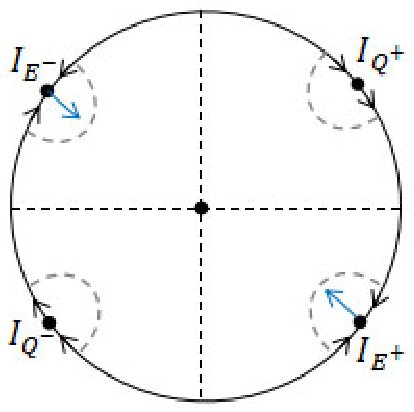}}
 	\quad
 	\subfigure[{  $\Phi(\varrho_2)>0$}]{
 		\includegraphics[width=0.28\textwidth]{29}}
 	\caption{{Locally qualitative property of} equilibria at  infinity  $I_{F_1^\pm}$, $I_{F_2^\pm}$ and $I_{F_3^\pm}$,  or $I_{F^\pm}$ and $I_{Q^\pm}$, or $I_{E^\pm}$  in the Poincar\'e disc for $b\in\left(-\infty, -\sqrt{3ac}\right)$ {(not including all equilibria at infinity)}. }
 	\label{tuFQ}
 \end{figure}

 \begin{lemma}
 	\label{EFK}
 	Consider  $b\in\left(\sqrt{3ac},+\infty\right)$.
 	Then, $E$ is a saddle for $\Phi(\varrho_1)<0$; $F$ is a saddle and $K$ is a degenerate equilibrium for $\Phi(\varrho_1)=0$. When $\Phi(\varrho_1)>0$, $E$ is a saddle for  $\Phi(\varrho_2)>0$;  $F_1$ and $F_3$ are saddles, $F_2$ is an unstable node   for  $\Phi(\varrho_2)<0$;  $F$ is a saddle and  $Q$ is a degenerate equilibrium for  $\Phi(\varrho_2)=0$. Moreover,  the qualitative properties of $K$ or $Q$ are shown in {\rm Tables \ref{K1}--\ref{K3}} or {\rm \ref{Q}--\ref{Q3}}.
 	
 	\begin{table}[htp]
 		\renewcommand\arraystretch{2}
 		\setlength{\tabcolsep}{2.5mm}{
 			\caption{\label{K1} Numbers of orbits connecting $K$ for   $\varrho_1^2+\mu \varrho_1+\lambda<0$.}
 			\begin{tabular}{c|c}
 				\hline
 				Exceptional directions & Numbers of orbits
 				\\
 				\hline
 				$\theta=0$ & one \ $(-)$
 				\\
 				\hline
 				$\theta=\pi$ & one \ $(+)$
 				\\
 				\hline
 		\end{tabular}}
 	\end{table}
 	\begin{table}[htp]
 		\renewcommand\arraystretch{2}
 		\setlength{\tabcolsep}{2.5mm}{
 			\caption{\label{K2} Numbers of orbits connecting $K$ for   $\varrho_1^2+\mu \varrho_1+\lambda=0$.}
 			\begin{tabular}{c|c}
 				\hline
 				Exceptional directions & Numbers of orbits
 				\\
 				\hline
 				$\theta=0$ & one \ $(-)$
 				\\
 				\hline
 				$\theta=\frac{\pi}{2}$ & $\infty$  \ $(-)$
 				\\
 				\hline
 				$\theta=\pi$ & one \ $(+)$
 				\\
 				\hline
 				$\theta=\frac{3\pi}{2}$ &  $\infty$  \ $(-)$
 				\\
 				\hline
 		\end{tabular}}
 	\end{table}
 	\begin{table}[htp]
 		\renewcommand\arraystretch{2}
 		\setlength{\tabcolsep}{2.5mm}{
 			\caption{\label{K3} Numbers of orbits connecting $K$ for    $\varrho_1^2+\mu \varrho_1+\lambda>0$.}
 			\begin{tabular}{c|c}
 				\hline
 				Exceptional directions & Numbers of orbits
 				\\
 				\hline
 				$\theta=0$ & one \ $(-)$
 				\\
 				\hline
 				$\theta=\arctan\left(\sqrt{\frac{\sqrt{b^2-3ac}}{\varrho_1^2+\mu \varrho_1+\lambda}}\right)$ &  one \ $(+)$
 				\\
 				\hline
 				$\theta=\pi-\arctan\left(\sqrt{\frac{\sqrt{b^2-3ac}}{\varrho_1^2+\mu \varrho_1+\lambda}}\right)$ &  $\infty$\ $(+)$
 				\\
 				\hline
 				$\theta=\pi$ & one \ $(+)$
 				\\
 				\hline	
 				$\theta=\pi+\arctan\left(\sqrt{\frac{\sqrt{b^2-3ac}}{\varrho_1^2+\mu \varrho_1+\lambda}}\right)$ & $\infty$\ $(+)$
 				\\
 				\hline
 				$\theta=2\pi-\arctan\left(\sqrt{\frac{\sqrt{b^2-3ac}}{\varrho_1^2+\mu \varrho_1+\lambda}}\right)$ &  one \ $(+)$
 				\\
 				\hline
 		\end{tabular}}
 	\end{table}
 	\begin{table}[htp]
 		\renewcommand\arraystretch{2}
 		\setlength{\tabcolsep}{2.5mm}{
 			\caption{\label{Q} Numbers of orbits connecting $Q$ for  $\varrho_2^2+\mu \varrho_2+\lambda>0$.}
 			\begin{tabular}{c|c}
 				\hline
 				Exceptional directions & Numbers of orbits
 				\\
 				\hline
 				$\theta=0$ & one \ $(+)$
 				\\
 				\hline
 				$\theta=\pi$ & one \ $(-)$
 				\\
 				\hline
 		\end{tabular}}
 	\end{table}
 	\begin{table}[htp]
 		\renewcommand\arraystretch{2}
 		\setlength{\tabcolsep}{2.5mm}{
 			\caption{\label{Q2} Numbers of orbits connecting $Q$ for  $\varrho_2^2+\mu \varrho_2+\lambda=0$.}
 			\begin{tabular}{c|c}
 				\hline
 				Exceptional directions & Numbers of orbits
 				\\
 				\hline
 				$\theta=0$ & one \ $(+)$
 				\\
 				\hline
 				$\theta=\frac{\pi}{2}$ & $\infty$  \ $(-)$
 				\\
 				\hline
 				$\theta=\pi$ & one \ $(-)$
 				\\
 				\hline
 				$\theta=\frac{3\pi}{2}$ &  $\infty$  \ $(-)$
 				\\
 				\hline
 		\end{tabular}}
 	\end{table}
 	\begin{table}[htp]
 		\renewcommand\arraystretch{2}
 		\setlength{\tabcolsep}{2.5mm}{
 			\caption{\label{Q3} Numbers of orbits connecting $Q$ for $\varrho_2^2+\mu \varrho_2+\lambda<0$.}
 			\begin{tabular}{c|c}
 				\hline
 				Exceptional directions & Numbers of orbits
 				\\
 				\hline
 				$\theta=0$ & one \ $(+)$
 				\\
 				\hline
 				$\theta=\arctan\left(\sqrt{\frac{-\sqrt{b^2-3ac}}{\varrho_2^2+\mu \varrho_2+\lambda}}\right)$ & one \ $(-)$
 				\\
 				\hline
 				$\theta=\pi-\arctan\left(\sqrt{\frac{-\sqrt{b^2-3ac}}{\varrho_2^2+\mu \varrho_2+\lambda}}\right)$ &  $\infty$  \ $(-)$
 				\\
 				\hline
 				$\theta=\pi$ & one \ $(-)$
 				\\
 				\hline	
 				$\theta=\pi+\arctan\left(\sqrt{\frac{-\sqrt{b^2-3ac}}{\varrho_2^2+\mu \varrho_2+\lambda}}\right)$ & $\infty$ \ $(-)$
 				\\
 				\hline
 				$\theta=2\pi-\arctan\left(\sqrt{\frac{-\sqrt{b^2-3ac}}{\varrho_2^2+\mu \varrho_2+\lambda}}\right)$ &  one \ $(-)$
 				\\
 				\hline
 		\end{tabular}}
 	\end{table}
 \end{lemma}
 \begin{proof}
 	Firstly,
 	consider $\Phi(\varrho_1)<0$. It is easy to check $\varrho_2<u_3<0$ and $3cu_3^2+2bu_3+a>0$.	 For $E$, considering system \eqref{yyy}, 	by the implicit function theorem, $P_3(u,z)=0$
 	has a unique root $u=\phi_3(z)=-(u_3^2+\mu u_3+\lambda)/(3cu_3^2+2bu_3+a)z^2+o(z^2)$ for small $|z|$.
 	Thus,
 	$$
 	Q_3(\phi_3(z),z)=-u_3 z^3+o(z^3).
 	$$
 	By Theorem \ref{thm7.1} of Appendix B, $E$ is a saddle.
 	
 	Secondly,	
 	consider $\Phi(\varrho_1)=0$. Notice that $\varrho_2<u_7<0$ and $3cu_7^2+2bu_7+a>0$.  Considering system \eqref{u7},
 	by the implicit function theorem, $P_6(u,z)=0$
 	has a unique root $u=\phi_6(z)=-(u_7^2+\mu u_7+\lambda)/(3cu_7^2+2bu_7+a)z^2+o(z^2)$ for small $|z|$.
 	Thus,
 	$$
 	Q_6(\phi_6(z),z)=-u_7 z^3+o(z^3).
 	$$
 	Thus,	$F$ is a saddle by Theorem \ref{thm7.1} of Appendix B.
 	
 	Concerning the degenerate equilibrium $K$, 	 we can rewrite system  \eqref{y3} by the transformation
 	$
 	(u,z)\to (u+\varrho_1, z)$  as
 	\begin{equation}
 	\label{y3}
 	\left\{\begin{aligned}
 	\frac{du}{d\tau}&=-z^2(u^2+(\mu+2\varrho_1) u+\varrho_1^2+\mu \varrho_1+\lambda)-cu^3+\left(\sqrt{b^2-3ac}\right) u^2,\\
 	\frac{dz}{d\tau}&=-z^3u-\varrho_1z^3.
 	\end{aligned}
 	\right.
 	\end{equation}
 	Considering a polar transformation $(u,z)=(r\cos\theta,r\sin\theta)$, 	
 	system  \eqref{y3} is changed into
 	\begin{equation}
 	\notag
 	\frac{1}{r}\frac{dr}{d\theta}=\frac{H_8(\theta)+\widetilde{H_8}(\theta,r)}{G_8(\theta)+\widetilde{G_8}(\theta,r)},
 	\end{equation}
 	where $G_8(\theta)=\sin\theta[(\varrho_1^2+\mu \varrho_1+\lambda)\sin^2\theta-(\sqrt{b^2-3ac})\cos^2\theta]$, $H_8(\theta)=-\cos\theta[(\varrho_1^2+\mu \varrho_1+\lambda)\sin^2\theta-(\sqrt{b^2-3ac})\cos^2\theta]$.
 	As a similar progress by discussing the sign of $\varrho_1^2+\mu \varrho_1+\lambda$ in Lemma \ref{lem10}, we can obtain that there is a unique orbit  approaching  $K$ in the direction $\theta = \pi$ as $\tau\to+\infty$,  and 	a unique orbit approaching $K$  in the direction $\theta = 0$  as $\tau\to-\infty$ for $\varrho_1^2+\mu \varrho_1+\lambda<0$;
 	there is a unique orbit  approaching  $K$ in the direction $\theta = \pi$ as $\tau\to+\infty$, 	a unique orbit approaching $K$  in the direction $\theta = 0$  as $\tau\to-\infty$, infinitely many orbits approaching $K$ in  respectively the  directions $\pi/2$ and $3\pi/2$ as $\tau\to-\infty$ for $\varrho_1^2+\mu \varrho_1+\lambda=0$;
 	there is a unique orbit  approaching  $K$ in  respectively the directions
 	$$
 	\theta = \pi,
 	\quad \arctan\left(\sqrt{\frac{\sqrt{b^2-3ac}}{\varrho_1^2+\mu \varrho_1+\lambda}}\right),
 	\quad 2\pi-\arctan\left(\sqrt{\frac{\sqrt{b^2-3ac}}{\varrho_1^2+\mu \varrho_1+\lambda}}\right)
 	$$
 	as $\tau\to+\infty$, a unique orbit approaching $K$  in the direction $\theta = 0$  as $\tau\to-\infty$,  infinitely many orbits approaching $K$  in  respectively the directions
 	$$
 	\theta =\pi-\arctan\left(\sqrt{\frac{\sqrt{b^2-3ac}}{\varrho_1^2+\mu \varrho_1+\lambda}}\right),
 	\quad \pi+\arctan\left(\sqrt{\frac{\sqrt{b^2-3ac}}{\varrho_1^2+\mu \varrho_1+\lambda}}\right)$$
 	as $\tau\to+\infty$
 	for $\varrho_1^2+\mu \varrho_1+\lambda>0$, see Tables \ref{K1}--\ref{K3}.
 	
 	Finally, consider $\Phi(\varrho_1)>0$. Then, we need to discuss
 	the sign of  $\Phi(\varrho_2)$, which makes our study divided in three subcases.
 	When $\Phi(\varrho_2)>0$,
 	considering system \eqref{yyy}, $E$ is a saddle by Theorem \ref{thm7.1} of Appendix B, $u_3<0$ and $3cu_3^2+2bu_3+a>0$.
 	When $\Phi(\varrho_2)<0$, considering system \eqref{u4}, $F_1$ is a saddle by Theorem \ref{thm7.1} of Appendix B, $u_4<0$ and $3cu_4^2+2bu_4+a>0$. Similarly, $F_2$ is an unstable node and $F_3$ is a saddle.
 	
 	When $\Phi(\varrho_2)=0$, it is easy to check that  $u_7<\varrho_1<\varrho_2<0$ and $3cu_7^2+2bu_7+a>0$. Considering system \eqref{u7}, $F$ is a saddle by Theorem \ref{thm7.1} of Appendix B.	
 	Concerning the degenerate equilibrium $Q$,  we consider system \eqref{y4} and equation \eqref{G}. By the expression of $G_7(\theta)=\sin\theta[(\varrho_2^2+\mu \varrho_2+\lambda)\sin^2\theta+(\sqrt{b^2-3ac})\cos^2\theta]$, the sign of $\varrho_2^2+\mu \varrho_2+\lambda$ determines the zeros of $G_7(\theta)$.
 	As a similar progress  in Lemma \ref{lem10}, we can obtain that there is a unique orbit  approaching  $Q$ in the direction $\theta = \pi$ as $\tau\to-\infty$,  and 	a unique orbit approaching $Q$  in the direction $\theta = 0$  as $\tau\to+\infty$ for $\varrho_2^2+\mu \varrho_2+\lambda>0$;
 	there is a unique orbit  approaching  $Q$ in the direction $\theta = \pi$ as $\tau\to-\infty$, 	a unique orbit approaching $Q$  in the direction $\theta = 0$  as $\tau\to+\infty$, infinitely many orbits approaching $Q$ in the  respectively directions $\pi/2$ and $3\pi/2$ as $\tau\to-\infty$ for $\varrho_2^2+\mu \varrho_2+\lambda=0$;
 	there is a unique orbit  approaching  $Q$ in  respectively the directions
 	$$
 	\theta = \pi, \quad \arctan\left(\sqrt{\frac{-\sqrt{b^2-3ac}}{\varrho_2^2+\mu \varrho_2+\lambda}}\right),\quad 2\pi-\arctan\left(\sqrt{\frac{-\sqrt{b^2-3ac}}{\varrho_2^2+\mu \varrho_2+\lambda}}\right)
 	$$ as $\tau\to-\infty$, a unique orbit approaching $Q$  in the direction $\theta = 0$  as $\tau\to+\infty$,  infinitely many orbits approaching $Q$  in the  respectively directions $$\theta =\pi-\arctan\left(\sqrt{\frac{-\sqrt{b^2-3ac}}{\varrho_2^2+\mu \varrho_2+\lambda}}\right),\quad \pi+\arctan\left(\sqrt{\frac{-\sqrt{b^2-3ac}}{\varrho_2^2+\mu \varrho_2+\lambda}}\right)$$
 	as $\tau\to-\infty$
 	for $\varrho_2^2+\mu \varrho_2+\lambda<0$, see Tables \ref{Q}-\ref{Q3}.
 \end{proof}

 By Lemma \ref{EFK},  when $b\in(\sqrt{3ac}, +\infty)$, the qualitative properties    of equilibria at infinity $I_{E^\pm}$ for $\Phi(\varrho_1)<0$, or $\Phi(\varrho_1)>0$ and $\Phi(\varrho_2)>0$, $I_{K^\pm}$ and $I_{F^\pm}$ for $\Phi(\varrho_1)=0$, $I_{Q^\pm}$ and $I_{F^\pm}$ for $\Phi(\varrho_1)>0$ and $\Phi(\varrho_2)=0$,  $I_{F_1^\pm}$, $I_{F_2^\pm}$ and $I_{F_3^\pm}$ for $\Phi(\varrho_1)>0$ and $\Phi(\varrho_2)<0$ in the Poincar\'e disc of system \eqref{61}, which correspond respectively the equilibria $E$, $K$, $F$, $Q$, $F_1$, $F_2$, $F_3$ of system \eqref{y},  are as shown in Figure \ref{tuEFK}.
 \begin{figure}[htp]
 	\centering
 	\subfigure[{$\Phi(\varrho_1)=0$, $\varrho_1^2+\mu \varrho_1+\lambda<0$}]{
 		\includegraphics[width=0.3\textwidth]{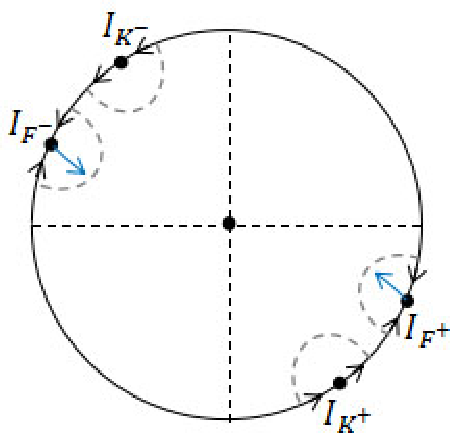}}
 	\quad
 	\subfigure[{$\Phi(\varrho_1)=0$, $\varrho_1^2+\mu \varrho_1+\lambda=0$}]{
 		\includegraphics[width=0.3\textwidth]{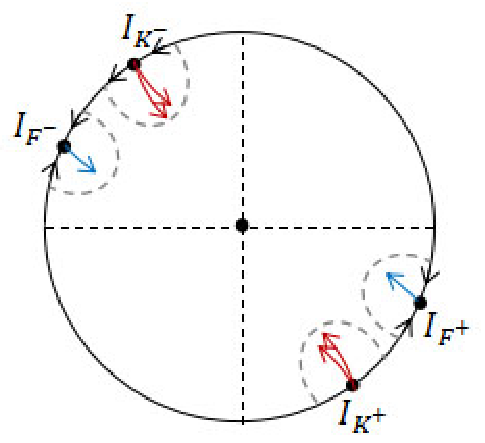}}
 	\quad
 	\subfigure[{$\Phi(\varrho_1)=0$, $\varrho_1^2+\mu \varrho_1+\lambda>0$}]{
 		\includegraphics[width=0.3\textwidth]{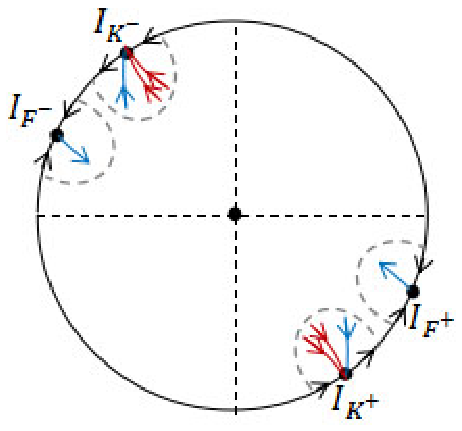}}
 	\centering
 	\subfigure[{$\Phi(\varrho_1)>0$, $\Phi(\varrho_2)=0$, $\varrho_2^2+\mu \varrho_2+\lambda>0$}]{
 		\includegraphics[width=0.3\textwidth]{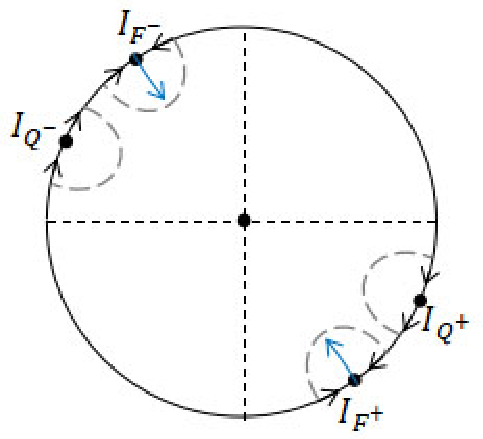}}
 	\quad
 	\subfigure[{$\Phi(\varrho_1)>0$, $\Phi(\varrho_2)=0$, $\varrho_2^2+\mu \varrho_2+\lambda=0$}]{
 		\includegraphics[width=0.3\textwidth]{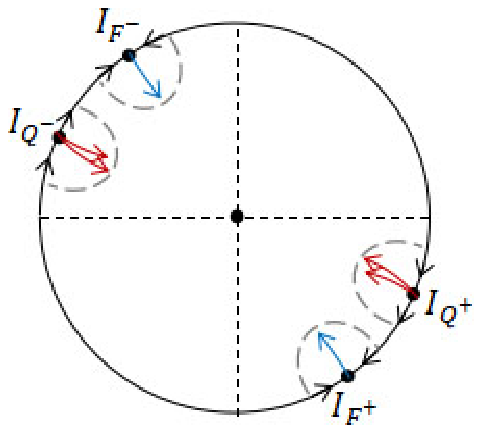}}
 	\quad
 	\subfigure[{$\Phi(\varrho_1)>0$, $\Phi(\varrho_2)=0$, $\varrho_2^2+\mu \varrho_2+\lambda<0$}]{
 		\includegraphics[width=0.3\textwidth]{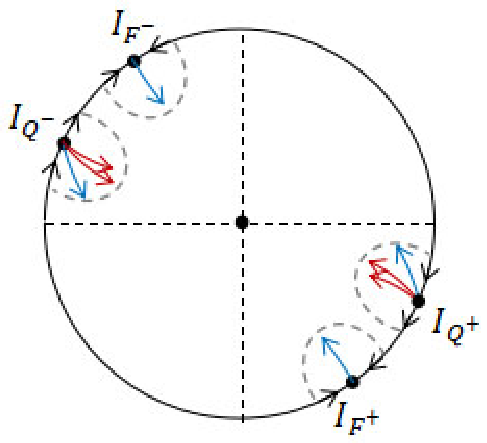}}		 													
 	\subfigure[{$\Phi(\varrho_1)<0$, or $\Phi(\varrho_1)>0$ and $\Phi(\varrho_2)>0$ }]{
 		\includegraphics[width=0.3\textwidth]{29}}
 	\quad
 	\subfigure[{ $\Phi(\varrho_1)>0$, $\Phi(\varrho_2)<0$ }]{
 		\includegraphics[width=0.3\textwidth]{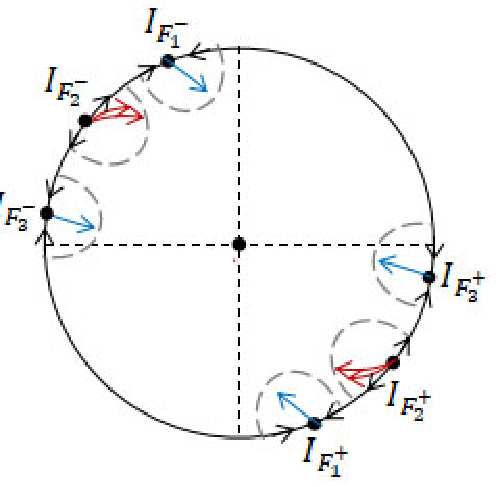}}	
 	\caption{ {Locally qualitative property of}	
 		equilibria at  infinity  $I_{E^{\pm}}$, $I_{K^\pm}$ and $I_{F^\pm}$, or $I_{Q^\pm}$ and $I_{F^\pm}$, or  $I_{F_1^\pm}$, $I_{F_2^\pm}$ and $I_{F_3^\pm}$  in the Poincar\'e disc for $b\in\left(\sqrt{3ac}, +\infty\right)$ {(some equilibria at infinity are not marked here)}.}
 	\label{tuEFK}
 \end{figure}

 With the other Poincar\'e transformation
 $$
 x=\frac{v}{z},\quad y=\frac{1}{z},
 $$
 system \eqref{61} is written as
 \begin{equation}
 \label{x}
 \left\{\begin{aligned}
 \frac{dv}{d\tau}&=z^2(\lambda v^2+\mu v+1)+v(v^3+av^2+bv+c),\\
 \frac{dz}{d\tau}&=z^3(\lambda v+\mu)+z(v^3+av^2+bv+c),
 \end{aligned}
 \right.
 \end{equation}
 where $d\tau=dt/z^2$.
 By   \cite[Chapter 5]{ZDHD}, we only need to study the equilibrium $D=(0, 0)$
 of system \eqref{x}.
 \begin{lemma}\label{lemx}
 	$D$ is an unstable star node for $c>0$, and a degenerate equilibrium for $c=0$. Moreover, the qualitative properties of $D$ is shown  in  {\rm Tables  \ref{D1}--\ref{D2}} for $c=0$.
 	\begin{table}[htp]
 		\renewcommand\arraystretch{2}
 		\setlength{\tabcolsep}{2.5mm}{
 			\caption{\label{D1} Numbers of orbits connecting $D$ for  $b>0$.}
 			\begin{tabular}{c|c}
 				\hline
 				Exceptional directions & Numbers of orbits
 				\\
 				\hline
 				$\theta=0$ & one \ $(-)$
 				\\
 				\hline
 				$\theta=\pi$ & one \ $(+)$
 				\\
 				\hline
 		\end{tabular}}
 	\end{table}
 	\begin{table}[]
 		\renewcommand\arraystretch{2}
 		\setlength{\tabcolsep}{2.5mm}{
 			\caption{\label{D2} Numbers of orbits connecting $D$ for  $b<0$.}
 			\begin{tabular}{c|c}
 				\hline
 				Exceptional directions & Numbers of orbits
 				\\
 				\hline
 				$\theta=0$ & $\infty$\ $(+)$
 				\\
 				\hline
 				$\theta=\pi$ & $\infty$\ $(-)$
 				\\
 				\hline
 		\end{tabular}}
 	\end{table}
 \end{lemma}
 
 \begin{proof}
 	It is easy to  prove that $D$ is an unstable star node for $c>0$.
 	
 	Consider $c=0$. In the polar coordinate $(v,z)=(r\cos\theta,r\sin\theta)$, system \eqref{x} is transformated into
 	\begin{equation}
 	\notag
 	\frac{1}{r}\frac{dr}{d\theta}=\frac{H_9(\theta)+\widetilde{H_9}(\theta,r)}{G_9(\theta)+\widetilde{G_9}(\theta,r)},
 	\end{equation}
 	where $G_9(\theta)=-\sin^3\theta$, $H_9(\theta)=(b+1)\sin^2\theta\cos\theta+b\cos^3\theta$.  $\theta=0$, $\pi$  are zeros of three-multiple  of $G_9(\theta)$.  It is easy to check  $G_9'''(0)H_9(0)=G_9'''(\pi)H_9(\pi)=-6b<0$ for $b>0$, and $G_9'''(0)H_9(0)=G_9'''(\pi)H_9(\pi)=-6b>0$ for $b<0$. By  \cite[Theorems 3.7 and 3.8 of Chapter 2]{ZDHD}, $H_9(0)=b$ and $H_9(\pi)=-b$, system \eqref{x} has
 	a unique orbit approaching $D$ in the direction $\theta = 0$ as $\tau\to-\infty$, and a unique orbit  approaching $D$ in the direction $\theta =\pi$ as $\tau\to+\infty$ for $b>0$;   infinitely many orbits approaching  $D$ in the direction $\theta = 0$ as $\tau\to+\infty$, and infinitely many orbits approaching $D$ in the direction $\theta  = \pi$ as $\tau\to-\infty$ for $b<0$.
 \end{proof}
 By  Lemma \ref{lemx}, the qualitative properties of equilibria $I_{D^\pm}$  at infinity in the Poincar\'e disc of system \eqref{61}, which correspond the equilibrium $D$ of system \eqref{x},  are as shown in Figure 	\ref{tu19}.
 \begin{figure}
 	\centering
 	\subfigure[{$c=0$, $b>0$ }]{
 		\includegraphics[width=0.3\textwidth]{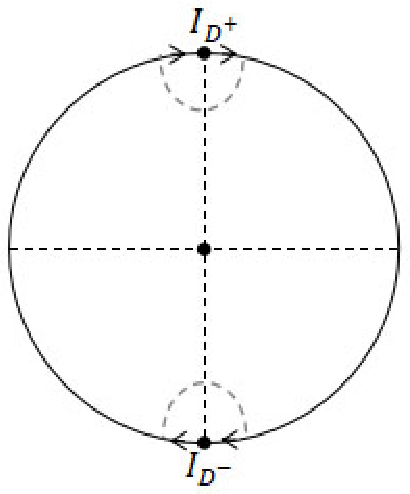}}
 	\quad
 	\subfigure[{$c=0$, $b<0$}]{
 		\includegraphics[width=0.3\textwidth]{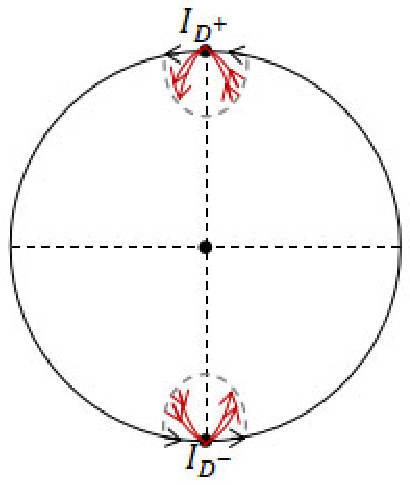}}
 	\quad
 	\subfigure[{ $c>0$}]{
 		\includegraphics[width=0.3\textwidth]{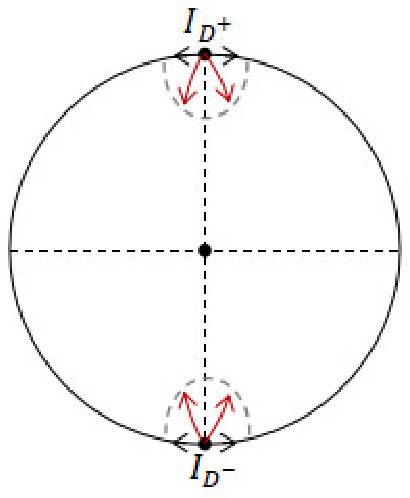}}
 	\caption{{Locally qualitative property of} equilibria at infinity $I_{D^{\pm}}$   in the Poincar\'e disc {(some equilibria at infinity are not marked here)}. }\label{tu19}
 \end{figure}

 	\section*{Appendix D}

 	By a Poincar\'e transformation
 	$$
 	x=\frac{1}{z},\quad y=\frac{u}{z},
 	$$
 	system \eqref{71} is changed into
 	\begin{equation}\label{Y}
 	\left\{\begin{aligned}
 	\frac{du}{d\tau}&=-z^2(u^2+\mu u+1)-u(cu^2+bu+a)=:P_7(u,z),\\
 	\frac{dz}{d\tau}&=-z^3u=:Q_7(u,z),
 	\end{aligned}
 	\right.
 	\end{equation}
 	where $d\tau=dt/z^2$. The equilibria of system \eqref{Y}  at $z=0$ are shown in Table \ref{T6}.
 	\begin{table}[htp]\normalsize
 		\renewcommand\arraystretch{2}
 		\setlength{\tabcolsep}{3mm}{
 			\caption{ \label{T6} Equilibria of system \eqref{Y} at $z=0$}
 			\begin{tabular}{c|c|c|c}
 				\hline
 				\multicolumn{3}{c|}
 				{ Relations between $a$ and $c$} & Equilibria
 				\\
 				\hline
 				\multirow{2}{*}
 				{\begin{tabular}[c]{@{}c@{}}$c=0$\end{tabular}} & \multicolumn{2}{c|}{$a=0$} & $G:(0,0)$
 				\\
 				\cline{2-4}
 				& \multicolumn{2}{c|}{$a>0$} & $G:(0,0)$, $R:\left(-\frac{a}{b},0\right)$
 				\\
 				\hline
 				\multirow{4}{*}
 				{\begin{tabular}[c]{@{}c@{}} $c>0$\end{tabular}} & \multicolumn{2}{c|}{$a=0$} & $G:(0,0)$, $S:\left(-\frac{b}{c},0\right)$
 				\\
 				\cline{2-4}
 				& \multirow{3}{*}{\begin{tabular}[c]{@{}c@{}} $a>0$\end{tabular}} & $b^2-4ac<0$ & $G: (0,0)$
 				\\
 				\cline{3-4}
 				&  & $b^2-4ac=0$ & $G:(0,0)$, $T:\left(-\frac{b}{2c},0\right)$
 				\\
 				\cline{3-4}
 				&  & $b^2-4ac>0$ & $G:(0,0)$, $P_1:\left(\frac{-b+\sqrt{b^2-4ac}}{2c},0\right)$, $P_2:\left(\frac{-b-\sqrt{b^2-4ac}}{2c},0\right)$
 				\\
 				\hline
 		\end{tabular}}
 	\end{table}
 	
 	\begin{lemma}
 		\label{lem17}
 		Consider $a=c=0$. $G$ is a degenerate equilibrium. Moreover, the qualitative properties of $G$ is shown  in {\rm Tables \ref{G3}--\ref{G2}}.
 		\begin{table}
 			\renewcommand\arraystretch{2}
 			\setlength{\tabcolsep}{2.5mm}{
 				\caption{\label{G3} Numbers of orbits connecting $G$ for  $b>0$ when $a=c=0$.}
 				\begin{tabular}{c|c}
 					\hline
 					Exceptional directions & Numbers of orbits
 					\\
 					\hline
 					$\theta=0$ & one \ $(+)$
 					\\
 					\hline
 					$\theta=\pi$ & one \ $(-)$
 					\\
 					\hline
 			\end{tabular}}
 		\end{table}
 		\begin{table}
 			\renewcommand\arraystretch{2}
 			\setlength{\tabcolsep}{2.5mm}{
 				\caption{\label{G2} Numbers of orbits connecting $G$ for  $b<0$ when $a=c=0$.}
 				\begin{tabular}{c|c}
 					\hline
 					Exceptional directions & Numbers of orbits
 					\\
 					\hline
 					$\theta=0$ & one \ $(-)$
 					\\
 					\hline
 					$\theta=\arctan\left(\sqrt{-b}\right)$ &  one \ $(+)$
 					\\
 					\hline
 					$\theta=\pi-\arctan\left(\sqrt{-b}\right)$ &  one \ $(-)$
 					\\
 					\hline
 					$\theta=\pi$ & one \ $(+)$
 					\\
 					\hline	
 					$\theta=\pi+\arctan\left(\sqrt{-b}\right)$ & one \ $(-)$
 					\\
 					\hline
 					$\theta=2\pi-\arctan\left(\sqrt{-b}\right)$ &  one \ $(+)$
 					\\
 					\hline
 			\end{tabular}}
 		\end{table}
 	\end{lemma}
 	\begin{figure}
 		\centering
 		\subfigure[{$u$-$\widetilde{z}$ plane for system \eqref{BY1} }]{
 			\includegraphics[width=0.35\textwidth]{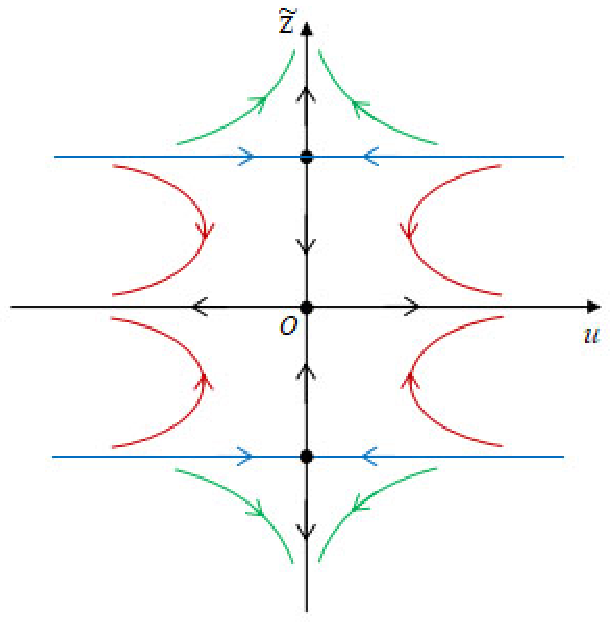}}
 		\quad
 		\subfigure[{$u$-$z$ plane for system \eqref{Y} } ]{
 			\includegraphics[width=0.35\textwidth]{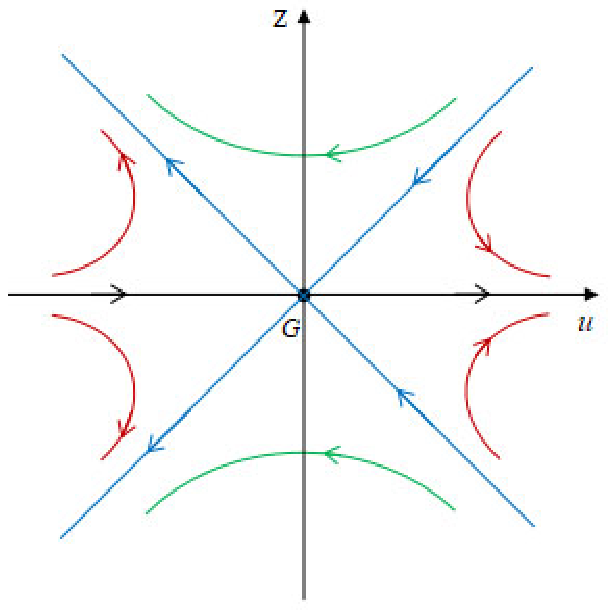}}
 		\caption{ Orbits changing under Briot--Bouquet transformations when $a=c=0$ and $b<0$. }\label{tu20}
 	\end{figure}
 	\begin{proof}
 		Considering a polar transformation $(u,z)=(r\cos\theta,r\sin\theta)$,
 		system \eqref{Y} can be rewritten as
 		\begin{equation}
 		\label{G10}
 		\frac{1}{r}\frac{dr}{d\theta}=\frac{H_{10}(\theta)+\widetilde{H_{10}}(\theta,r)}{G_{10}(\theta)+\widetilde{G_{10}}(\theta,r)},
 		\end{equation}
 		where $G_{10}(\theta)=\sin\theta(b\cos^2\theta+\sin^2\theta)$ and $H_{10}(\theta)=-\cos\theta(b\cos^2\theta+\sin^2\theta)$.	 Clearly, the zeros of $G_{10}(\theta)$ is related to   the sign of $b$.
 		
 		When $b>0$, $G_{10}(\theta)$ has two zeros   $\theta=0$, $\pi$ in $[0,2\pi)$. It is clear that $G_{10}'(0)H_{10}(0)=G_{10}'(\pi)H_{10}(\pi)=-b^2<0$.
 		By   \cite[Theorem 3.7 of Chapter 2]{ZDHD},  $H_{10}(0)=-b<0$ and $H_{10}(\pi)=b>0$, system \eqref{Y} has
 		a unique orbit approaching  $G$ in the direction $\theta = \pi$ as $\tau\to-\infty$,  and 	a unique orbit approaching $G$  in the direction $\theta = 0$  as $\tau\to+\infty$.
 		
 		When $b<0$, $G_{10}(\theta)$ has six zeros  $\theta=0$, $\arctan\left(\sqrt{-b}\right)$, $\pi-\arctan \left(\sqrt{-b}\right)$, $\pi$, $ \pi+\arctan\left(\sqrt{-b}\right)$, $2\pi-\arctan\left(\sqrt{-b}\right)$ in $[0,2\pi)$. It is easy to get $G_{10}'(0)H_{10}(0)=G_{10}'(\pi)H_{10}(\pi)=-b^2<0$.  Similarly by
 		\cite[Theorem 3.7 of Chapter 2]{ZDHD},  $H_{10}(0)>0$ and $H_{10}(\pi)<0$,  system \eqref{Y} has
 		a unique orbit approaching  $G$ in the direction $\theta = \pi$ as $\tau\to+\infty$,  and 	a unique orbit approaching $G$  in the direction $\theta = 0$  as $\tau\to-\infty$.
 		However, we need to blow up the other four directions because   $H_{10}\left(\arctan\left(\sqrt{-b}\right)\right)=H_{10}\left(\pi-\arctan \left(\sqrt{-b}\right)\right)=H_{10}\left(\pi+\arctan \left(\sqrt{-b}\right)\right)=H_{10}\left(2\pi-\arctan \left(\sqrt{-b}\right)\right)=0$.
 		
 		With the
 		Briot--Bouquet transformation
 		$z=\widetilde{z}u$,
 		system \eqref{Y} is changed into
 		\begin{equation}\label{BY1}
 		\left\{\begin{aligned}
 		\frac{du}{d\delta}&=-\widetilde{z}^2u(u^2+\mu u+1)-bu,\\
 		\frac{d\widetilde{z}}{d\delta}&=\widetilde{z}^3(\mu u+1)+b\widetilde{z},
 		\end{aligned}
 		\right.
 		\end{equation}
 		where $d\delta=ud\tau$.  System \eqref{BY1}  has three  equilibria $(0,0)$, $\left(0,\sqrt{-b}\right)$ and $\left(0,-\sqrt{-b}\right)$.  It is sure that $(0,0)$ is a saddle and the other  two are semi--degenerate equilibria  of system \eqref{BY1}.
 		Moreover,
 		a  transformation $(u,\widetilde{z})=(u,\widetilde{z}+\sqrt{-b})$
 		yields that
 		system
 		\eqref{BY1}  is changed into
 		\begin{equation}\label{BY2}
 		\left\{\begin{aligned}
 		\frac{du}{d\delta}&=-\widetilde{z}^2u(u^2+\mu u+1)-2\sqrt{-b}\widetilde{z}u(u^2+\mu u+1)+b(u^3+\mu u^2),\\
 		\frac{d\widetilde{z}}{d\delta}&=-\mu b\sqrt{-b}u-2b\widetilde{z}-3\mu b\widetilde{z}u+(\widetilde{z}^3+3\sqrt{-b}\widetilde{z}^2)(\mu u+1).
 		\end{aligned}
 		\right.
 		\end{equation}	
 		A  further transformation $(u,\widetilde{z}) \to\left(u, \left(\widetilde{z}+\mu b\sqrt{-b}u\right)/(-2b) \right)$
 		sends system \eqref{BY2}   to
 		\begin{equation}\label{BY3}
 		\left\{\begin{aligned}
 		\frac{du}{d\delta}&=\left(-\frac{3\mu^2b}{4}+b\right)u^3+
 		\frac{\sqrt{-b}}{b}\widetilde{z}u+h.o.t.=:P_8(u,\widetilde{z}),\\
 		\frac{d\widetilde{z}}{d\delta}&=-2b\widetilde{z}-
 		\frac{3\mu^2b^2\sqrt{-b}}{2}u^2+h.o.t.=:Q_8(u,\widetilde{z}).
 		\end{aligned}
 		\right.
 		\end{equation}	
 		By the implicit function theorem, $Q_8(u,\widetilde{z})=0$
 		has a unique root $$\widetilde{z}=\phi_8(u)=-\frac{3\mu^2b\sqrt{-b}}{4}u^2+o(u^2)$$ for small $|u|$.
 		Thus,
 		$P_8(u,\phi_8(u))=bu^3+o(u^3)$.
 		By Theorem \ref{thm7.1} of Appendix B and $b<0$, the orgin of system \eqref{BY3} is a saddle. So is $\left(0,\sqrt{-b}\right)$  for system \eqref{BY1}.
 		Similarly,  $\left(0,-\sqrt{-b}\right)$ is  a saddle of system \eqref{BY1}, see Figure \ref{tu20}(a).
 		Further, we obtain the  qualitative properties of $G$ in the $(u, z)$ plane for system \eqref{Y}, as shown in Figure \ref{tu20}(b). The proof is finished.
 	\end{proof}
 	By Lemma \ref{lem17},  when $a=c=0$, the qualitative properties of equilibria $I_{G^\pm}$ at infinity in the Poincar\'e disc of system \eqref{71}, which correspond the equilibrium $G$ of system \eqref{Y},  are as shown in Figure \ref{tuG}.
 	\begin{figure}[htp]
 		\centering
 		\subfigure[{ $b>0$ }]{
 			\includegraphics[width=0.3\textwidth]{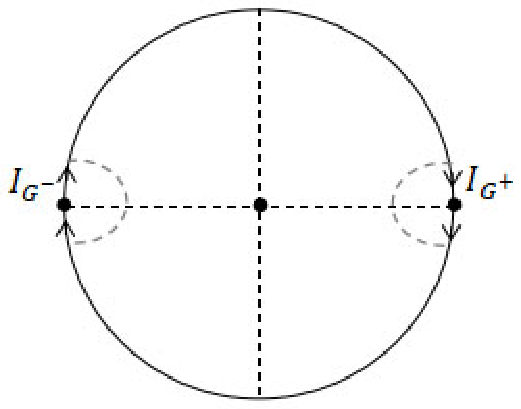}}
 		\quad
 		\subfigure[{$b<0$}]{
 			\includegraphics[width=0.3\textwidth]{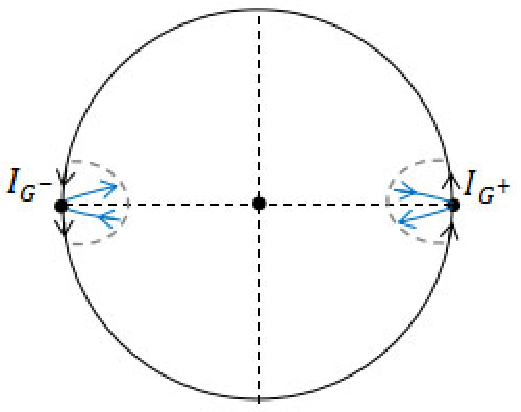}}
 		\caption{{Locally qualitative property of} equilibria at infinity  $I_{G^{\pm}}$ in the Poincar\'e disc when $a=c=0$. }\label{tuG}
 	\end{figure}
 	
 	\begin{lemma}
 		\label{lem18}
 		Consider $c=0$ and $a>0$. $G$ is a saddle. $R$ is a saddle for $b<0$, and an unstable node for $b>0$.
 	\end{lemma}
 	\begin{proof}
 		The Jacobian matrices at $G$, $R$ are
 		$$
 		J_G:=	
 		\begin{pmatrix}
 		-a&0\\
 		0&0\\
 		\end{pmatrix},
 		\quad
 		J_R:=	
 		\begin{pmatrix}
 		a&0\\
 		0&0\\
 		\end{pmatrix}
 		$$
 		respectively.
 		Considering system \eqref{Y},
 		by the implicit function theorem, $P_7(u,z)=0$
 		has a unique root $u=\phi_7(z)=-1/az^2+o(z^2)$ for small $|z|$.
 		Thus,
 		$$
 		Q_7(\phi_7(z),z)=\frac{1}{a}z^5+o(z^5).
 		$$
 		By Theorem \ref{thm7.1} of Appendix B, $G$ is  a saddle.
 		For $R$,
 		with a transformation $(u,z)\to (u-a/b,z)$, system \eqref{Y} is changed into \begin{equation}\notag
 		\left\{\begin{aligned}
 		\frac{du}{d\tau}&=-z^2\left(u^2+\left(\mu-\frac{2a}{b}\right)u+\frac{a^2}{b^2}-\frac{a\mu}{b}+1\right)-bu^2+au=:P_9(u,z),\\
 		\frac{dz}{d\tau}&=\frac{a}{b}z^3-z^3u=:Q_9(u,z).
 		\end{aligned}
 		\right.
 		\end{equation}
 		By the implicit function theorem, $P_9(u,z)=0$
 		has a unique root $u=\phi_9(z)=(a^2-\mu ab+b^2)/(ab^2)z^2+o(z^2)$ for small $|z|$.
 		Thus,
 		$$
 		Q_9(\phi_9(z),z)=\frac{a}{b}z^3+o(z^3).
 		$$
 		By Theorem \ref{thm7.1} of Appendix B,
 		$R$ is a saddle for $b<0$, and an unstable node for $b>0$.
 	\end{proof}
 	
 	By Lemma \ref{lem18},  when $c=0$ and $a>0$, the qualitative properties of equilibria $I_{G^\pm}$ and $I_{R^\pm}$ at infinity in the Poincar\'e disc of system \eqref{71}, which correspond the equilibria $G$, $R$ of system \eqref{Y},  are as shown in Figure \ref{tuGR}.
 	\begin{figure}[htp]
 		\centering
 		\subfigure[{ $b>0$ }]{
 			\includegraphics[width=0.3\textwidth]{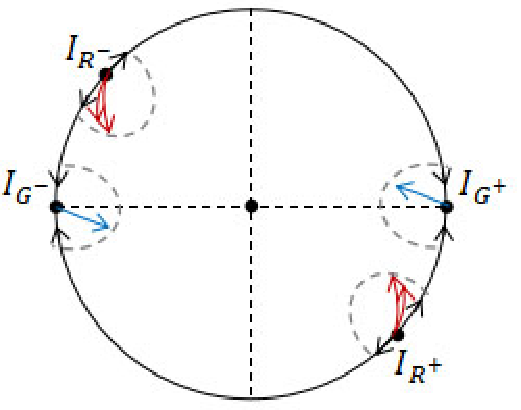}}
 		\quad
 		\subfigure[{$b<0$}]{
 			\includegraphics[width=0.3\textwidth]{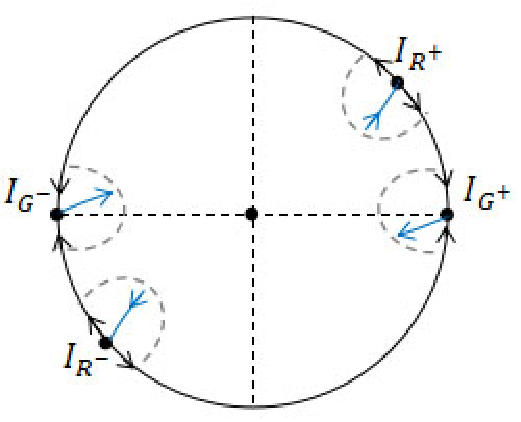}}
 		\caption{{Locally qualitative property of} equilibria at infinity  $I_{G^\pm}$ and $I_{R^\pm}$ in the Poincar\'e disc when $c=0$ and $a>0$. }\label{tuGR}
 	\end{figure}
 	\begin{lemma}
 		\label{lem19}
 		Consider $c>0$ and $a=0$.  $S$  is  a saddle for $b>0$, and a stable node for $b<0$. $G$ is a degenerate equilibrium and its  qualitative properties  is also shown  in  {\rm Tables \ref{G3}--\ref{G2}}.
 	\end{lemma}
 	\begin{proof}
 		Applying a transformation $(u,z)\to (u-b/c,0)$, system \eqref{Y} is changed into
 		\begin{equation}\label{Y1}
 		\left\{\begin{aligned}
 		\frac{du}{d\tau}&=-z^2\left(u^2+\left(\mu-\frac{2b}{c}\right)u+\frac{b^2}{c^2}-\frac{b\mu}{c}+1\right)-cu^3+2bu^2-\frac{b^2}{c}u=:P_9(u,z),\\
 		\frac{dz}{d\tau}&=\frac{b}{c}z^3-z^3u=:Q_9(u,z).
 		\end{aligned}
 		\right.
 		\end{equation}
 		By the implicit function theorem, $P_{10}(u,z)=0$
 		has a unique root $u=\phi_{10}(z)=(b^2-\mu bc+c^2)/(b^2c)z^2+o(z^2)$ for small $|z|$.
 		Thus,
 		$$
 		Q_{10}(\phi_{10}(z),z)=\frac{b}{c}z^3+o(z^3).
 		$$
 		By Theorem \ref{thm7.1} in Appendix B,  the orgin of system \eqref{Y1} is  a saddle for $b>0$, and a stable node for $b<0$. So is $S$.
 		
 		Concerning the equilibrium $G$, it is easy to check that $G$ is a degenerate equilibrium.	To obtain the  qualitative properties of $G$, we need to  similarly consider a  polar transformation $(u,z)=(r\cos\theta,r\sin\theta)$. Then, system \eqref{Y} is changed into equation \eqref{G10}, where $G_{10}(\theta)=\sin\theta(b\cos^2\theta+\sin^2\theta)$ and $H_{10}(\theta)=-\cos\theta(b\cos^2\theta+\sin^2\theta)$. Thus, as studied in Lemma
 		\ref{lem17}, we can get the results. The proof is finished.
 	\end{proof}
 	
 	By Lemma \ref{lem19},  when $a=0$ and $c>0$, the qualitative properties of equilibria $I_{G^\pm}$ and $I_{S^\pm}$ at infinity in the Poincar\'e disc of system \eqref{71}, which correspond the equilibria $G$, $S$ of system \eqref{Y},  are as shown in Figure \ref{tuGS}.
 	\begin{figure}[htp]
 		\centering
 		\subfigure[{ $b>0$ }]{
 			\includegraphics[width=0.3\textwidth]{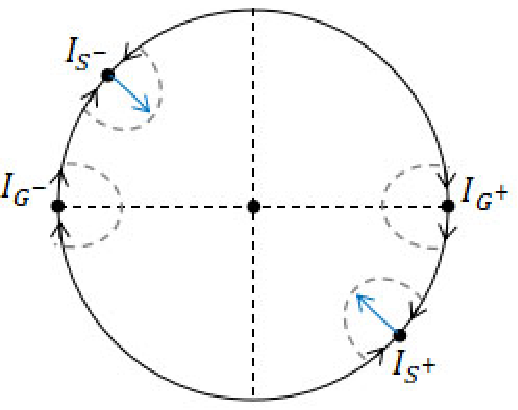}}
 		\quad
 		\subfigure[{$b<0$}]{
 			\includegraphics[width=0.3\textwidth]{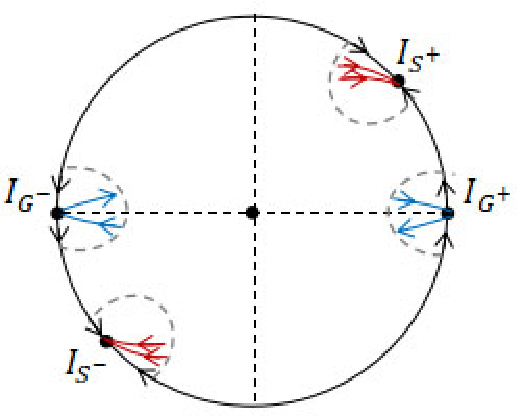}}
 		\caption{{ Locally qualitative property of} equilibria at infinity  $I_{G^\pm}$ and $I_{S^\pm}$ in the Poincar\'e disc when $a=0$ and $c>0$ {(some equilibria at infinity are not marked here)}.}\label{tuGS}
 	\end{figure}
 	\begin{lemma}
 		\label{lem20}
 		Consider $c>0$ and $a>0$. $G$ is a saddle.  $P_1$ is a saddle, and $P_2$ is an $($resp. a$)$ unstable $($resp. stable$)$ node for $b<0$ $($resp. $b>0$$)$ when $b^2-4ac>0$.
 		When $b^2-4ac=0$, $T$ is a degenerate equilibrium. Moreover, the qualitative properties of $T$ is as shown in {\rm Tables
 			\ref{T1}--\ref{T3}}, where $\omega:= \sqrt{c\sqrt{ac}/(a-\mu\sqrt{ac}+c)}$.
 		\begin{table}[htp]
 			\renewcommand\arraystretch{2}
 			\setlength{\tabcolsep}{2.5mm}{
 				\caption{\label{T1} Numbers of orbits connecting $T$ for  $b=-2\sqrt{ac}$ (resp. $b=2\sqrt{ac}$ and $a-\mu\sqrt{ac}+c<0$)  when $a>0$ and $c>0$.}
 				\begin{tabular}{c|c}
 					\hline
 					Exceptional directions & Numbers of orbits
 					\\
 					\hline
 					$\theta=0$ & one \ $(+)$\ $($resp. $(-))$
 					\\
 					\hline
 					$\theta=\pi$ & one \ $(-)$\  $($resp. $(+))$
 					\\
 					\hline
 			\end{tabular}}
 		\end{table}
 		\begin{table}[htp]
 			\renewcommand\arraystretch{2}
 			\setlength{\tabcolsep}{2.5mm}{
 				\caption{\label{T2} Numbers of orbits connecting $T$ for  $b=2\sqrt{ac}$ and $a-\mu\sqrt{ac}+c=0$ when $a>0$ and $c>0$.}
 				\begin{tabular}{c|c}
 					\hline
 					Exceptional directions & Numbers of orbits
 					\\
 					\hline
 					$\theta=0$ & one \ $(+)$\ $($resp. $(-))$
 					\\
 					\hline
 					$\theta=\frac{\pi}{2}$ & $\infty$\ $(-)\ ($resp. $(+))$
 					\\
 					\hline
 					$\theta=\pi$ & one \ $(-)$\ $($resp. $(+))$
 					\\
 					\hline
 					$\theta=\frac{3\pi}{2}$ &  $\infty$ \ $(-)\ ($resp. $(+))$
 					\\
 					\hline
 			\end{tabular}}
 		\end{table}
 		\begin{table}[htp]
 			\renewcommand\arraystretch{2}
 			\setlength{\tabcolsep}{2.5mm}{
 				\caption{\label{T3} Numbers of orbits connecting $T$ for $b=2\sqrt{ac}$ and $a-\mu\sqrt{ac}+c>0$ when $a>0$ and $c>0$.}
 				\begin{tabular}{c|c}
 					\hline
 					Exceptional directions & Numbers of orbits
 					\\
 					\hline
 					$\theta=0$ & one \ $(-)$
 					\\
 					\hline
 					$\theta=\arctan\omega$ & $\infty$\ $(-)$
 					\\
 					\hline
 					$\theta=\pi-\arctan\omega$ &  one \ $(-)$
 					\\
 					\hline
 					$\theta=\pi$ & one \ $(+)$
 					\\
 					\hline	
 					$\theta=\pi+\arctan\omega$ &  one \ $(-)$
 					\\
 					\hline
 					$\theta=2\pi-\arctan\omega$ &  $\infty$\ $(-)$
 					\\
 					\hline
 			\end{tabular}}
 		\end{table}
 	\end{lemma}
 	\begin{figure}[htp]
 		\centering
 		\subfigure[{ 	$b=-2\sqrt{ac}$}]{
 			\includegraphics[width=0.3\textwidth]{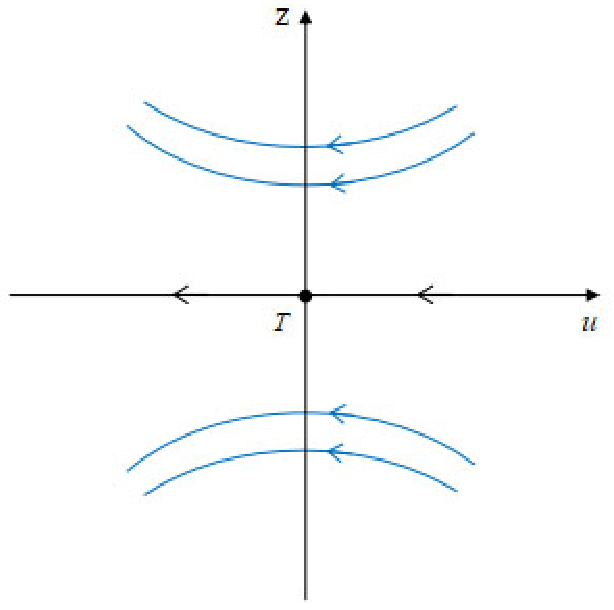}}
 		\quad
 		\subfigure[{ 	$b=2\sqrt{ac}$,  $a-\mu\sqrt{ac}+c<0$}]{
 			\includegraphics[width=0.3\textwidth]{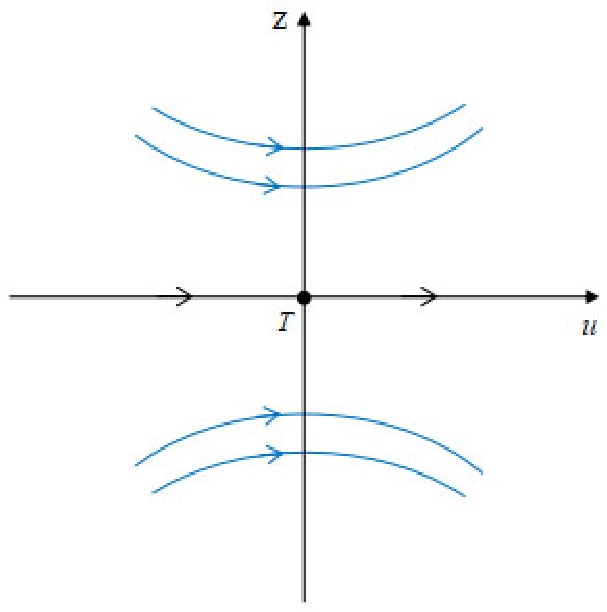}}
 		\quad
 		
 		\subfigure[{  $b=2\sqrt{ac}$, $a-\mu\sqrt{ac}+c=0$}]{
 			\includegraphics[width=0.3\textwidth]{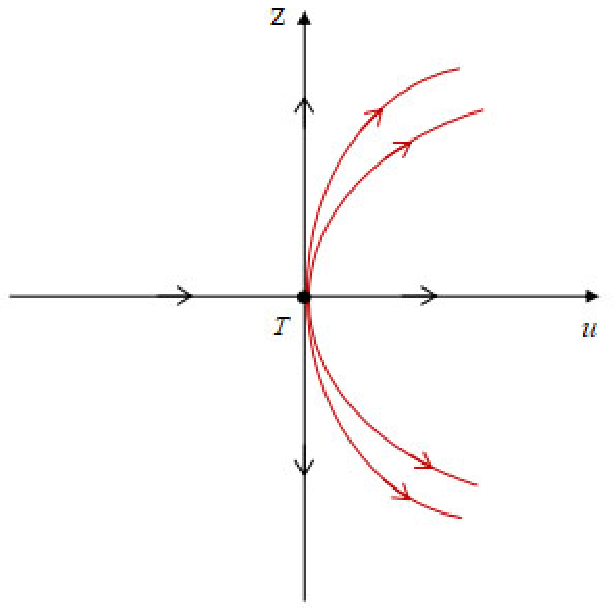}}
 		\quad	
 		\subfigure[{ $b=2\sqrt{ac}$, $a-\mu\sqrt{ac}+c>0$}]{
 			\includegraphics[width=0.3\textwidth]{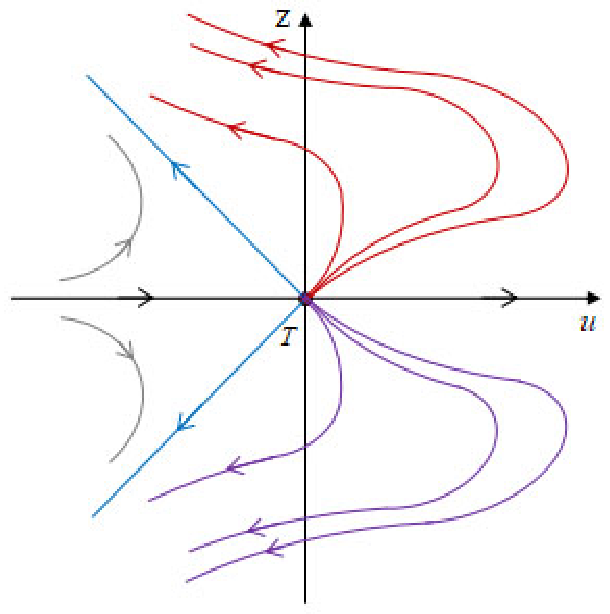}}
 		\caption{ The qualitative properties of $T$ of system \eqref{Y} for $a>0$, $c>0$ and $b^2-4ac=0$.}
 		\label{tu22}
 	\end{figure}
 	
 	\begin{proof}
 		As proved in Lemma \ref{lem18}, $G$ is a saddle of system \eqref{Y}.
 		
 		Consider $b^2-4ac>0$ (i.e. 
 		$b<-2\sqrt{ac}$ or $b>2\sqrt{ac}$	
 		).
 		Let $u_8:=\left(-b+\sqrt{b^2-4ac}\right)/2c$ and $u_9:=\left(-b-\sqrt{b^2-4ac}\right)/2c$.
 		Notice that $u_9<u_8<0$ for $b>0$ and $0<u_9<u_8$ for $b<0$.
 		The Jacobian matrices at $P_1$, $P_2$ are
 		$$
 		J_{P_1}:=	
 		\begin{pmatrix}
 		\frac{\sqrt{b^2-4ac}\left(b-\sqrt{b^2-4ac}\right)}{2c}&0\\
 		0&0\\
 		\end{pmatrix},
 		\quad
 		J_{P_2}:=	
 		\begin{pmatrix}
 		\frac{\sqrt{b^2-4ac}\left(-b-\sqrt{b^2-4ac}\right)}{2c}&0\\
 		0&0\\
 		\end{pmatrix}
 		$$
 		respectively.
 		With a transformation $(u,z)\to(u+u_8,z)$, system \eqref{Y} is changed into
 		\begin{equation}\label{Y3}
 		\left\{\begin{aligned}
 		\frac{du}{d\tau}=&-z^2\left(u^2+(\mu+2u_8)u+u_8^2+\mu u_8+1\right)-cu^3
 		\\
 		&-(3cu_8^2+b)u^2-(3cu_8^2+2bu_8+a)u=:P_{11}(u,z),\\
 		\frac{dz}{d\tau}=&u_8z^3-z^3u=:Q_{11}(u,z),
 		\end{aligned}
 		\right.
 		\end{equation}
 		where $-(3cu_8^2+2bu_8+a)=\left(\sqrt{b^2-4ac}\left(b-\sqrt{b^2-4ac}\right)\right)/2c$.
 		By the implicit function theorem, $P_{11}(u,z)=0$
 		has a unique root $u=\phi_{11}(z)=-(u_8^2+\mu u_8+1)/(3cu_8^2+2bu_8+a)z^2+o(z^2)$ for small $|z|$.
 		Thus,
 		$$
 		Q_{11}(\phi_{11}(z),z)=u_8z^3+o(z^3).
 		$$
 		By Theorem \ref{thm7.1} in Appendix B,  the origin of system \eqref{Y3} is  a saddle.  So is $P_1$. Similarly, $P_2$ is a stable node for $b>0$ and an unstable node for $b<0$.
 		
 		Consider $b^2-4ac=0$  (i.e. 
 		$b=-2\sqrt{ac}$ or $b=2\sqrt{ac}$	
 		). For $T$, considering a transformation $(u,z)\to(u-b/(2c),z)$, system \eqref{Y}
 		can be rewritten as \begin{equation}\label{Y2}
 		\left\{\begin{aligned}
 		\frac{du}{d\tau}&=-z^2\left(u^2+\left(\mu-\frac{b}{c}\right)u+\frac{b^2}{4c^2}-\frac{b\mu}{2c}+1\right)-cu^3+\frac{b}{2}u^2,\\
 		\frac{dz}{d\tau}&=\frac{b}{2c}z^3-z^3u.
 		\end{aligned}
 		\right.
 		\end{equation}
 		In the polar coordinate $(u,z)=(r\cos\theta,r\sin\theta)$,  system \eqref{Y2} is transformated into
 		\begin{equation}
 		\notag
 		\frac{1}{r}\frac{dr}{d\theta}=\frac{H_{11}(\theta)+\widetilde{H_{11}}(\theta,r)}{G_{11}(\theta)+\widetilde{G_{11}}(\theta,r)},
 		\end{equation}
 		where
 		$$
 		G_{11}(\theta)=\sin\theta\left(-\frac{b}{2}\cos^2\theta+\frac{b^2-2bc\mu+4c^2}{4c^2}\sin^2\theta\right)
 		$$
 		and
 		$$
 		H_{11}(\theta)=-\cos\theta\left(-\frac{b}{2}\cos^2\theta+\frac{b^2-2bc\mu+4c^2}{4c^2}\sin^2\theta\right).
 		$$
 		Obviously, the zeros of $G_{11}(\theta)$ is strongly related to the sign of  $b(b^2-2bc\mu+4c^2)$.
 		
 		Firstly, if $b(b^2-2bc\mu+4c^2)<0$ (i.e.
 		$b=-2\sqrt{ac}$, or $b=2\sqrt{ac}$ and $a-\mu\sqrt{ac}+c<0$	
 		), $G_{11}(\theta)=0$ has two roots $\theta=0$, $\pi$ in $\theta\in[0,2\pi)$.   It is clear that $G_{11}'(0)H_{11}(0)=G_{11}'(\pi)H_{11}(\pi)=-b^2/4<0$. Thus, by  \cite[Theorem 3.7 of Chapter 2]{ZDHD},  $H_{11}(0)=b/2$ and $H_{11}(0)=-b/2$, there is a unique orbit approaching  the orgin in the direction $\theta = \pi$ as $\tau\to-\infty$ (resp. $\tau\to+\infty$),  and a unique orbit approaching the orgin in the direction $\theta = 0$  as $\tau\to+\infty$ (resp. $\tau\to-\infty$) in system \eqref{Y2} for $b=-2\sqrt{ac}$ (resp. $b=2\sqrt{ac}$). $T$ of system \eqref{Y} has the same qualitative properties as the orgin of system \eqref{Y2}, see Figures \ref{tu22} (a) and (b).
 		
 		Secondly, if $b^2-2bc\mu+4c^2=0$  (i.e. $b=2\sqrt{ac}$ and $a-\mu\sqrt{ac}+c=0$),  $G_{11}(\theta)=0$ has four roots  $\theta=0$, $\pi/2$, $\pi$, $3\pi/2$ in $\theta\in[0,2\pi)$. Compute that $G_{11}'(0)H(0)=G_{11}'(\pi)H_{11}(\pi)=-b^2/4<0$
 		and $H_{11}(\pi/2)=H_{11}(3\pi/2)=0$. As studied the case $\gamma=0$ of Lemma
 		\ref{lem10},  we can obtain that	there are infinitely many orbits approaching $T$ in  respectively the directions $\pi/2$ and $3\pi/2$ respectively as $\tau\to-\infty$, a unique orbit approaching  $T$ in the direction $\theta = \pi$ as $\tau\to-\infty$, and a unique orbit approaching $T$ in the direction $\theta = 0$ as $\tau\to+\infty$ in system \eqref{Y}, see Figure \ref{tu22}(c).
 		
 		Thirdly, if $b(b^2-2bc\mu+4c^2)>0$ (i.e. $b=2\sqrt{ac}$ and $a-\mu\sqrt{ac}+c>0$), $G_{11}(\theta)=0$ has six roots  $\theta=0$, $\arctan\omega$, $\pi-\arctan\omega$, $\pi$, $\pi+\arctan\omega$,
 		$2\pi-\arctan\omega$ in $\theta\in[0,2\pi)$, where $\omega:= \sqrt{c\sqrt{ac}/(a-\mu\sqrt{ac}+c)}$. Compute that $G_{11}'(0)H_{11}(0)=G_{11}'(\pi)H_{11}(\pi)=-b^2/4<0$ and the other roots satisfy $H_{11}(\theta)=0$.  As studied the case $\gamma<0$ of Lemma
 		\ref{lem10},  we can obtain that the qualitative properties of $T$ of system \eqref{Y},  as shown in Figure \ref{tu22}(d).
 		The proof is finished.
 	\end{proof}
 	
 	By Lemma \ref{lem20},  when $a>0$ and $c>0$, the qualitative properties of equilibria $I_{G^\pm}$ for $b^2-4ac<0$,  $I_{G^\pm}$ and $I_{T^\pm}$ for $b^2-4ac=0$,  $I_{G^\pm}$, $I_{P_1^\pm}$ and $I_{P_2^\pm}$ for $b^2-4ac>0$ at infinity in the Poincar\'e disc of system \eqref{71}, which correspond the equilibria $G$, $T$, $P_1$, $P_2$ of system \eqref{Y},  are as shown in Figure \ref{tuGTP1P2}.
 	\begin{figure}[htp]
 		\centering
 		\subfigure[{ $-2\sqrt{ac}<b<2\sqrt{ac}$ }]{
 			\includegraphics[width=0.3\textwidth]{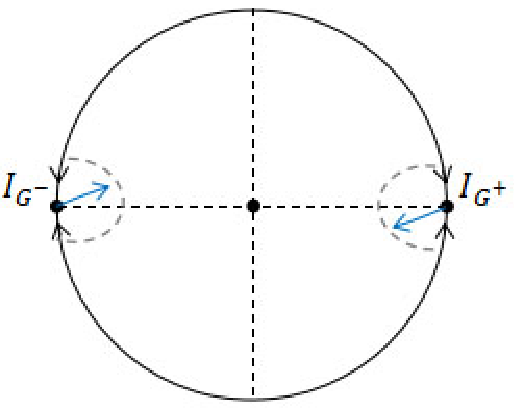}}
 		\quad
 		\subfigure[{  $b>2\sqrt{ac}$}]{
 			\includegraphics[width=0.3\textwidth]{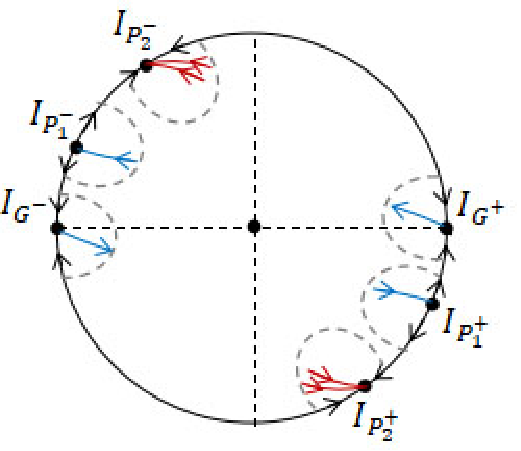}}
 		\quad
 		\subfigure[{  $b<-2\sqrt{ac}$}]{
 			\includegraphics[width=0.3\textwidth]{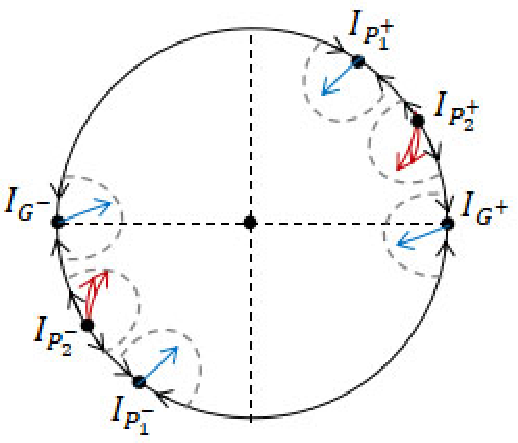}}
 		
 		\quad		
 		\subfigure[{$b=-2\sqrt{ac}$ }]{
 			\includegraphics[width=0.3\textwidth]{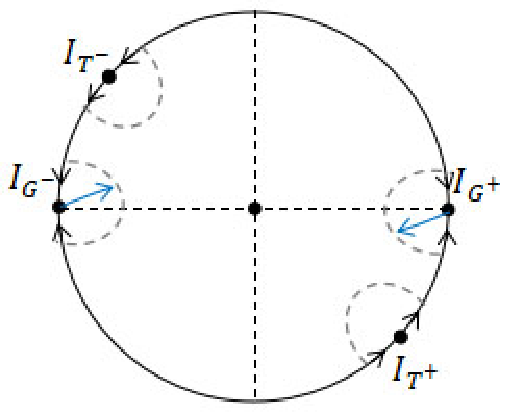}}
 		\quad
 		\subfigure[{ $b=2\sqrt{ac}$,  $a-\mu\sqrt{ac}+c<0$ }]{
 			\includegraphics[width=0.3\textwidth]{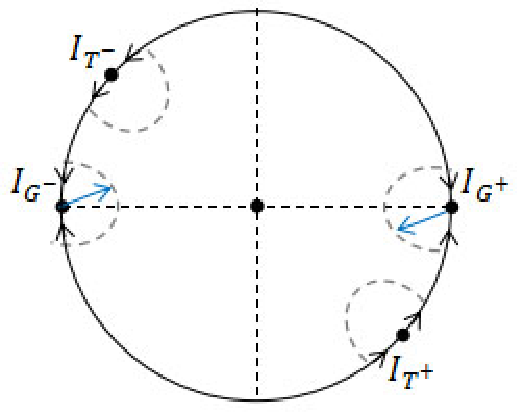}}

 		\quad
 		\subfigure[{$b=2\sqrt{ac}$,  $a-\mu\sqrt{ac}+c=0$}]{
 			\includegraphics[width=0.3\textwidth]{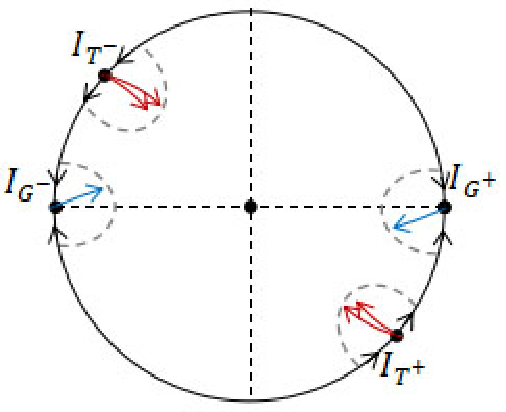}}
 		\quad
 		\subfigure[{ $b=2\sqrt{ac}$,  $a-\mu\sqrt{ac}+c>0$}]{
 			\includegraphics[width=0.3\textwidth]{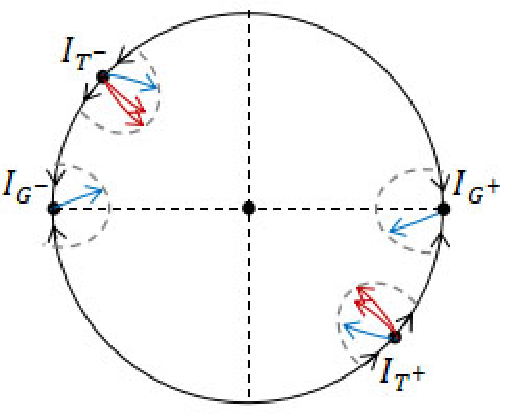}}	
 		\caption{{Locally qualitative property of} equilibria at infinity  $I_{G^\pm}$, $I_{T^\pm}$, $I_{P_1^\pm}$ and $I_{P_2^\pm}$  in the Poincar\'e disc when $a>0$ and $c>0$ {(some equilibria at infinity are not marked here)}.}\label{tuGTP1P2}
 	\end{figure}
 	
 	With  the other Poincar\'e transformation
 	$$
 	x=\frac{v}{z},\quad y=\frac{1}{z},
 	$$  system \eqref{71} is written as
 	\begin{equation}
 	\label{X}
 	\left\{\begin{aligned}
 	\frac{dv}{d\tau}&=z^2(v^2+\mu v+1)+v(av^2+bv+c),\\
 	\frac{dz}{d\tau}&=z^3( v+\mu)+z(av^2+bv+c),
 	\end{aligned}
 	\right.
 	\end{equation}
 	where $d\tau=dt/z^2$.  We only need to study the equilibrium $D=(0,0)$ of system \eqref{X}. Moreover, the qualitative properties of $D$ can be seen in Lemma \ref{lemx}.

	\section*{Acknowledgments}

 The first, second and third authors are supported by the  National Natural Science Foundation of China (Nos. 11801079, 12171485).
The fourth author is partially supported by  the  National Natural Science Foundation of China  (Nos. 11871334, 12071284), by Innovation Program of Shanghai Municipal Education Commission grant number 2021-01-07-00-02-E00087, and by Institute of Modern Analysis--A Frontier Research Center of Shanghai.

\end{document}